\documentclass[11pt]{article}
\usepackage[latin1]{inputenc}
\usepackage{latexsym}
\usepackage{amsmath,amssymb,amsfonts,amsthm}
\usepackage[pdftex]{graphicx}
\usepackage[pdftex,colorlinks=true,linkcolor=blue,citecolor=blue,urlcolor=blue]{hyperref}
\usepackage{geometry}
\usepackage{xcolor}
\usepackage{xspace}
\usepackage{upgreek}
\newcommand{\blackboard}[1]{\mathbf{#1}} 
\newcommand{\R}{\blackboard{R}}%		reels
\newcommand{\N}{\blackboard{N}}%		entiers
\renewcommand{\P}{\blackboard{P}}%	 	probabilite
\newcommand{\E}{\blackboard{E}}% 		esperance
\renewcommand{\L}{\blackboard{L}}%	 	Lebesgue

\newcommand{\calig}[1]{\mathcal{#1}}
\newcommand{\A}{\calig{A}}% 			generateur infinitesimal diff
\DeclareMathOperator{\var}{var}%	 	variance
\renewcommand{\d}[1]{\textup{d} #1}% 	le dx de l'integrale
\newcommand{\delimleft}[2]{\ifcase #1\or%	delimiteur gauche
    \bigl#2\or %
    \Bigl#2\or %
    \biggl#2\or %
    \Biggl#2\or %
    \left#2\fi}
\newcommand{\delimright}[2]{\ifcase #1\or%	delimiteur droite
    \bigr#2\or %
    \Bigr#2\or %
    \biggr#2\or %
    \Biggr#2\or %
    \right#2\fi}
\newcommand{\pa}[2][5]{% 			parentheses
    \delimleft{#1}{(} #2 \delimright{#1}{)}}
\newcommand{\br}[2][5]{%	 		crochets
    \delimleft{#1}{[} #2 \delimright{#1}{]}}
\newcommand{\ac}[2][1]{%			accolades
    \delimleft{#1}{\{} #2 \delimright{#1}{\}}}

\newcommand{\psca}[2][1]{%			produit scalaire
    {\delimleft{#1}{\langle} #2%
    \delimright{#1}{\rangle}}}
\newcommand{\abs}[2][1]{%			valeur absolue
    {\delimleft{#1}{\lvert} #2%
    \delimright{#1}{\rvert}}}
\newcommand{\norm}[2][1]{% 			norme
    {\delimleft{#1}{\lVert} #2%
    \delimright{#1}{\rVert}}}
\newcommand{\normLp}[3][1]{%			norme L^p
    \norm[#1]{#3}_{\scriptscriptstyle #2}}

\newcommand{\ind}[2][1]{%			indicatrice
    \boldsymbol{1}_{\ac[#1]{#2}}}

\newcommand{\esp}[2][5]{%			esperance
    \E \br[#1]{#2}}
\newcommand{\esps}[3][5]{%			esperance sous
    \E_{#2} \br[#1]{#3}}
\newcommand{\espc}[3][5]{%			esperance conditionnelle
    \E \br[#1]{#2%
    \delimleft{#1}{.} \vphantom{#2}\vphantom{#3}%
    \delimright{#1}{\vert} #3}}

\newcommand{\ie}{\textit{i.e. }}%		i.e.
\renewcommand{\le}{\leqslant}% 			<= plus beau
\renewcommand{\ge}{\geqslant}%			>= plus beau
\newcommand{\ds}{\displaystyle}%		displaystyle
\usepackage{subcaption}

\numberwithin{table}{section}
\usepackage{booktabs}

\newcommand{\iid}{\emph{i.i.d.}\xspace}

\DeclareMathOperator*{\argmin}{argmin}

\newtheorem{Theorem}{Theorem}[section]
\newtheorem{Definition}[Theorem]{Definition}
\newtheorem{Proposition}[Theorem]{Proposition}

\newtheorem{Lemma}[Theorem]{Lemma}

\newtheorem{Remark}[Theorem]{Remark}

\DeclareMathOperator{\Cost}{Cost}
\DeclareMathOperator{\cost}{\upkappa}
\DeclareMathOperator{\varf}{\upnu}
\DeclareMathOperator{\perf}{\upphi}
\DeclareMathOperator{\bias}{\upmu}
\DeclareMathOperator{\mT}{\mathbf{T}}
\DeclareMathOperator{\W}{\mathbf{W}}
\DeclareMathOperator{\w}{\mathbf{w}}
\DeclareMathOperator{\Hr}{\calig{H}}

\newcommand{\e}[1]{\ensuremath{\cdot {\scriptstyle 10}^{\scriptscriptstyle #1}}}
\newcommand{\MLMC}{MLMC\xspace}
\newcommand{\MLRR}{ML2R\xspace}
\newcommand{\espY}[2][5]{\esps[#1]{Y\!\!}{#2}}

\author{Vincent Lemaire\footnote{Laboratoire de Probabilit\'es et Mod\`eles Al\'eatoires, UMR 7599, UPMC Paris 6 (Sorbonne Universit\'e), 
E-mail: \url{vincent.lemaire@upmc.fr}}, Gilles Pag\`es\footnote{Laboratoire de Probabilit\'es et Mod\`eles Al\'eatoires, UMR 7599, UPMC Paris 6 (Sorbonne Universit\'e), E-mail: \url{gilles.pages@upmc.fr}}
}

\begin{document}
\title{Multilevel Richardson-Romberg Extrapolation}
\maketitle

\begin{abstract}
    We propose and analyze a Multilevel Richardson-Romberg (\MLRR) estimator which combines the higher order bias cancellation of the Multistep Richardson-Romberg method introduced in~\cite{PAG0} and the variance control resulting from Multilevel Monte Carlo (\MLMC) paradigm (see~\cite{GILES, Hein}). Thus, in standard frameworks like discretization schemes of diffusion processes, the root mean squared error (RMSE) $\varepsilon > 0$ can be achieved with our \MLRR estimator with a global complexity of $\varepsilon^{-2} \log(1/\varepsilon)$ instead of $\varepsilon^{-2} (\log(1/\varepsilon))^2$ with the standard \MLMC method, at least when the weak error $\esp{Y_h}-\esp{Y_0}$ of the biased implemented estimator $Y_h$  can be expanded at any order in $h$ and $\normLp{2}{Y_h - Y_0} = O(h^{\frac{1}{2}})$. The \MLRR estimator is then halfway between a regular \MLMC and a virtual unbiased Monte Carlo. When the strong error $\normLp{2}{Y_h - Y_0} = O(h^{\frac{\beta}{2}})$, $\beta < 1$, the gain of \MLRR over \MLMC becomes even more striking. We carry out numerical simulations to compare these estimators in two settings: vanilla and path-dependent option pricing by Monte Carlo simulation and the less classical Nested Monte Carlo simulation.
\end{abstract}



\section{Introduction}
The aim of this paper is to combine the multilevel Monte Carlo estimator introduced by S. Heinrich in~\cite{Hein} and developed by M. Giles in~\cite{GILES}  (see also~\cite{KEBAIER} for the statistical Romberg approach) and the (consistent) Multistep Richardson-Romberg extrapolation (see~\cite{PAG0}) in order to minimize the simulation cost of a  quantity of interest $I_0 = \esp{Y_0}$ when the random variable $Y_0$ cannot be simulated at a reasonable cost (typically a functional of a generic multidimensional diffusion process or a conditional expectation). Both methods rely on the existence of a family of random variables $Y_h$, $h>0$, which strongly approximate $Y_0$ as $h$ goes to 0 whose bias $\esp{Y_h} - \esp{Y_0}$ can be expanded as a polynomial function of $h$ (or $h^{\alpha}$, $\alpha>0$). 

However, the two methods suffer from opposite but significant drawbacks: the multilevel Monte Carlo estimator does not fully take advantage of the existence of such an expansion beyond the first order whereas the Multistep Richardson-Romberg extrapolation induces an increase of the variance of the resulting estimator. Let us be more precise. 

Consider a probability space $(\Omega, \A, \P)$ and suppose that we have a family $(Y_h)_{h\ge 0}$ of real valued random variables in $\L^2(\P)$, associated to $Y_0$ supposed to be non degenerate, and satisfying $\ds \lim_{h \to 0} \normLp{2}{Y_h - Y_0} = 0$ where $h$ takes values in $\Hr = \ac{\mathbf h /n, \, n \ge 1}$ for a fixed $\mathbf{h} \in (0,+\infty)$.
%an admissible subset of parameters $\Hr \subset (0, {\mathbf h}]$ such that $0 \in \Hr$, $\mathbf h \in \Hr$ and $\frac{\Hr}{n}\subset \Hr$ for every integer $n\ge 1$. We also assume that ${\mathbf h}\!\in \Hr $. 
Usually, the random variable $Y_h$ appears as a functional of a time discretization scheme of step $h$ or from an inner approximation in a nested Monte Carlo simulation. The parameter $h$ is called the {\em bias parameter} in what follows. Furthermore, we assume that, for every admissible $h \!\in \Hr$, the random variable $Y_h$ can be simulated at a reasonable computational cost whereas $Y_0$ can not. %{\em is not} (at a reasonable cost). %For this reason, the specification of $h$ will be often made in connection with the complexity of the simulation $Y_h$ with in mind to make it inverse linear in $h$.

We aim at computing an as good as possible approximation of $I_0 = \esp{Y_0}$ by a Monte Carlo type simulation. The starting point is of course to fix a parameter $h \in \Hr$ to consider a standard Monte Carlo estimator based on $Y_h$ to compute $I_0$. So, let $(Y^{(k)}_h)_{k \ge 1}$ be a sequence of independent copies of $Y_h$ and the estimator $I^{(h)}_N = \frac{1}{N} \sum_{k=1}^N Y^{(k)}_h$. By the strong law of numbers and the central limit theorem we have a control of the renormalized \emph{statistical error} $\sqrt{N} (I^{(h)}_N - \esp{Y_h})$ which behaves as a centered Gaussian with variance $\var(Y_h)$. On the other hand, there is a \emph{bias error} due to the approximation of $I_0$ by $I_h = \esp{Y_h}$. This bias error is also known as the \emph{weak error} when $Y_{h}$ is a functional of the time discretization scheme of a stochastic differential equation with step $h$. In many applications, the bias error can be expanded as 
\begin{equation}\label{eq:R0}
	\esp{Y_h} - \esp{Y_0} = c_1 h^{\alpha} + \dots + c_{_R} h^{\alpha R} + o(h^{\alpha R})
\end{equation}
where $\alpha$ is a positive real parameter (usually $\alpha = \frac 12, 1$ or $2$). In this paper, we fully take into account this error expansion and provide a very efficient estimator which can be viewed as a coupling between an \MLMC estimator and a Multistep Richardson-Romberg extrapolation. Multilevel methods require the additional strong convergence rate assumption $\normLp{2}{Y_h - Y_0}^2 = \mathcal{O}(h^{\beta})$ involving a parameter $\beta \in (0,1)$. About the need of this assumption, we refer to~\cite{Hutz} for some counterexamples. However, the recent paper~\cite{Belo2} shows how      this strong error assumption can be relaxed in some situations.

We first present a brief description of the original \MLMC estimator as described in~\cite{GILES}. The main idea is to use the following telescopic summation with depth $L \ge 2$,  
\begin{equation*}
	\esp{Y_{h_L}} = \esp{Y_h} + \sum_{j = 2}^L \esp{Y_{h_j} - Y_{h_{j-1}}}
\end{equation*}
where $(h_j)_{j=1,\dots,L}$ is a geometrically decreasing sequence of different bias parameters $h_j = M^{-(j-1)} h$, $M\!\ge2$. For each level $j \in \ac{1,\dots,L}$, the computation of $\esp{Y_{h_j}-Y_{h_{j-1}}}$ is performed by a standard Monte Carlo procedure. The key point is that, at each level $j$, we consider a random sample of $(Y_{h_{j-1}}, Y_{h_j})$ of size  $N_j = \lceil N q_j \rceil$, where $q=(q_1,\dots,q_{_L}) \in \calig{S}_+(L) = \ac{q \!\in (0,1)^L, \sum_{j=1}^L q_j = 1}$ ($L$-dimensional simplex),  with  in mind that  the marginals $Y_{h_{j-1}}$ and $Y_{h_j}$ are highly correlated since $\lim_{h \to 0}\normLp{2}{Y_h-Y_0} = 0$ (see Section~\ref{sec:correlation} for details).
% and $Y_{h_{j-1}}$ supposed to be correlated. 
More precisely, we consider $L$ copies of the biased family denoted $Y^{(j)} = (Y^{(j)}_h)_{h \in \Hr}$, $j \in \ac{1,\dots,L}$ attached to \emph{independent} random copies $Y^{(j)}_0$ of $Y_0$. The \MLMC estimator then writes 
\begin{equation}\label{MLR-estimator}
	I^N_{h,L,q} = \frac{1}{N_1} \sum_{k=1}^{N_1} Y_{h}^{(1),k} + \sum_{j=2}^L \frac{1}{N_j} \sum_{k=1}^{N_j} \pa{Y_{h_j}^{(j),k} - Y_{h_{j-1}}^{(j),k}}
\end{equation}
where $(Y^{(j),k})_{k\ge 1}$, $j=1, \dots, L$ are independent sequences of independent copies of $Y^{(j)}$ and $N_1,\dots,N_L$ are positive integers. The analysis of the computational complexity and the study of the bias--variance structure of this estimator will appear as a particular case of a generalized multilevel framework that we will introduce and analyze in Section~\ref{sec:MLRR}. This framework, following the original \MLMC, highly relies on the combination of a strong rate of approximation of $Y_0$ by $Y_h$ and a first order control of weak error $\esp{Y_h}-\esp{Y_0}$. This \MLMC estimator has been extensively applied to various fields of numerical probability (jump diffusions~\cite{DER1, DER2}, American options~\cite{Belo}, computational statistics and more general numerical analysis problems (high dimensional parabolic SPDEs, see~\cite{BLS}, etc). For more references, we refer to the web page \url{http://people.maths.ox.ac.uk/gilesm/mlmc\_community.html} and the references therein.

On the other hand, the Multistep Richardson-Romberg extrapolation takes advantage of the full expansion~\eqref{eq:R0}. Let us first recall the one-step Richardson-Romberg Monte Carlo estimator. We still consider  a  biased family denoted $Y = (Y_h)_{h \in \Hr}$ attached to the random variable $Y_0$. The one-step Richardson-Romberg Monte Carlo estimator then writes  
\begin{equation*}
	I^N_{h, \frac{h}{2}} = \frac{1}{N} \sum_{k=1}^N \pa{2 Y_{\frac{h}{2}}^{k} - Y_h^{k}}
\end{equation*}
where $(Y^{k})_{k\ge 1}$ is a sequence of independent copies of $Y$. It is clear that this linear combination of Monte Carlo estimators satisfies the following bias error expansion (of order 2 in $h$) 
\begin{equation*}
	\esp{2 Y_{\frac{h}{2}} - Y_h} - \esp{Y_0} = -\frac{c_2}{2} h^2 + o(h^2).
\end{equation*}
Moreover, the asymptotic variance of this estimator satisfies $ \var(I^N_{h,\frac{h}{2}})=\var(2Y_{\frac{h}{2}}-Y_h)/N\approx\var(Y_0)/N$ which is the same as the crude Monte Carlo estimator. Then, it is natural  to design  a linear estimator with bias error   in $h^3$ by linearly combining  $Y_h$, $Y_{\frac{h}{2}}$ and $Y_{\frac{h}{3}}$  and so on. Such an  extension called Multistep Richardson-Romberg extrapolation for Monte Carlo estimator has been introduced and extensively investigated in~\cite{PAG0} in the framework of discretization of diffusion processes. More details are given in Section~\ref{sec:MSRR}. 

The aim of this paper is to show that an appropriate combination of the \MLMC estimator and the Multistep Richardson-Romberg extrapolation outperforms the standard \MLMC. More precisely, we will see in Section \ref{sec:MLRR} that an implementation of the Multilevel Richardson Romberg estimator (\MLRR) turns out to be a {\em weighted version of the \MLMC} and writes  
\begin{equation}\label{intro:MLRR-estimator}
	I^N_{h,R,q} = \frac{1}{N_1} \sum_{k=1}^{N_1} Y_{h}^{(1),k} + \sum_{j=2}^R \frac{\W_j}{N_j} \sum_{k=1}^{N_j} \pa{Y_{h_j}^{(j),k} - Y_{h_{j-1}}^{(j),k}}
\end{equation}
where $R \ge 2$ is the {\em depth level} -- similar to $L$ in~\eqref{MLR-estimator} -- and $(Y^{(j),k})_{k\ge 1}$ are like in~\eqref{MLR-estimator}. We denote by $n_j = M^{j-1}$ the $j$--th \emph{refiner coefficient} of the initial bias parameter $h \in \Hr$ where  the integer $M \ge 2$ is called  the \emph{root} of the refiners. A strong feature of our approach comes from the fact that the weights $(\W_k)_{k = 2, \dots, R}$ are explicit and only depend on $\alpha$ (given by~\eqref{eq:R0}), $M$ and $R$. In practice these \MLRR weights read $\W_j = \sum_{i = j}^R \w_i$ where $\w_i$ are given further on by~\eqref{eq:weights}. These derivative weights $(\w_i)_{i \in \{1,\dots,R\}}$ have been introduced in~\cite{PAG0} to kill the successive bias terms that appear in the expansion~\eqref{eq:R0}.

To compare the two methods \MLMC and \MLRR, we consider the following optimization problem: minimizing the global simulation cost (of one estimator) subject to the constraint that the resulting $\L^2$--error or root mean squared error (RMSE) must be lower than a prescribed $\varepsilon > 0$. We solve the problem step by step for both estimators (in fact for a more general unifying class of estimators). In the first stage, we minimize the \emph{effort} of the estimator  (product of its variance by its complexity) to optimally dispatch the $N_j$ across all levels (it can be viewed as a stratification procedure). Doing so, we are able to specify the initial bias parameter $h$ and the depth level $R$ as functions of $\varepsilon$ and the structural  parameters ($\alpha$, $\beta$, $V_1$, $\var(Y_0)$). A light preprocessing makes possible to optimize the choice of the root $M$ of the refiners. Basically (see Theorem~\ref{prop:asympReps}), the numerical cost of the \MLRR estimator implemented with these optimal parameters (depending on $\varepsilon$) and denoted $\Cost(\text{\MLRR})$ satisfies 
\begin{equation*}
	%\limsup_{\varepsilon \to 0} v(\beta, \varepsilon) \times 
	\Cost(\text{\MLRR}) \lesssim K(\alpha, \beta, M)  v(\beta, \varepsilon)
\end{equation*}
where $f \lesssim g$ means that $\limsup_{\varepsilon \to 0} f(\varepsilon) / g(\varepsilon) \le 1$, $K(\alpha, \beta, M)$ is an explicit bound (see~\eqref{eq:K_MLRR} in Theorem~\ref{prop:asympReps}) and
$$v(\beta, \varepsilon) = \ds \begin{cases}
	\varepsilon^{-2} & \text{if $\beta > 1$}, \\
	\varepsilon^{-2} \log(1/\varepsilon) & \text{if $\beta = 1$}, \\
	\varepsilon^{-2} e^{\frac{1-\beta}{\sqrt{\alpha}} \sqrt{2 \log(1/\varepsilon) \log(M)}} & \text{if $\beta < 1$}.
\end{cases}$$
Note that $e^{\frac{1-\beta}{\sqrt{\alpha}} \sqrt{2 \log(1/\varepsilon) \log(M)}} = o(\varepsilon^{-\eta})$ for all $\eta > 0$.
As first established in~\cite{GILES}, we prove likewise that the optimal numerical cost of the \MLMC estimator denoted $\Cost(\text{\MLMC})$ satisfies a similar result with $v(\beta,\varepsilon) = \varepsilon^{-2}$ if $\beta >1 $, $v(\beta, \varepsilon) = \varepsilon^{-2} \log(1/\varepsilon)^{2}$ if $\beta = 1$ and $v(\beta, \varepsilon) = \varepsilon^{-2 - \frac{1-\beta}{\alpha}}$ if $\beta < 1$. In the case $\beta = 1$, the gain of $\log(1/\varepsilon)$ may look as a minor improvement but, beyond the fact that it is halfway to a virtual unbiased simulation, this improvement is  obtained with respect to an already   tremendously efficient method to speed up crude Monte Carlo simulation. In fact,  as emphasized in our numerical experiments (see Section~\ref{sec:NumEx}), this may lead to a significant reduction factor for  CPU time, $e.g.$  when $\alpha<1$: pricing a Black-Scholes Lookback Call option  with  a prescribed quadratic error $\varepsilon = 2^{-8}$, yields  a reduction factor of $3.5$ in favor of \MLRR. When $\beta < 1$, the  above  theoretical reduction factor asymptotically  goes to $+\infty$ as $\varepsilon$ goes to 0 in a very steep way.   Thus, the reduction in CPU time factor reaches, {\em mutatis mutandis},  $22$ for a Black-Scholes Up\&Out Barrier  call option for which $\beta = \frac 12$ (still using a regular  Euler scheme {\em without Brownian bridge}). When compared on the basis of the resulting empirical RMSE, these factors become even larger (approximately $48$ and $61$ respectively). In fact, it confirms  that $\beta<1$ is the setting where  our \MLRR  estimator is the most powerful compared  to regular \MLMC. Additional simulations are available on the web page \url{https://simulations.lpma-paris.fr/multilevel/}.

The paper is organized as follows: in Section~\ref{sec:prelim}, we propose a general parametrized  framework to formalize the optimization of  a {\em biased} Monte Carlo simulation based on the $\L^2$--error minimization.
The crude Monte Carlo estimator and the Multistep Richardson-Romberg estimator appear as the first two examples, allowing us to make precise few notations as well as our main assumptions. In Section~\ref{sec:MLRR}, we first introduce the extended family of multilevel estimators attached to design matrices~$\mT$. Among them, we  specify  in more details our proposal: the new \MLRR estimator, but also the standard \MLMC estimator. 
Two typical fields of application are presented in Section~\ref{sec:appli}: the  time discretization of stochastic processes (Euler scheme)  and the nested Monte Carlo method, for which a weak expansion of the error at any order is established in the regular case.  In Section~\ref{sec:NumEx}, we present and comment on numerical experiments carried out in the above two fields.

\medskip
\noindent {\sc Notations:} 
$\bullet$  Let $\N^* = \{1,2, \ldots\}$ denote the set of positive integers and $\N=\N^*\cup\{0\}$.
	
\noindent $\bullet$  If $\underline n=(n_1,\ldots,n_{_R})\!\in (\N^*)^R$, $|\underline n|= n_1+\cdots+n_{_R}$ and $\displaystyle \underline n!=\prod_{1\le i\le R} n_i$.

\noindent $\bullet$  Let $(e_1,\ldots,e_{_R})$ denote the canonical basis of $\R^R$ (viewed as a vector space of column vectors). Thus $e_i = (\delta_{ij})_{1\le j\le R}$ where $\delta_{ij}$ stands for the classical Kronecker symbol.
	
\noindent $\bullet$  $\psca{.,.}$ denotes the canonical inner product on $\R^R$.
	
\noindent $\bullet$  For every $x\!\in \R_+ = [0, +\infty)$, $\lceil x \rceil$ denotes the unique $n \in \N^*$ satisfying $n-1 < x \le n$.% and $\lfloor x \rfloor$ denotes the unique $n \in \N$ satisfying $n \le x <n+1$.

%\noindent $\bullet$  For every $x\!\in \R_+ = [0, +\infty)$, and $y\!\in \R_+^* = (0, +\infty)$, $\lceil x \rceil_y$ denotes the unique $n \in \N^*$ satisfying $(n-1) y < x 
%\le n y$ and $\lfloor x \rfloor_y$ denotes the unique $n\! \in \N$ satisfying $n y \le x < (n+1) y$.
%		
\noindent $\bullet$  If $(a_n)_{n \in \N}$ and $(b_n)_{n \in \N}$ are two sequences of real numbers, $a_n \sim b_n$ if $a_n = \varepsilon_n b_n$ with $\lim_n \varepsilon_n = 1$, $a_n = O(b_n)$ if $(\varepsilon_n)_{n \in \N}$ is bounded and $a_n = o(b_n)$ if $\lim_n \varepsilon_n = 0$.

\noindent $\bullet$ $\var\pa{X}$ and $\sigma(X)$ denote the variance and the standard deviation of a random variable $X$ respectively.

\section{Preliminaries}\label{sec:prelim}
\subsection{Mixing variance and complexity (effort)}\label{sec:perf}
We first introduce some notations and recall basic facts on (possibly biased) {\em linear} estimators. We consider a family of linear statistical estimator $(I^N_{\uppi})_{N \ge 1}$ of $I_0 \in \R$ where $\uppi$ lies in a parameter set $\Uppi \subset \R^d$. By linear, we mean, on the one hand, that, for every integer $N\ge 1$, 
\[
\esp{I^N_{\uppi}}= \esp{I^1_{\uppi}} 
\]
and, on the other hand, that the numerical cost $\Cost(I^N_{\uppi})$  induced by  the simulation of $I^N_{\uppi}$ is given by 
\[
	\Cost(I^N_{\uppi}) = N \cost(\uppi)
\]
where $\cost(\uppi)= \Cost(I^1_{\uppi}) $ is the cost of a single simulation or {\em unitary complexity}. 
 
 We also assume that our estimator is of {\em Monte Carlo type} in the sense that its variance is {\em inverse linear} in the size $N$ of the simulation:
 \[
	 \var(I^N_{\uppi})= \frac{\varf(\uppi)}{N}
\]
where $\varf(\uppi) = \var(I^1_{\uppi})$ denotes the variance of one simulation. For example, in a crude  biased Monte Carlo $\uppi = h \in \Hr$, in a Multilevel Monte Carlo $\uppi = (h, R, q)\! \in \Hr \times \N^* \times \calig{S}_+$ and in the Multistep Monte Carlo $\uppi = (h, R) \in \Hr \times \N^*$. 
 
We are looking for the ``best'' estimator in this family $\ac{(I^N_{\uppi})_{N \ge 1}, \uppi \in \Uppi}$ \emph{i.e.} the estimator minimizing the computational cost for a given error $\varepsilon > 0$. In the sequel, we will often consider $N$ as a continuous variable lying in $\R_+$. A natural choice for measuring the random error $I^N_{\uppi}-I_0$ is to consider the $\L^2$--error or root mean squared error (RMSE) $\sqrt{\esp[1]{(I^N_{\uppi} - I_0)^2}} = \normLp{2}{I^N_{\uppi} - I_0}$. Our aim is to minimize the cost of the simulation for a given target error, say $\varepsilon>0$. This generic problem reads
\begin{equation} \label{probleme_generique}
	\big(\uppi(\varepsilon), N(\varepsilon)\big) = \ds \argmin_{\|I^N_{\uppi} - I_0\|_2 \le \varepsilon} \Cost(I^N_{\uppi}).
\end{equation}
In order to solve  this minimization problem, we introduce the notion of \emph{effort} $\perf \pa[1]{\uppi}$ of a linear Monte Carlo type estimator $I^N_{\uppi}$.

\begin{Definition} The \emph{effort} of the estimator $I^N_{\uppi}$ is defined for every $\uppi\!\in \Uppi$ by 
\begin{equation} \label{def_perf_index}
	\perf \pa[1]{\uppi} = \varf(\uppi) \cost(\uppi).
\end{equation}
\end{Definition} 
By definition of a linear estimator $I^N_{\uppi}$, we have that 
$$
\perf \pa[1]{\uppi} = \varf(\uppi) \cost(\uppi) = \var(I^N_{\uppi}) \Cost(I^N_{\uppi}) = \var(I^1_{\uppi}) \Cost(I^1_{\uppi})
$$ 
for every integer $N\ge 1$, so that we obtain the fundamental relation 
\begin{equation}
	\Cost(I^N_{\uppi}) = N \frac{\perf \pa[1]{\uppi}} {\upnu(\uppi)}.
\end{equation}
\begin{itemize}
	\item
		If the estimators $(I^N_{\uppi})_{N \ge 1}$ are {\em unbiased} \ie  $\esp[1]{I^N_{\uppi}} = I_0$ for every $\uppi \!\in \Uppi$, then $\esp{(I^N_{\uppi} - I_0)^2} =\|I^N_{\uppi} - I_0\|_2^2= \var \pa[1]{I^N_{\uppi}} = \frac{1}{N} \varf(\uppi)$. The solution of the generic problem~\eqref{probleme_generique} then reads 
\begin{equation} \label{sol_sans_biais}
	\uppi(\varepsilon) = \uppi^*=  \argmin_{\uppi \in \Uppi} \perf \pa[1]{\uppi}, \quad N(\varepsilon) = \frac{\upnu(\uppi^*)}{\varepsilon^2}
	=\frac{\upphi(\uppi^*)}{\cost(\uppi^*)\varepsilon^2}.
\end{equation}
Consequently, the most performing estimator $I^N_{\uppi}$ is characterized as a minimizer of the effort $\perf\pa[1]{\uppi}$ as defined above (and the parameter $\uppi$ does not depend on $\varepsilon$).

	\item
		When the estimators $(I^N_{\uppi})_{N \ge 1}$, $\uppi\!\in \Uppi$, are {\em biased}, the mean squared error writes 
\begin{equation*}
	\esp[1]{(I^N_{\uppi} - I_0)^2} = \bias^2(\uppi) + \frac{\varf(\uppi)}{N}
\end{equation*}
where 
$$
\bias(\uppi) = \esp[1]{I^N_{\uppi}} - I_0 = \esp[1]{I^1_{\uppi}} - I_0
$$ 
denotes the bias  (which does not depend on $N$). Using that $\varf(\uppi) = N \pa[1]{\normLp{2}{I^N_{\uppi}-I_0}^2 - \bias(\uppi)^2}$, the solution of the generic problem~\eqref{probleme_generique} reads 
\begin{equation} \label{sol_avec_biais}
	\uppi(\varepsilon) = \argmin_{\uppi \in \Uppi,\, \,|\!\bias(\uppi)| < \varepsilon} \pa{\frac{\perf \pa[1]{\uppi}}{\varepsilon^2 - \bias^2(\uppi)}}, 
	\quad N(\varepsilon) = \frac{\upnu(\uppi(\varepsilon))}{\varepsilon^2 - \bias^2(\uppi(\varepsilon))}
	= \frac{\upphi(\uppi(\varepsilon))}{\cost(\uppi(\varepsilon))(\varepsilon^2-\bias^2(\uppi(\varepsilon)))}.
\end{equation}
\end{itemize}

\subsection{Assumptions on weak and strong approximation errors} \label{framework}
We come back to the framework described in the introduction: let $(Y_h)_{h \in \Hr}$ be a family of real valued random variables associated to a random variable $Y_0 \in \L^2$. The index set $\Hr$  is a  set of {\em bias} parameters (in fact {\em representative} of a bias) defined by $\Hr = \ac{\mathbf h/n,\,n \ge 1}$ for a fixed $\mathbf h \in (0, +\infty)$.
All random variables $Y_h$ are defined on the same  probability space $(\Omega,\A, \P)$. The family satisfies  two assumptions which formalize the strong and weak rates of approximation of $Y_0$ by $Y_h$ when $h\to 0$ in $\Hr$. These assumptions are the basement of multilevel simulation methods (see~\cite{GILES, Hein}):
\begin{description}
	\item[Bias error expansion (weak error rate):]
\begin{equation} \tag{$WE_{\alpha,\bar R}$} \label{weak_error}
	\hskip -1cm\exists \,\alpha > 0, \bar R \ge 1, \quad \esp{Y_h} = \esp{Y_0} + \sum_{k=1}^{\bar R} c_k h^{\alpha k} + h^{ \alpha\bar R} \eta_{_{\bar R}}(h),\quad \lim_{h\to 0}\eta_{_{\bar R}}(h)=0,
\end{equation}
where $c_k$, $k=1,\dots,\bar R$,  are real coefficients and $\eta_{_{\bar R}}$ is a real valued function defined on $\Hr$.

	\item[Strong approximation error assumption:]
\begin{equation} \tag{$SE_\beta$} \label{strong_error}
	\exists\, \beta > 0, \; V_1 \in \R_+, \quad \normLp{2}{Y_h - Y_0}^2 = \esp{\abs{Y_h - Y_0}^2} \le V_1 h^\beta.
\end{equation}
\end{description}
Note that the parameters $\alpha$, $\beta$ and $\bar R$ are structural parameters which only depend on the family $(Y_h)_{h \in \Hr}$.  When $(Y_h)_{h\in \Hr}$ satisfies~\eqref{weak_error} for every integer $\bar R$, we will say that $(WE_{\alpha,\infty})$ is fulfilled.  Such a family is said to be {\em admissible} (at level $\bar R$ with parameters $\beta$ and $\alpha$). 

Note that consistency of weak error $(WE_{\alpha,\bar R})$ (when $c_1 \neq 0$) and  strong  error $(SE_\beta)$ implies that $\beta\le 2\alpha$. 

 In the sequel we will consider that the depth parameter $R \!\in \ac{2,\dots, \bar R}$ so that  it always satisfies~\eqref{weak_error}. 
 This parameter $R$ corresponds to the depth level $L$ used in the multilevel literature.

In what follows, we will use the following ratio 
\begin{equation}\label{eq:theta}
	\theta = \sqrt{\frac{V_1}{\var(Y_0)}}
\end{equation}
which relates the quadratic rate of convergence of $Y_h$ to $Y_0$ and the variance of $Y_0$. 

\begin{Remark}[Alternative strong approximation error assumptions]\label{rq:strong_assumption} Throughout the paper, whenever  $(WE_{\alpha, 1})$ holds (with $c_1\neq 0$), \eqref{strong_error} can be replaced  by one of the following assumptions
 \begin{equation} \tag{$V\!ar_\beta$} \label{var_error}
	\exists\, \beta > 0, \; V_1 \in \R_+, \quad \var\pa{Y_h - Y_0} \le V_1 h^{\beta}, \;h\!\in \Hr.
\end{equation}
\begin{equation} \tag{$V\!ar'_\beta$}\label{var_error'}
	\exists\, \beta > 0, \; V_1 \in \R_+, \quad \var\pa{Y_h - Y_{h'}} \le V_1 \abs{h-h'}^{\beta}, \;h,\,h'\!\in \Hr,
\end{equation}
The aim being is to reduce the numerical value of $V_1$ in view of practical implementation for~\eqref{var_error} and  the estimate of the variance of refined levels for~\eqref{var_error'}. 
Note that, still if $(WE_{\alpha, 1})$ holds,  $(\eqref{var_error}   \hbox{ and }   \alpha\le 2\beta)$ $\Longleftrightarrow$ \eqref{strong_error} (the first one with a lower $V_1$ if $\alpha=2\beta$) and  \eqref{var_error'} $\Rightarrow$ \eqref{var_error}  by letting $h'\to 0$ (the converse being false).
 
    The theoretical results obtained using one of these two assumptions (combined with~\eqref{weak_error}) may
  induce slight modifications in the exposition of the results of  this paper (in particular Theorem~\ref{thm:phaseI}). These variants are briefly discussed in  Remark~\ref{rq:strong_assum2} and practical guidelines for the estimation  of $V_1$ are discussed in Section~\ref{sec:practi} (Practitioner's corner).
\end{Remark}

All estimators considered in this work are based on independent copies $(Y^{(j)}_h)_{h\in \Hr }$,   (attached to  random variables $Y^{(j)}_0$) of $(Y_h)_{h\in \Hr}$,  $j=1,\dots,R$. All random variables are supposed to be defined on the same probability space $(\Omega,{\cal A}, \P)$. Note that, since the above properties~\eqref{strong_error} and~\eqref{weak_error}, $\bar R \ge 1$, only depend on the distribution of $(Y_h)_{h\in \Hr}$, all these copies will also satisfy these two properties.

We associate to the family $(Y_h)_{h\in \Hr}$  and a given bias parameter $h\!\in \Hr$, the $\R^R$-valued random vector
$$
Y_{h, \underline n} = (Y_h, Y_{\frac{h}{n_2}}, \dots, Y_{\frac{h}{n_{_R}}})
$$ 
where the $R$-tuple  of integers  $\underline n: =(n_1,n_2,\dots,n_{_R})\!\in \N^R$, called {\em refiners} in the sequel,  satisfy
\begin{equation*}
	 n_1 = 1 < n_2 < \cdots < n_{_R}.
\end{equation*}
One defines likewise $Y^{(j)}_{h,\underline n}$ for the (independent) copies of $Y_{h, \underline n}$. 
 
 \medskip
 \noindent $\rhd$ {\em Specification of the refiners}: In most applications, we will choose refiners $n_i$ as $n_i= M^{i-1}$ where $M \ge 2$. Indeed, this is the standard choice in the regular Multilevel Monte Carlo method as described in~\cite{GILES}. Other choices like $n_i=i$ are possible (see below the original Multistep Richardson-Romberg estimator in Section~\ref{sec:MSRR}).

\subsection{Crude Monte Carlo estimator}
In our formalism a crude Monte Carlo simulation and its cost can be described as follows.
\begin{Proposition}\label{pro:MC} Assume~\eqref{weak_error} with $c_1\neq 0$ and~\eqref{strong_error} with $\bar R \ge 1$. The Monte Carlo estimator of $\esp{Y_0}$ defined by 
\begin{equation*}
	\forall N \ge 1, \; h \in \Hr, \quad
	\bar Y_{h}^{N} = \frac 1 N \sum_{k=1}^N  Y^{k}_{h}
\end{equation*}
where $\pa{Y_{h}^{k}}_{k \ge 1}$ is an \iid sequence of copies of $Y_{h}$, satisfies 
\begin{equation*}
	\bias(h) = c_1 h^{\alpha} (1+\eta_1(h)), \quad \cost(h) = \frac{1}{h}, 
	\quad \upphi(h) = \frac{\var(Y_h)}{h}
\end{equation*}
and, for a prescribed $\L^2$-error $\varepsilon > 0$, the optimal parameters $h^*(\varepsilon)$ and $N^*(\varepsilon)$ solution to~\eqref{probleme_generique} are given by  
\begin{equation}\label{eq:NCrude}
	h^*(\varepsilon) = (1 + 2 \alpha)^{-\frac{1}{2 \alpha}} \pa{\frac{\varepsilon}{|c_1|}}^{\frac{1}{\alpha}}, \quad  
	N^*(\varepsilon) = \pa{1+\frac{1}{2\alpha}} \frac{\var(Y_0)(1 + \theta (h^*(\varepsilon))^{\frac{\beta}{2}})^2}{\varepsilon^2}.
\end{equation}
Furthermore, we have 
\begin{equation*}
	\limsup_{\varepsilon \to 0} \varepsilon^{2 + \frac{1}{\alpha}} \min_{
		\substack{h \in \Hr, \\
		|\bias(h)| < \varepsilon}
	} \Cost(\bar Y^N_h) \le |c_1|^{\frac{1}{\alpha}} \pa{1 + \frac{1}{2 \alpha}}(1+2 \alpha)^{\frac{1}{2 \alpha}} \var(Y_0).
\end{equation*}
\end{Proposition}

\begin{proof} The proof is postponed to Appendix B.
\end{proof}

We refer to the seminal paper~\cite{Duffie} for more details on practical implementation of this estimator.

\begin{Remark} \label{rq:CrudeMC} For crude Monte Carlo simulation, Assumption~\eqref{strong_error} can be replaced by $Y_h\stackrel{L^2}{\to} Y_0$   (without rate), provided $\var(Y_0)(1 + \theta (h^*(\varepsilon))^{\frac{\beta}{2}})^2$ is replaced by $\var\big(Y_{h^*(\varepsilon)}\big) $ in~\eqref{eq:NCrude}. 
 \end{Remark}

\subsection{Background on Multistep Richardson-Romberg extrapolation} \label{sec:MSRR}
The so-called {\em  Multistep Richardson-Romberg estimator} has been introduced in~\cite{PAG0} in the framework of Brownian diffusions. It relies on  $R$ (refined) Euler schemes $\bar X^{(\frac{h}{n_i})}$, $1 \le i \le R$, defined on a finite interval $[0,T]$ ($T > 0$),  where the bias parameter $h=\frac{T}{n}$, $n \ge 1$. In that case, the refiners are set as $n_i=i$, $i=1,\ldots,R$, (in order to produce a better control of both the variance  and the complexity for the proposed estimator, see Remark~\ref{rem2.5} below). The main results are obtained when all the schemes are {\em consistent} \ie such that all the Brownian increments are generated from the same underlying Brownian motion. As a consequence, under standard smoothness assumptions on the coefficients of the diffusion, the family $Y_h = \bar X^{(h)}$, $h\!\in \Hr=\{\frac{T}{n}, \, n\ge 1\}$, makes up an admissible family in the above sense, as will be seen further on in more details.

For a refiner  vector $(n_1,n_2,\dots,n_{_R})$ we define the {\em weight} vector $\w = (\w_1, \dots, \w_R)$ as the unique solution to the Vandermonde system $V \!\w = e_1$ where
\begin{equation}\label{eq:Vdmde}
	V = V(1,n_2^{-\alpha},\dots,n_{_R}^{-\alpha}) = \left( \begin{array}{cccc}
		1 & 1 & \cdots & 1 \\
			1 & n_2^{-\alpha} & \cdots & n_{_R}^{-\alpha} \\
		\vdots & \vdots & \cdots & \vdots \\
		1 & n_2^{-\alpha(R-1)} & \cdots & n_{_R}^{-\alpha(R-1)} \\
	\end{array} \right).
\end{equation}
The solution $\w$ of the system has a closed form given by Cramer's rule (see Lemma~\ref{lem:VdM} in Appendix~A):
\begin{equation}
	\label{eq:alpha}
	\forall i \in \ac{1,\dots,R},\quad \w_i = \frac{(-1)^{R-i} n_i^{\alpha(R-1)}}{\ds \prod_{1 \le j < i}(n^{\alpha}_i-n^{\alpha}_j)  \prod_{i < j \le R}(n^{\alpha}_j-n^{\alpha}_i) }.
\end{equation}
 We also derive the following identity of interest
\begin{equation}\label{alphatilde}
\widetilde \w_{_{R+1}}:= \sum_{i=1}^R\frac{\w_i}{n_i^{\alpha R}}=\frac{(-1)^{R-1}}{\underline n!^{\alpha}},
\end{equation}
which will be used in \eqref{eq:Multistepbias} and \eqref{eq:etaRn} to control the residual bias.

Note that all coefficients $(\w_i)_{1 \le i \le R}$ depend on the depth $R$ of the combined extrapolations. For the standard choices $n_i = i$ or $n_i = M^{i-1}$, $i=1,\ldots,R$, we obtain the following expressions: 
\begin{equation} \label{eq:weights}
	\w_i = \begin{cases}
		\ds \frac{(-1)^{R-i} i^{\alpha R}}{\prod_{j=1}^{i-1} (i^{\alpha}-j^{\alpha})\prod_{i+1}^{R} (j^{\alpha}-i^{\alpha})} & \text{ if $n_j = j, \; j\in\ac{1,\dots,R}$ }, \\\\
		\ds \frac{(-1)^{R-i} M^{-\frac{\alpha}{2}(R-i)(R-i+1)}}{\prod_{j=1}^{i-1}(1-M^{-j\alpha }) 
	\prod_{j=1}^{R-i}(1-M^{-j\alpha})} & \text{ if $n_j = M^{j-1}, \; j\in\ac{1,\dots,R}$ }. \\
	\end{cases}
\end{equation}

Note that when $\alpha=1$ and $n_i=i$, then $\displaystyle \w_i =  \frac{(-1)^{R-i} i^{R}}{i!(R-i)!}$, $i=1,\ldots,R$.

Assume now~\eqref{weak_error} and $R\in\ac{1,\dots,\bar R}$. In order to design an estimator which kills the bias up to order $R$, we  focus on the random variable resulting from  the linear combination $\ds \psca{\w, Y_{h,\underline n}} = \sum_{i=1}^R \w_i Y_{\frac{h}{n_i}}$. 

The first equation of the Vandermonde system $V  \w = e_1$, namely $(V\w)_1= \sum_{r=1}^R \w_r = 1$, implies 
$$
\lim_{h \to 0} \esp{\psca{\w, Y_{h,\underline n}}}  = \esp{Y_0}.
$$
 Furthermore, when expanding the (weak) error, one checks that the  other $R-1$  equations satisfied by the weight vector $\w$ make  all  terms  in front of the $c_{r-1}$, $r = 2, \dots, R$ vanish. Finally, we obtain 
\begin{align}\label{eq:Multistepbias}
  \esp[1]{ \psca{\w, Y_{h,\underline n}}}& = \esp{Y_0} +\sum_{r=2}^{R} c_{r-1} h^{\alpha (r-1)} \big(V\w)_{r}  + c_{_R} \widetilde \w_{_{R+1}}h^{\alpha R}\big(1+ \eta_{R, \underline n}(h)\big)\\
 & = \esp{Y_0} + c_{_R} \widetilde \w_{_{R+1}}h^{\alpha R}\big(1+ \eta_{R, \underline n}(h)\big)
\end{align}
where 
\begin{equation}\label{eq:etaRn}
\eta_{R, \underline n}(h)= \frac{1}{c_{_R} \widetilde \w_{_{R+1}}} \sum_{r=1}^R \frac{\w_r}{n_r^{\alpha R}}\eta_{_R}\Big(\frac{h}{n_r}\Big)\to 0\quad\mbox{ as }\quad h\to 0.
\end{equation}

\begin{Proposition}\label{pro:multiRR} 
	Assume~\eqref{weak_error} and~\eqref{strong_error}. Let $R \in \ac{2, \dots, \bar R}$ be such that $c_{_R}\neq 0$.  The Multistep Richardson-Romberg estimator of $\esp{Y_0}$ defined by 
\begin{equation}\label{MSR-estimator}
	\forall N \ge 1, \; h \in \Hr, \quad \bar Y_{h, \underline n}^{N} = \frac 1 N \sum_{k=1}^N  \psca{\w, Y^{k}_{h,\underline n}}=  \psca{\w, \frac 1 N \sum_{k=1}^NY^{k}_{h,\underline n}}
\end{equation}
where $\pa{Y_{h, \underline n}^{k}}_{k \ge 1}$ is an \iid sequence of copies of $Y_{h, \underline n}$, satisfies 
\begin{equation*}
	\bias(h) = (-1)^{R-1} c_{_R} \pa{\frac{h^R}{\underline n!}}^{\alpha} \pa{1 +\eta_{R,\underline n}(h)},
	\quad \cost(h) = \frac{|\underline n|}{h}, 
	\quad \upphi(h) = \frac{|\underline n| \var(\psca{\w, Y_{h, \underline n}})}{h}
\end{equation*}
and, for a prescribed $\L^2$-error $\varepsilon > 0$, the optimal parameters $h^*(\varepsilon)$ and $N^*(\varepsilon)$ solution of~\eqref{probleme_generique} are
\begin{equation*}
	h^*(\varepsilon) = (1 + 2 \alpha R)^{-\frac{1}{2 \alpha R}} \pa{\frac{\varepsilon}{|c_{_R}|}}^{\frac{1}{\alpha R}} {\underline n !^{\frac{1}{R}}}, \quad  
	N^*(\varepsilon) = \pa{1+\frac{1}{2\alpha R}} \frac{\var(Y_0)(1 + \theta (h^*(\varepsilon))^{\frac{\beta}{2}})^2}{\varepsilon^2}.
\end{equation*}
Furthermore,
\begin{equation} \label{pro:multiRR:result}
	\inf_{
		\substack{h \in \Hr \\
		|\bias(h)| < \varepsilon}
	} \Cost(\bar Y^N_h) 
	\sim \pa{\frac{(1+2\alpha R)^{1+\frac{1}{2\alpha R}}}{2 \alpha R}} \frac{|c_{_R}|^{\frac{1}{\alpha R}} \abs{\underline n} \var(Y_0)}
	{\underline n !^{\frac{1}{R}} \varepsilon^{2 + \frac{1}{\alpha R}} } \quad \text{as } \varepsilon \to 0.
\end{equation}
\end{Proposition}
\begin{proof}The proof is postponed to Appendix~B (but takes advantage of the formalism developed in the next section).
\end{proof}

\begin{Remark}\label{rem2.5} $\bullet$ As for~\eqref{strong_error}, Remark~\ref{rq:CrudeMC} still applies 

\smallskip
\noindent $\bullet$ In this approach the bias reduction suffers from an increase of the simulation cost by the $|\underline n|$ factor which appears in the numerator of~\eqref{pro:multiRR:result}. The choice of the refiners made in~\cite{PAG0}, namely $n_i = i$, $i=1,\ldots,R$, is justified  by the control  of the ratio $\frac{\abs{\underline n}}{\underline n !^{\frac 1R}}$: for such a choice,  it behaves linearly,  like $\frac{e}{2} (R+1)$,  for large values of $R$ whereas with $n_i=M^{i-1}$ it goes to infinity geometrically in  $O(M^{\frac{R-1}{2}})$.
\end{Remark}

\section{A paradigm for Multilevel simulation methods}\label{sec:MLRR}
\subsection{General framework}
\subsubsection*{Multilevel decomposition}
In spite of Proposition~\ref{pro:multiRR} which shows that the numerical cost of the Multistep method behaves like $\varepsilon^{2 + \frac{1}{\alpha R}}$, one observes in practice that the increase of the ratio $\frac{|\underline n|}{\underline n !}$ (when $R$ grows) in front of $\var (Y_0)$ in~\eqref{pro:multiRR:result} reduces the impact of the bias reduction. 

An idea is then to introduce independent linear combination of copies of $\bar Y_{h, \underline n}$ to reduce the variance taking advantage of the basic fact that if $X$ and $X'$ are independent with the same distribution then $\esp{\frac{X+X'}{2}} = \esp{X}$ and $\var(\frac{X+X'}{2})= \frac 12 \var(X)$, combined  with an appropriate allocation policy to control the complexity of the resulting estimator.
So, let us consider now $R$ \emph{independent} copies  $(Y_{h,\underline n}^{(j)})$, $j=1,\dots,R$, of the random vector $Y_{h,\underline n}$ and the  linear combination 
  \begin{equation*}
	  \sum_{j = 1}^R \psca{\mT^j, Y_{h,\underline n}^{(j)}} = \sum_{i,j=1}^R \mT_{i}^j Y_{\frac{h}{n_i}}^{(j)}
\end{equation*}
where $\mT = [\mT^1\dots\mT^R]$ is an $R\times R$ matrix with column vectors $\mT^j\!\in \R^R$ satisfying the constraint 
$$
\sum_{1\le i,j\le R} \mT^j_i =1.
$$
As emphasized further on, we will also need that each column vector $\mT^j$, $j \in {2, \dots, R}$, has zero sum. In turn, this suggests to introduce the  notion of Multilevel estimator as a family of ``stratified'' estimators of $\esp{Y_0}$ attached to  the random vectors $\psca{\mT^j, Y_{h,\underline n}^{(j)}}$, $j=1,\ldots,R$. This leads to the following definitions. 

\begin{Definition}[Design matrix] Let $R \ge 2$. An $R\times R$-matrix  $\mT = [\mT^1\dots\mT^R]$ is an $R$-level design matrix if 
\begin{equation} \label{Abeta} %\tag{$\A_{\mT}$}
		 \psca{\mT^j, \mathbf{1}} = \sum_{i = 1}^R \mT_i^j = 0, \; j = 2,\ldots,R.
\end{equation}
\end{Definition}

Note that such a design matrix always satisfies $\displaystyle \sum_{i,j = 1}^d \mT_i^j =1$.

\begin{Definition}[General Multilevel estimator] Let $R \ge 2$ and let $\pa[1]{Y_{h,\underline n}^{(j),k}}_{k \ge 1}$ be an \iid sequence of copies of $Y_{h,\underline n}^{(j)}$.
	A Multilevel estimator of depth $R$ attached to an allocation policy $q=(q_1,\ldots,q_{_R})$ with  $ q_j >0$, $j=1,\ldots,R$, and $\sum_j q_j = 1$ and a design matrix $\mT$, is defined for every integer $N \ge 1$ and $h \in \Hr$ by 
\begin{equation}\label{EstimMatr}
	\bar Y^{N,q}_{h,\underline n} = \sum_{j=1}^R \frac{1}{N_j} \sum_{k=1}^{ N_j }\psca{\mT^j, Y_{h,\underline n}^{(j),k}}
\end{equation}
where for all $j \in \ac{1,\dots,R}$, $N_j = \lceil q_j N \rceil$ (allocated budget to compute $\esp{\psca{\mT^j, Y_{h,\underline n}^{(j)}}}$).

\begin{itemize}
	\item If furthermore the $R$-level design matrix $\mT$ satisfies
		\begin{equation}\label{eq:desMLMC}
			\mT^1 = e_1 \quad \text{and}\quad \sum_{j=1}^R \mT^j = e_{_R}, 
\end{equation}
the estimator is called a {\bf Multilevel Monte Carlo (\MLMC) estimator} of order $R$.

\item If, furthermore,  the $R$-level design matrix $\mT$  satisfies
	\begin{equation}\label{eq:desML2R}
\mT^1 = e_1 \quad \text{and}\quad \sum_{j=1}^R\mT^j = \w, \text{ where  $\w$ is  the unique solution to~(\ref{eq:alpha}),}
\end{equation}
the estimator is called a {\bf Multilevel Richardson-Romberg (\MLRR) estimator} of order $R$.
\end{itemize}
\end{Definition}
\begin{Remark}
\begin{itemize}
	\item Note that the assumption $\mT^1 = e_1$ is  not really necessary. It simply allows for more concise formulas in what follows.  
	\item In this framework, denoting by $\mathbf{0}$ the null column vector of $\R^R$, the crude Monte Carlo is associated to the design matrix $\mT = \pa{e_1, \mathbf{0}, \dots, \mathbf{0}}$ and the Multistep Richardson-Romberg estimator is associated to $\mT = \pa{\w, \mathbf{0}, \dots, \mathbf{0}}$.
	\item Introducing such   general families of  matrices  will allow us to justify the  final choice of design matrices.  To reduce the complexity of the resulting estimators leads us to choose  as sparse as possible design matrices satisfying the   constraints~\eqref{eq:desMLMC} or~\eqref{eq:desML2R}.
\end{itemize}
\end{Remark}

Within the abstract framework of a parametrized Monte Carlo simulation described in Section~\ref{sec:perf}, the structure parameter $\uppi$  of the multilevel estimators $(\bar Y^{N,q}_{h,\underline n})_{N \ge 1}$ defined by~\eqref{EstimMatr} is  
\begin{equation*}
	\uppi=(\uppi_0, q)\quad\mbox{where}\quad 
	\begin{cases} 
	q = (q_1,\ldots, q_{_R})\!\in (0,1)^R,\quad \sum_i q_i=1, \\
		\uppi_0 = (h, n_1,\ldots,n_{_R}, R, \mT) \in \Uppi_0.
	\end{cases}
\end{equation*}

\subsubsection*{Cost, complexity and effort of a Multilevel estimator}
In order to  minimize the effort $\upphi(\uppi)$ of the estimator~\eqref{EstimMatr}, let us first evaluate  its unitary computational complexity.
For a simulation size $N$, the numerical cost induced by the estimators $Y^{N, q}_{h, \underline n}$, $N \ge 1$, reads
\begin{equation}
 \Cost(\bar Y^{N,q}_{h,\underline n}) = 
 \sum_{j=1}^R N_j \sum_{i=1}^R \frac{1}{h} n_i\mbox{\bf 1}_{\{\mT^j_i\neq 0\}}
 = N \cost(\uppi)
\end{equation}
where  the unitary complexity  $\cost(\uppi)$ is given by  
 \begin{equation}\label{eq:unitcomplexity}
	 \cost(\uppi) = \frac{1}{h} \sum_{j=1}^R q_j \sum_{i=1}^R n_i\mbox{\bf 1}_{\{\mT^j_i\neq 0\}}.
 \end{equation}
At this stage, it is clear that the design matrix $\mT$  must be {\em as sparse as} possible  to minimize        $ \cost(\uppi) $. 
However, it may happen, like for nested Monte Carlo (see~Section~\ref{sec:nested} for details), that the unitary complexity writes
\begin{equation} \label{eq:nested_compl}
	 \cost(\uppi) = \frac{1}{h} \sum_{j=1}^R q_j \max_{1\le i\le R} \big(n_i\mbox{\bf 1}_{\{\mT^j_i\neq 0\}}\big).
\end{equation}

The variance of the Multilevel estimator is  inverse linear in $N$ (hence of Monte Carlo type) since, using the independence of the levels, we get
\begin{align*}
    \var\pa{\bar Y^{N,q}_{h,\underline n}} &= \sum_{j=1}^R \frac{1}{N_j^2} \var\pa{\sum_{k=1}^{N_j}\psca{\mT^j, Y^{(j),k}_{h, \underline n}}}\\
    &= \frac{1}{N} \sum_{j=1}^R \frac{1}{q_j} \var\pa{\psca{\mT^j, Y^{(j)}_{h, \underline n}}}
\end{align*}
so that the effort of such a Multilevel estimator is given by 
\begin{equation} \label{def:upphi} 
	\upphi(\uppi) = \upnu(\uppi) \cost(\uppi) = \pa{\sum_{j=1}^R \frac{1}{q_j} \var\pa{\psca{\mT^j, Y^{(j)}_{h, \underline n}}}} \cost(\uppi).
\end{equation}

\subsubsection*{Bias error of a Multilevel estimator}
We now establish the bias error in this general framework. The proposition below about the bias error  follows straightforwardly from the weak error expansion~\eqref{weak_error} and the definition of a design matrix $\mT$. 
\begin{Proposition} \label{prop_bias} Assume~\eqref{weak_error}.
	\begin{enumerate}
		\item[$(a)$] {\em \MLRR estimator:} Let $R\!\in \ac{2,\ldots, \bar R}$ be the depth of an \MLRR estimator. For any admissible allocation policy~$q=(q_1,\ldots,q_{_R})$, the bias error reads  
	\begin{equation}\label{eq:biasmultiRR}
		\bias(\uppi_0,q) =(-1)^{R-1} c_{_R} \pa{\frac{h^R}{\underline n!}}^{\alpha} \pa{1 +\eta_{R,\underline n}(h)}
	\end{equation}
	where $\ds \eta_{R, \underline n}(h)= (-1)^{R-1} \underline n!^{\alpha} \sum_{r=1}^R \frac{\w_r}{n_r^{\alpha R}}\eta_{_R}\Big(\frac{h}{n_r}\Big) $ (see~\eqref{eq:etaRn}) with $\eta_{_R}$ defined in~\eqref{weak_error}.
\item[$(b)$] {\em \MLMC estimator:} Let $R \ge 2$ be the depth of an \MLMC estimator. For any admissible allocation policy~$q\!=\!(q_1,\ldots,q_{_R})$, the bias error reads 
	\begin{equation}\label{eq:biasmultiMC}
		\bias(\uppi_0,q) = c_1 \pa{\frac{h}{n_{_R}}}^{\alpha} \pa[2]{1 + \eta_1\pa[1]{\frac{h}{n_{_R}}}}	
	\end{equation}
	with $\eta_1$ defined in~\eqref{weak_error}.	
	\end{enumerate}
\end{Proposition}

\subsubsection*{Toward the optimal parameters}
The optimization problem~\eqref{sol_avec_biais} is not attainable directly, so we decompose it into two successive  steps: 
\begin{description}
	\item[Step 1:]  Minimization of the effort $\upphi$ over all allocation policies $q={(q_j)}_{1\le j\le R}$ (as a function of  a fixed bias parameter $h$). In practice, we will minimize an upper-bound~$\bar \upphi$ of the effort $\upphi$ 
\begin{equation}\label{def:optPhi}
	q^* = \argmin_{q \in \calig{S}_+(R)} \bar \upphi(\uppi_0, q), \quad \text{where} \quad \upphi(\uppi) \le \bar \upphi(\uppi), \quad \text{and} \quad \upphi^*(\uppi_0) = \upphi(\uppi_0, q^*).
\end{equation}
This phase is solved in~Theorem~\ref{thm:phaseI} below (an explicit expression for $\bar \upphi$ is provided in~\eqref{def:bar_upphi}). The quantity $\upphi^*(\uppi_0)$  is called the {\em optimally allocated effort} (with a slight abuse of terminology since $\bar \upphi$ is only an upper bound of $\upphi$).

	\item[Step 2:] Minimization of the resulting cost as a function of the remaining parameters $\uppi_0$ for a prescribed $\L^2$--error $\varepsilon > 0$ (and specification of the resulting size of the simulation and its cost): 
		\begin{equation*}
			\uppi_0(\varepsilon) = \argmin_{\substack{\uppi_0 \in \Uppi_0\\ \,|\!\bias(\uppi_0, q^*)|< \varepsilon}} \pa{\frac{\upphi^*(\uppi_0)}{\varepsilon^2 - \bias^2(\uppi_0, q^*)}}, \quad
			N(\uppi_0(\varepsilon)) = \frac{\upphi^*(\uppi_0(\varepsilon))}{\cost(\uppi_0(\varepsilon), q^*) (\varepsilon^2 -  \bias^2(\uppi_0, q^*))}. 
\end{equation*}
It will be  solved asymptotically when $\varepsilon$ goes to $0$  in two sub-steps. First we consider a fixed depth~$R$ (with  general refiners) in Proposition~\ref{thm:phaseII}  which provides a closed form for $h^*(\varepsilon)$. Secondly,  we let $R$  vary  as a function of $\varepsilon$ (only for geometric refiners $n_i=M^{i-1}$). This leads to the main result of the paper Theorem~\ref{prop:asympReps} which yields a closed  form for  $R^*(\varepsilon)$  (and $N^*(\varepsilon)$) and the various asymptotics for the cost,  depending on $\beta$ and other structural parameters. 
\end{description}

\subsection{Optimally allocated effort (Step 1)}
Throughout our investigations on these estimators, we will make extensive use  of the following lemma which is a straightforward consequence of Schwarz's Inequality including its  equality case.
\begin{Lemma} \label{lemme_CS}
	For all $j \in \ac{1,\dots,R}$, let $a_j > 0$, $b_j > 0$ and $q_j > 0$ such that $\ds \sum_{j=1}^R q_j = 1$. Then 
	\begin{equation*}
		\pa{\sum_{j=1}^R \frac{a_j}{q_j}} \pa{\sum_{j=1}^R b_j q_j} \ge \pa{\sum_{j=1}^R \sqrt{a_j b_j}}^2
	\end{equation*}
	and equality holds if and only if $q_j = \mu \sqrt{a_j b_j^{-1}}$, $j=1,\dots,R$, with $\mu = \pa[2]{\sum_{k=1}^R \sqrt{a_k b_k^{-1}}}^{-1}$.
\end{Lemma}

\begin{Theorem}\label{thm:phaseI} Assume~\eqref{strong_error} holds and let $\theta$ be defined by~\eqref{eq:theta}.
Then, the optimally allocated effort $\upphi^*$ defined by~\eqref{def:optPhi} satisfies 
\begin{equation*}
	\upphi^*(\uppi_0)
	\le \bar \upphi(\uppi_0, q^*) = \frac{\var(Y_0)}{h} \pa{1
	+ \theta h^{\frac{\beta}{2}} \sum_{j=1}^R 
	\pa[3]{\sum_{i=1}^R \abs{\mT_i^j} n_i^{-\frac{\beta}{2}}}
	\pa[3]{\sum_{i=1}^R n_i \ind{\mT_i^j \neq 0}}^{\frac 12}}^2
\end{equation*}
where $q^*=q^*(\uppi_0)$ is an optimal policy (with respect to the upper bound $\bar \upphi$) given by  
\begin{equation} \label{eq:q^*}
	\begin{cases}
		\ds q^*_1(\uppi_0) = \mu^*_{_R} (1 +\theta h^{\frac \beta 2}) \\
		\ds q^*_j(\uppi_0) = \mu^*_{_R} \theta h^{\frac \beta 2} \pa[3]{\sum_{i=1}^R  \abs{\mT_i^j} n_i^{-\frac{\beta}{2}}} \pa[3]{\sum_{i=1}^R n_i \ind{\mT_i^j \neq 0}}^{-\frac{1}{2}}, \; j=2,\ldots,R,
	\end{cases}
\end{equation} 
and $\mu^*_{_R}$ is the  normalizing constant such that $\sum_{j=1}^R q^*_j = 1$. 
\end{Theorem}

\begin{proof}
    Under assumption~\eqref{Abeta}, we have $\psca{\mT^1, Y^{(1)}_{h, \underline n}}  = Y_h^{(1)}$ and, for every $j\!\in \{2,\ldots,R\}$,  $\psca{\mT^j, Y^{(j)}_{h, \underline n}} =  \psca{\mT^j, Y_{h,\underline n}^{(j)} - Y_0^{(j)}\mathbf{1}}$ since $\psca{\mT^j, \mathbf{1}}=0$. Hence, using the sub-additivity of standard deviation derived from (Minkowski's inequality) and the strong error assumption, we obtain   
    \begin{align} \label{eq:maj_var}
        \begin{split}
            \forall j \ge 2, \quad \var \pa{\psca{\mT^j, Y_{h,\underline n}^{(j)} }} 
    &=  \sigma\pa{\sum_{i=1}^R \mT_i^j \pa{Y_{\frac{h}{n_i}}^{(j)} - Y_0^{(j)} }}^2 
    \le \pa{\sum_{i=1}^R \abs{\mT_i^j} \sigma\pa{Y_{\frac{h}{n_i}}^{(j)} - Y_0^{(j)} }}^2, \\
	& \le V_1 h^{\beta} \pa{\sum_{i=1}^R \abs{\mT_i^j} n_i^{-\frac{\beta}{2}} }^2.
\end{split}
\end{align}
The variance of the Multilevel estimator is then given by
\begin{equation}\label{eq:varianceML}
	\var\pa{\bar Y_{h, \underline n}^{N,q}} \le \frac{1}{N} \pa{\frac{\var\pa{Y_h^{(1)}}}{q_1} + V_1 h^{\beta} \sum_{j=2}^R \frac{1}{q_j} \pa{\sum_{i=1}^R  \abs{\mT_i^j}n_i^{-\frac{\beta}{2}} }^2}.
\end{equation}

On the other hand, we have
\begin{align*}
	\var\pa{Y_h^{(1)}}=\var\pa{Y_h} &\le \esp{Y_h - \esp{Y_0} }^2\\
			     &\le   \normLp{2}{Y_h-Y_0}^2 +2 \esp{(Y_h-Y_0)(Y_0-\esp{Y_0})}+ \var\pa{Y_0}\\
			      &\le \var(Y_0) + V_1 h^{\beta} +2 \sqrt{V_1}h^{\beta/ 2}\sqrt{\var{Y_0}}
			      =   \var(Y_0)(1 +\theta h^{\frac \beta 2})^2.
\end{align*}
Combining~\eqref{eq:unitcomplexity}, the above inequality~\eqref{eq:varianceML} and the above upper-bound for $\var\pa{Y_h^{(1)}}$,  we derive the following upper bound   $\bar \upphi(\uppi)$ for the effort $\upphi(\uppi) $ defined by 
\begin{equation} \label{def:bar_upphi}
	\bar \upphi(\uppi) = \frac{\var(Y_0)}{h} \pa{
	\frac{(1 +\theta h^{\frac \beta 2})^2}{q_1} + \theta^2 h^{\beta} \sum_{j=2}^R  \frac{1}{q_j} 
\pa{ \sum_{i=1}^R \abs{\mT_i^j}n_i^{-\frac{\beta}{2}}}^2} 
\pa{ \sum_{i,j=1}^R q_j n_i \ind{\mT_i^j \neq 0}}.
\end{equation}
Applying Lemma \ref{lemme_CS} with $a_1 = (1 +\theta h^{\frac \beta 2})^2$, $b_1=1$ and $\ds a_j = \theta^2 h^{\beta} \pa{\sum_{i=1}^R   \abs{\mT_i^j}n_i^{-\frac{\beta}{2}}}^2 $, $\ds b_j = \sum_{i=1}^R n_i \ind{\mT_i^j \neq 0}$, $j\! \in \ac{2,\dots,R}$ completes the proof. 
\end{proof}
\begin{Remark}[Accuracy of the bound] \label{rq:strong_assum2} 
    As announced in Remark~\ref{rq:strong_assumption}, we can replace the strong error assumption~\eqref{strong_error} by a slight modified version $e.g.$ $\var\pa{Y_h-Y_{h'}} \le V_1 \abs{h-h'}^\beta$. Using this assumption, the upper bound of the previous theorem can be improved. For instance, if we make the natural choice $\mT^j = \W_j (e_j-e_{j-1})$ corresponding to the \MLRR estimator  (see Section~\ref{sec:templates}), we can replace~\eqref{eq:maj_var} by 
    \begin{equation} \label{eq:opt_majoration}
        \sigma \pa{\psca{\mT^j, Y_{h, \underline n}^{(j)}}} = |\W_j| \sigma \pa{Y_{\frac{h}{n_j}} - Y_{\frac{h}{n_{j-1}}}} \le |\W_j| \sqrt{V_1} \left|\frac{h}{n_j} - \frac{h}{n_{j-1}}\right|^{\beta/2}.
    \end{equation}
    Note that if the constant $V_1$ is sharp, the resulting upper bound derived in Theorem~\ref{thm:phaseI} is tight.
\end{Remark}

\begin{Remark}[About variance minimization] We established in the above proof that for every allocation policy $q= (q_1,\ldots,q_{_R})$, 
\begin{equation*}
	\var\pa{\bar Y_{h, \underline n}^{N,q}} \le \frac{\var(Y_0)}{N}\pa{
	\frac{(1 +\theta h^{\frac \beta 2})^2}{q_1} + \theta^2 h^{\beta} \sum_{j=2}^R  \frac{1}{q_j} 
\pa{ \sum_{i=1}^R  \abs{\mT_i^j} n_i^{-\frac{\beta}{2}}}^2}.
\end{equation*}
Then, applying Lemma \ref{lemme_CS} with $a_1 = (1 +\theta h^{\frac \beta 2})^2$, $\ds b_1 = 1$ and $\ds a_j = \theta^2 h^{\beta} \pa{\sum_{i=1}^R   \abs{\mT_i^j}n_i^{-\frac{\beta}{2}}}^2 $,  $\ds b_j = 1$, $j \!\in \ac{2,\dots,R}$, we obtain (since $\sum_{j=1}^R q_j b_j = 1$) 
\begin{equation*} 
    \inf_{q \in \calig{S}_+(R)}\var\pa{\bar Y_{h, \underline n}^{N,q}} \le \var(Y_0)\pa{1 + \theta h^{\frac{\beta}{2}} \sum_{j=1}^R \sum_{i=1}^R  \abs{\mT_i^j} n_i^{-\frac{\beta}{2}}}^2 
\end{equation*}
with an optimal choice (to minimize  the variance): $\ds q^{\dag}_1 = \mu^{\dag} (1 +\theta h^{\frac \beta 2})$, $\ds q^{\dag}_j = \mu^{\dag} \theta h^{\frac \beta 2} \pa[2]{\sum_{i=1}^R \abs{\mT_i^j}n_i^{-\frac{\beta}{2}} }$ ($\mu^{\dag}$ normalizing constant such that $\sum_{j=1}^n q^{\dag}_j\!=\! 1$). Note that this choice $q^\dag$ differs from the optimal one $q^*$ obtained in Theorem~\ref{thm:phaseI}.
\end{Remark}

\subsection{Resulting cost optimization (Step 2)}
 
\subsubsection{Bias parameter optimization ($R$ fixed)}
In this first sub-step, we fix the depth $R \ge 2$, the design matrix $\mT$ and the refiners $n_1, \dots, n_{_R}$ and we only optimize the bias parameter $h\in \Hr$ with respect to $\varepsilon > 0$, so that
\begin{equation*}
	\uppi_0(\varepsilon) = h(\varepsilon, n_1, \dots, n_{_R}, R, \mT).
\end{equation*}

We recall that $\upphi^*(h) \le \bar \upphi(h, q^*) =: \bar \upphi^*(h)$ where 
\begin{equation} \label{def:bar_upphi_star}
	\bar \upphi^*(h) = \frac{\var(Y_0)}{h} \pa{1
	+ \theta h^{\frac{\beta}{2}} \sum_{j=1}^R 
	\pa[3]{\sum_{i=1}^R \abs{\mT_i^j} n_i^{-\frac{\beta}{2}}}
	\pa[3]{\sum_{i=1}^R n_i \ind{\mT_i^j \neq 0}}^{\frac 12}}^2.
\end{equation}

\begin{Proposition}[Bias parameter optimization]\label{thm:phaseII} Assume~\eqref{weak_error} and~\eqref{strong_error}. Let $R \ge 2$ and let $n_i$, $i=1,\ldots,R$ ,be fixed refiners.
\begin{enumerate}
	\item[$(a)$]{\em \MLRR estimator:} Let $R\!\in \ac{2,\ldots, \bar R}$ be  such that $c_{_R}\neq 0$. A   \MLRR estimator of depth $R$ obtained with the allocation policy $ q^*$ defined by~\eqref{eq:q^*} and a bias parameter  	
	 \begin{equation}\label{eq:hMLRR}
			h^*(\varepsilon, R) = (1 + 2\alpha R)^{-\frac{1}{2\alpha R}} \pa{\frac{\varepsilon}{|c_{_R}|}}^{\frac{1}{\alpha R}} \underline n!^{\frac{1}{R}}
		\end{equation}
 achieves the asymptotic minimal cost, namely
		\begin{equation*}
			\inf_{\substack{h \in \Hr \\ \,|\!\bias(h, q^*)| < \varepsilon}}
			\Cost \pa{\bar Y^{N,q^*}_{h, \underline n}} 	
			\sim \pa[2]{\frac{(1+2\alpha R)^{1+\frac{1}{2\alpha R}}}{2 \alpha R}}
			\frac{|c_{_{R}}|^{\frac{1}{\alpha R}} \var(Y_0)}{\underline n !^{\frac{1}{R}}\varepsilon^{2+\frac{1}{\alpha R}}}\quad \mbox{ as } \quad \varepsilon \to 0.
		\end{equation*}
		\item[$(b)$]{\em \MLMC estimator:} Assume $c_1 \neq 0$. An \MLMC estimator of depth $R$ obtained with the allocation policy $q^*$ defined in~\eqref{eq:q^*} and a  bias parameter  		
		\begin{equation}\label{eq:hMLMC}
			h^*(\varepsilon, R) = (1 + 2\alpha)^{-\frac{1}{2\alpha}} \pa{\frac{\varepsilon}{|c_1|}}^{\frac{1}{\alpha}} n_{_R}
		\end{equation}
 achieves the asymptotic minimal cost, namely
		\begin{equation*}
			\inf_{\substack{h \in \Hr \\ \,|\!\bias(h, q^*)| < \varepsilon}}
			\Cost \pa{\bar Y^{N,q^*}_{h, \underline n}} 	
			\sim \pa[2]{\frac{(1+2\alpha)^{1+\frac{1}{2\alpha}}}{2 \alpha}}
			\frac{|c_1|^{\frac{1}{\alpha}} \var(Y_0)}{n_{_R} \varepsilon^{2 + \frac{1}{\alpha}}}\quad \mbox{ as } \quad \varepsilon \to 0.
		\end{equation*}
		\end{enumerate}
 \end{Proposition}

\begin{proof} $(a)$ By definition of the effort $\upphi$ and the bias $\bias$ of the estimator, we have (see Section~\eqref{sec:perf}) 
	\begin{equation*}
		\Cost \pa{\bar Y^{N,q^*}_{h, \underline n}} = \frac{\phi^*(h)}{\varepsilon^2 - \bias^2(h, q^*)}.
	\end{equation*}
	It follows from~\eqref{def:bar_upphi_star} that the cost minimization problem
	is upper-bounded by the more tractable pro\-blem  
\begin{equation*}
	\inf_{h \in \Hr,\,|\!\bias(h, q^*)| <\varepsilon} \frac{h \bar \upphi^*(h)}{h(\varepsilon^2-\bias^2(h, q^*))}
\end{equation*}
with a bias $\bias(h, q^*)$ satisfying~\eqref{eq:biasmultiRR}. First note that $\lim_{h \to 0} h \bar \upphi(h, q^*) = \var(Y_0)$. 
We will consider now the denominator $h( \varepsilon^2 - \bias^2(h, q^*))$. 
Elementary computations show that, for  fixed real numbers $a,\,R'>0$,  the function $ g_{a,R'}$  defined by $g_{a,R'}(\xi)= \xi(1-a^{2} \xi^{2 R'})$, $\xi>0$, satisfies 
\[
\xi(a, R'):=  {\rm argmax}_{\xi >0}g_{a,R'}(\xi)= \big( (2R'+1)^{\frac 12}a\big)^{-\frac{1}{R'}} \quad\text{ and }\quad \max_{(0,+\infty)}g_{a,R'}= \frac{2R'}{(2R'+1)^{1+\frac{1}{2R'}}}a^{-\frac{1}{R'}}.
\]
Then, set $R'=R\alpha$, $\tilde a= \frac{|\widetilde \w_{R+1}c_{_R}|}{\varepsilon}$. Inspired by what precedes, we make the sub-optimal choice 
$\displaystyle h(\varepsilon) =h(\varepsilon,R,\alpha) =  \xi\Big(\tilde a, \alpha R\Big)= \left(\frac{\varepsilon\, }{(2\alpha R+1)^{\frac 12} |c_{_R}|}\right)^{\frac{1}{\alpha R}}\underline n!^{\frac 1R}$ corresponding to the case $\eta_{R,\underline n} \equiv 0$. It is clear that, at least for small enough $\varepsilon$, $\bias^2(h, q^*)<\varepsilon^2$ which makes this choice admissible. Hence
\begin{equation}\label{eq:upperCost}
\inf_{h \in \Hr, \\ \,|\!\bias(h, q^*)| < \varepsilon} \frac{\upphi^*(h)}{\varepsilon^2 - \bias^2(h, q^*)}
\le \Big(1+\frac{1}{2\alpha R}\Big)(2\alpha R+1)^{\frac{1}{2\alpha R}}|c_{_{R}}|^{\frac{1}{\alpha R}}\frac{h(\varepsilon) \bar\upphi^*(h(\varepsilon))}{\underline n !^{\frac{1}{ R}}\varepsilon^{2+\frac{1}{\alpha R}}}\frac{1}{1-\frac{(\eta_{R,\underline n}(h(\varepsilon))+1)^2-1}{2\alpha R}}.
\end{equation}
The ``limsup" side of the result follows since $\lim_{h\to 0}\eta_{R,\underline n}(h)=0$. 

On the other hand, it follows from the definition~\eqref{def:upphi} of the effort $\upphi$ that 
\begin{equation*}
	\upphi^*(h) = \frac{1}{h} \pa{\sum_{j=1}^R \frac{1}{q^*_j} \var\pa{\psca{\mT^j, Y^{(j)}_{h, \underline n}}}} \pa{\sum_{i,j=1}^R q_j n_i \ind{\mT_{i}^j \neq 0}}.
\end{equation*}
Then Schwarz's Inequality implies
\begin{align*}
\upphi^*(h) 
&\ge \frac{1}{h} \pa{\sum_{j=1}^R \sqrt{\var\pa{\psca{\mT^j, Y^{(j)}_{h, \underline n}}}}\sqrt{\sum_{i=1}^R n_i\mbox{\bf 1}_{\{\mT^j_i\neq 0\}}}}^2\\
&\ge \frac{1}{h}\max_{1\le j\le R} \pa{\var\pa{\psca{\mT^j, Y^{(j)}_{h, \underline n}}}\sum_{i=1}^R n_i\mbox{\bf 1}_{\{\mT^j_i\neq 0\}}}\\
&\ge\frac{1}{h}\max_{1\le j\le R} \var\pa{\psca{\mT^j, Y^{(j)}_{h, \underline n}}}
\end{align*}
since $n_i\ge n_1=1$, $i=1,\ldots,R$.
Denoting $g(h) = \max_{1\le j\le R} \var\pa{\psca{\mT^j, Y^{(j)}_{h, \underline n}}}$ one clearly has $\lim_{h\to 0} g(h) = \var(Y_0)$ under the strong assumption~\eqref{strong_error} and, as a consequence,  $\lim_{h\to 0} h\, \upphi(h) = \var(Y_0)$. Hence, the cost minimization problem is lower bounded by the more explicit problem
\begin{equation*}
	\inf_{\substack{h \in \Hr \\ \,|\!\bias(h, q^*)| < \varepsilon}}
	\frac{g(h)}{h(\varepsilon^2-\bias^2(h, q^*))}.
\end{equation*}
Let $\eta\!\in (0,1)$. There exists $\varepsilon_{\eta}>0$ such that, for every $h\! \in (0,h(\varepsilon_{\eta}))$,  
\[
|g(h)-\var(Y_0)|\le \eta \var (Y_0) \quad \mbox{ and }\quad |\eta_{R,\underline n}(h)| \le \eta. 
\]
Let $\varepsilon\!\in (0, \varepsilon_{\eta})$. We derive from Equation~(\ref{eq:biasmultiRR}) that 
\[
	\bias(h(\varepsilon_{\eta}), q^*)^2 \ge \frac{\varepsilon_{\eta}^2(1-\eta)}{2\alpha R+1}.
\]
Consequently, if $\varepsilon < \frac{\varepsilon_{\eta}\sqrt{1-\eta}}{\sqrt{2\alpha R+1}}$, for every $h>0$ such that $\mu^2(h, q^*)< \varepsilon^2$, one has
\[
\frac{g(h)}{h(\varepsilon^2-\bias(h,q^*)^2)}\ge \frac{\var(Y_0)(1-\eta)}{h(\varepsilon^2- (1-\eta)(\widetilde \w_{_{R+1}} c_{_R} )^2 h^{2\alpha R} )}.
\]
Taking advantage of what was done in the ``$\limsup$" part, we get
\[
	\inf_{\substack{h \in \Hr \\ \bias(h, q^*)< \varepsilon}}
	\frac{g(h)}{h(\varepsilon^2-\bias(h, q^*)^2)} 
\ge \Big(1+\frac{1}{2\alpha R}\Big)(2\alpha R+1)^{\frac{1}{2\alpha R}}|c_{_{R}}|^{\frac{1}{\alpha R}} \,\frac{\var(Y_0)}{\underline n !^{\frac{1}{R}}\varepsilon^{2+\frac{1}{\alpha R}}}(1-\eta)^{1+\frac{1}{2\alpha R}}.
\]
Letting  $\varepsilon $ and $\eta$ successively  go to zero, yields the ``$\liminf$" side.

\medskip
\noindent $(ii)$ 
Owing to~\eqref{eq:biasmultiMC}, the bias $\bias(h, q)$ is now given by
\[
\bias(h, q) = \Big(\frac{h}{n_{_R}}\Big)^{\alpha} \left(c_1+\eta_{_1}\Big(\frac{h}{n_{_R}}\Big)\right)\quad \mbox{with }\quad \lim_{h\to 0} \eta_{_1}(h)=0.
\] 
Following the lines of the proof of $(i)$ with $R'=\alpha$ completes the proof.
\end{proof}

\begin{Remark} \label{rq:thm3.8}	\begin{itemize}
		\item[$\bullet$] The fact that the function $\lim_{h \to 0} h \upphi^*(h) = \var(Y_0)$ follows from the $L^2$-strong convergence of $Y_h$ toward $Y_0$. Its rate of  convergence plays no explicit role in this asymptotic rate of the cost as $\varepsilon\to 0$. However, this strong rate is important to design a practical allocation across the $R$ levels, which is the key to avoid an explosion of this term. 
	\item[$\bullet$] When $c_{_R}=0$, the same reasoning can be carried out by considering any small parameter $ \epsilon_0^R>0$. Anyway in practice $c_{_R}$ is usual not known and the impact of this situation   is briefly discussed further on in Section~\ref{sec:Rinfty}.
	\item[$\bullet$] When $c_1=0$, specific weights can be computed (see Practitioner's corner in Section~\ref{sec:practi}).
	\end{itemize}
\end{Remark}

\begin{Remark} The asymptotic number $N$ of simulations given by~\eqref{sol_avec_biais} satisfies 
	\begin{equation*}
		N(\varepsilon) \sim \pa{1 + \frac{1}{2 \alpha R}} \frac{\var(Y_0)}{\varepsilon^2}  \pa{\sum_{j=1}^R q^*_j \sum_{i=1}^R n_i \ind{\mT_i^j \neq 0}}^{-1} \quad \text{ as } \varepsilon \to 0
	\end{equation*}
	for an \MLRR estimator and 
	\begin{equation*}
		N(\varepsilon) \sim \pa{1 + \frac{1}{2 \alpha}} \frac{\var(Y_0)}{\varepsilon^2} \pa{\sum_{j=1}^R q^*_j \sum_{i=1}^R n_i \ind{\mT_i^j \neq 0}}^{-1} \quad \text{ as } \varepsilon \to 0
	\end{equation*}
	for an \MLMC estimator.
\end{Remark}

\subsubsection{Templates for  the design matrix $\mT$} \label{sec:templates}
We now specify the design matrices $\mT$ in both multilevel settings \MLMC defined in~\eqref{eq:desMLMC} and \MLRR defined in~\eqref{eq:desML2R}. 
\begin{description}
	\item[\MLMC estimator] The standard Multilevel Monte Carlo design matrix used by~\cite{Hein, GILES} is derived from the telescopic summation 
\begin{equation*}
	\esp[2]{Y_{\frac{h}{n_{_R}}}} = \esp[1]{Y_h} + \sum_{j = 2}^R \esp[2]{Y_{\frac{h}{n_j}} - Y_{\frac{h}{n_{j-1}}}}.
\end{equation*}
This telescopic sum corresponds  to the design matrix $\mT$ defined by $\mT^j = e_j - e_{j-1}$, $j=2,\dots,R$ \ie
\begin{equation} \label{tempMLMC} \tag{MLMC}
	\mT = \pa{\begin{array}{cccccc}
			1 & -1 & 0 & \cdots & \cdots & 0 \\
			0 & 1 & -1 & 0 & \cdots & 0 \\
			\vdots & \ddots & \ddots & \ddots &\ddots & \vdots\\
			\vdots & & \ddots & \ddots & \ddots & 0\\
			0 & \cdots & \cdots & 0 & 1 & -1 \\
			0 & \cdots & \cdots & \cdots & 0 & 1 \\
	\end{array}}.
\end{equation}
In that case, the resulting upper-bound $\bar \upphi^*$ of $\upphi^*$ writes, with the convention $n_0=( n_{0})^{-1}=0$, 
\begin{equation}\label{eq:MLMC}
	\bar \upphi^*(\uppi_0) = \frac{\var(Y_0)}{h} \pa{1 + \theta h^{\frac{\beta}{2}} \sum_{j=1}^R \pa{ n_{j-1}^{-\frac{\beta}{2}} +n_j^{-\frac{\beta}{2}}} \sqrt{n_{j-1} +n_j}}^2
\end{equation}
With this design matrix~\eqref{tempMLMC} the \MLMC estimator writes 
\begin{equation}\label{eq:EstimMLMC}
	\bar Y^{N,q}_{h, \underline n} = \frac{1}{N_1} \sum_{k=1}^{N_1} Y_{h}^{(1),k} + \sum_{j=2}^R \frac{1}{N_j} \sum_{k=1}^{N_j} \pa{Y_{\frac{h}{n_j}}^{(j),k} - Y_{\frac{h}{n_{j-1}}}^{(j),k}}
\end{equation}
with $N_j = \lceil q_j N \rceil$.
	\item[\MLRR estimator]
The natural counterpart for the design matrix $\mT$ in the \MLRR setting appears as  $\mT^j = -\W_j e_{j-1} + \W_j e_j$, $j=2,\dots,R$ with $\displaystyle \W_j = \sum_{k = j}^R \w_k$ and  $\w$  given by~\eqref{eq:alpha}~\ie
\begin{equation} \label{tempMLRR} \tag{ML2R}
	\mT = \pa{\begin{array}{cccccc }
			1 & -\W_2 & 0 & \cdots & \cdots & 0 \\
			0 & \W_2 & -\W_3 & 0 & \cdots & 0 \\
			\vdots & \ddots & \ddots & \ddots &\ddots & \vdots\\
			\vdots & & \ddots & \ddots & \ddots & 0\\
			0 & \cdots & \cdots & 0 & \W_{R-1} & -\W_R \\
			0 & \cdots & \cdots & \cdots & 0 & \W_R \\
	\end{array}}.
\end{equation}
The resulting upper-bound $\bar \upphi^*$   reads (still with the convention $n_0=(n_{0})^{-1}=0$),
\begin{equation}\label{eq:upphiMLRR}
	\bar \upphi^*(\uppi_0) = \frac{\var(Y_0)}{h}\pa{1 +\theta h^{\frac{\beta}{2}}\sum_{j=1}^R \abs{\W_j} \pa{ n_{j-1}^{-\frac{\beta}{2}} + n_j^{-\frac{\beta}{2}}} \sqrt{n_{j-1} +n_j}}^2.
\end{equation}
In the sequel, we will focus on the above choice~\eqref{tempMLRR} for the design matrix $\mT$ which leads to the \MLRR estimator~\eqref{intro:MLRR-estimator} proposed in the introduction. With this design matrix~\eqref{tempMLRR} the \MLRR estimator writes as a weighted version of \MLMC
\begin{equation}\label{eq:EstimML2R}
	\bar Y^{N,q}_{h, \underline n} = \frac{1}{N_1} \sum_{k=1}^{N_1} Y_{h}^{(1),k} + \sum_{j=2}^R \frac{\W_j}{N_j} \sum_{k=1}^{N_j} \pa{Y_{\frac{h}{n_j}}^{(j),k} - Y_{\frac{h}{n_{j-1}}}^{(j),k}}
\end{equation}
where $N_j = \lceil q_j N \rceil$.
Alternative choices for $\mT$ are proposed  in Section~\ref{sec:practi}.
\end{description}

\subsubsection{Global optimization with varying depth $R$ for geometric refiners}\label{sec:Rinfty}
In this final sub-step, we consider geometric refiners with \emph{root} $M \ge 2$ of the form 
$$
n_i = M^{i-1},\; i=1,\ldots,R.
$$ 
and we only analyze  the \MLRR  and \MLMC estimators  defined by~\eqref{eq:EstimML2R} and~\eqref{eq:EstimMLMC} respectively.  Note that geometric refiners have  already been considered in regular multilevel Monte Carlo framework in~\cite{GILES}.

\begin{Theorem}\label{prop:asympReps} Assume
~\eqref{strong_error} holds for $\beta>0$. 
\begin{enumerate}
	\item[$(a)$]{\em \MLRR estimator:} Assume~$(WE_{\alpha,\infty})$, $\ds \sup_{R\in \N} \sup_{h'\in(0,h)}|\eta_{_R}(h')|<+\infty$ for every $h \!\in \Hr$ and $\displaystyle \lim_{R\to +\infty} |c_{_R}|^{\frac 1R} = \widetilde c_{_\infty} \in(0,+\infty)$.
	The \MLRR estimator~\eqref{eq:EstimML2R} with design matrix $\mT$ in~\eqref{tempMLRR} satisfies
	\begin{equation} \label{res:MLRR}
			\limsup_{\varepsilon \to 0} v(\beta, \varepsilon)
			\times \inf_{\substack{h \in \Hr, R \ge 2 \\ \,|\!\bias(h, R, q^*)| < \varepsilon}} \Cost \pa{\bar Y_{h, \underline n}^{N,q}}
			\le K_{_{\rm ML2R}}(\alpha, \beta, M)
		\end{equation}
		with $\ds v(\beta, \varepsilon) = \begin{cases}
				\varepsilon^2 & \text{if $\beta > 1$,} \\
				\varepsilon^2 \pa{\log(1/\varepsilon)}^{-1} & \text{if $\beta = 1$,} \\
				\varepsilon^2 e^{-\frac{1-\beta}{\sqrt{\alpha}} \sqrt{2\log(1/\varepsilon)\log(M)}} & \text{if $\beta < 1$.} \\
			\end{cases}$
            
            When $\beta < 1$, the best rate achieved with $M = 2$. These rates are achieved with a depth 
		\begin{equation*}
			R^*(\varepsilon)= \left\lceil\frac 12 +\frac{\log\pa[1]{\tilde c^{\frac{1}{\alpha}} \mathbf{h}}}{\log(M)} + \sqrt{\pa[3]{\frac 12+\frac{\log\pa[1]{\tilde c^{\frac{1}{\alpha}} \mathbf{h}}}{\log(M)} }^2 + 2 \frac{\log(A/\varepsilon)}{\alpha \log(M)}} \ \right\rceil, \quad A = \sqrt{1+4 \alpha} ,
		\end{equation*}
		with $\widetilde c >\widetilde c_{_\infty}$    satisfying $\lim_{\varepsilon \to 0} R^*(\varepsilon) = +\infty$ and a bias parameter 
		$h^*={\mathbf h} /\lceil {\mathbf h}/h^*(\varepsilon, R(\varepsilon))\rceil $ where $h^*(\varepsilon, R(\varepsilon))$  
		  is given by~\eqref{eq:hMLRR}.
		The finite real constant $K_{_{\rm ML2R}}(\alpha, \beta, M)$ depends on $M$ and on the structural parameters $\alpha, \beta, V_1, \var(Y_0), {\mathbf h}$, namely 
		\begin{equation} \label{eq:K_MLRR}
			K_{_{\rm ML2R}}(\alpha, \beta, M) =
			\begin{cases}
                \frac{\var(Y_0) M}{{\mathbf h}} \pa{1 + \theta\, {\mathbf h}^{\frac{\beta}{2}} \frac{\W_\alpha(M) M^{\frac{\beta-1}{2}} \sqrt{1+M} (1+M^{-\frac \beta 2})}{1-M^{\frac{1-\beta}{2}}}}^2 & \text{if $\beta > 1$,} \\
                \frac{2 V_1}{\alpha} \pa{\frac{\W_\alpha(M) M (1+M) (1+M^{-\frac{1}{2}})^2}{\log(M)}} & \text{if $\beta = 1$,} \\
				V_1 {\mathbf h}^{1-\beta}\, \widetilde c^{\frac{(1-\beta)}{\alpha}}
                \pa{\frac{\W^2_\alpha(M) M (1+M)(1+M^{- \frac \beta 2})^2}{(M^{\frac{1-\beta}{2}}-1)^2}} & \text{if $\beta < 1$,}
			\end{cases}
		\end{equation}
        where $\W_\alpha(M) = \frac{M^{-\alpha}}{\pi^2_{\alpha, M}} \sum_{k \ge 0} M^{-\alpha \frac{k(k+3)}{2}} + \frac{1}{\pi_{\alpha, M}}$ with  $\pi_{\alpha, M} = \prod_{k \ge 1} (1-M^{-\alpha k})$.

	\item[$(b)$]{\em \MLMC estimator:}  Assume~$(WE_{\alpha,1})$ and $c_1 \neq 0$. The \MLMC estimator~\eqref{eq:EstimMLMC}  (with design matrix $\mT$ defined in~\eqref{tempMLMC}) satisfies
		\begin{equation} \label{res:MLMC}
			\limsup_{\varepsilon \to 0} v(\beta, \varepsilon)
			\times \inf_{\substack{h \in \Hr, R \ge 2 \\ \,|\!\bias(h, R, q^*)| < \varepsilon}} \Cost \pa{\bar Y_{h, \underline n}^{N,q}}
			\le K_{_{\rm MLMC}}(\alpha, \beta, M)
		\end{equation}
		with $\ds v(\beta, \varepsilon) = \begin{cases}
				\varepsilon^2 & \text{if $\beta > 1$,} \\
				\varepsilon^2 \pa{\log(1/\varepsilon)}^{-2} & \text{if $\beta = 1$,} \\
                \varepsilon^{2 + \frac{1-\beta}{\alpha}} & \text{if $\beta < 1$.} \\
			\end{cases}$
		
		These rates are achieved with a depth
		\begin{equation*}
			R^*(\varepsilon)= \left\lceil 1 + \frac{\log\pa[1]{\abs{c_1}^{\frac{1}{\alpha}} \mathbf{h}}}{\log(M)} + \frac{\log(A / \varepsilon)}{\alpha \log(M)} \right\rceil, \quad A = \sqrt{1+2 \alpha}  
		\end{equation*}
		satisfying $\lim_{\varepsilon \to 0} R^*(\varepsilon) = +\infty$ and a bias parameter $h^*={\mathbf h} /\lceil {\mathbf h}/h^*(\varepsilon, R(\varepsilon))\rceil $ where $h^*(\varepsilon, R(\varepsilon))$ is  given by~\eqref{eq:hMLMC}.
		The finite real constant $K_{_{\rm MLMC}}(\alpha, \beta, M)$ depends on $M$ and the structural parameters $\alpha, \beta, V_1, \var(Y_0), {\mathbf h}$, namely 
		\begin{equation*}
			K_{_{\rm MLMC}}(\alpha, \beta, M) =
			\begin{cases}
				\pa{1+\frac{1}{2 \alpha}} \frac{\var(Y_0) M}{{\mathbf h}} \pa{1 + \theta\, {\mathbf h}^{\frac{\beta}{2}} \frac{M^{\frac{\beta-1}{2}} \sqrt{1+M} (1+M^{-\frac \beta 2})}{1-M^{\frac{1-\beta}{2}}}}^2
                & \text{if $\beta > 1$,} \\
				\pa{1+\frac{1}{2 \alpha}} \frac{V_1}{\alpha^2} \pa{\frac{M (1+M) (1+M^{-\frac{1}{2}})^2}{\log(M)^2}}
 & \text{if $\beta = 1$,} \\
				\frac{(1+2 \alpha)^{1 + \frac{1-\beta}{2 \alpha}}}{2 \alpha} V_1 {\mathbf h}^{1-\beta} |c_1|^{\frac{(1-\beta)}{\alpha}}
				\pa{\frac{M (1+M)(1+M^{- \frac \beta 2})^2}{(M^{\frac{1-\beta}{2}}-1)^2}} & \text{if $\beta < 1$.}
			\end{cases}
		\end{equation*}
\end{enumerate}
\end{Theorem}

\paragraph{Comments.} 
Claim $(b)$ is essentially established in Giles' complexity Theorem from~\cite{GILES}.
\begin{itemize}
    \item When $\beta < 1$, \MLRR (with $M=2$) is asymptotically more efficient than \MLMC by a factor $\varepsilon^{-\frac{1-\beta}{\sqrt{\alpha}}} e^{-\frac{1-\beta}{\alpha} \sqrt{2\log(M) \log(1/\varepsilon)}}$ which goes to $+\infty $  as $\varepsilon \to 0$ in a very steep way. To be precise the ratio is greater than $1$ as soon as 
\[
\varepsilon \le 2^{-\frac{2}{\alpha}}.
\] 
It is clear that it is for this setting that \MLRR is the most powerful compared to regular \MLMC.

    \item When $\beta = 1$, \MLRR is asymptotically more efficient than \MLMC by a factor $\log \big(1/\varepsilon)\to +\infty $  as $\varepsilon \to 0$.

    \item When $\beta > 1$, both estimators achieve the same rate $\varepsilon^{-2}$ as a virtual unbiased Monte Carlo method based on the direct simulation of $Y_0$. Some numerical experiments carried out with the call in Black-Scholes model discretized by a Milstein scheme strongly suggest that the constant of the \MLRR estimator is significant lower than the \MLMC one.  
\end{itemize}

\begin{Remark} 
	\begin{itemize}
		\item[$\bullet$] It is proved in Appendix~B that $\displaystyle \lim_{M\to+\infty}  \W_{\alpha}(M)=1$ and, to be more precise, that $\W_{\alpha}(M)-1 \sim M^{-\alpha}$ as $M \to +\infty$.

		\item[$\bullet$] The assumption on the functions $\eta_{_R}$ and  the sequence ${(c_{_R})}_{R\ge 2}$ in~$(a)$  of the above proposition are reasonable, though almost  impossible to check in practice. In particular, note that as soon as the sequence ${(c_{_R})}_{R\ge 2}$ has at most a polynomial growth as a function of  $R$, it satisfies the assumption since $\tilde c =1$.
		
		\item[$\bullet$]  When $\tilde c_{_\infty}=0$, the constant $K(\alpha,\beta, M)$ is equal to $0$   which emphasizes that we are not in the right asymptotic.  In practice $\tilde c_{_\infty}$ is   replaced $c_{_R}$ in this constant by the parameter $\tilde c>0$ used to define the depth.
    \end{itemize}
\end{Remark}

\noindent {\bf Proof of Theorem~\ref{prop:asympReps}}. We provide a detailed poof of claim $(a)$, that of $(b)$ following the same lines. 

\noindent {\sc Step~1}: We start from Equation~\eqref{eq:upperCost} in the proof of  Proposition~\ref{thm:phaseII} which reads 
\begin{equation*}
	\inf_{\substack{h \in \Hr \\ \,|\!\bias(h, q^*)| < \varepsilon}} \Cost \pa{\bar Y^{N, q^*}_{h, \underline n}}  \le \pa{1 + \frac{1}{2 \alpha R}} \frac{\bar \upphi^*(h^*(\varepsilon))}{\varepsilon^2} \frac{1}{1-\frac{(\eta_{R,\underline n}(h^*(\varepsilon))+1)^2-1}{2\alpha R}}
\end{equation*}
with 
\begin{equation*}
	\bar \upphi^*(h^*(\varepsilon)) = \frac{ \var(Y_0) }{h^*(\varepsilon)}\pa{1 + \theta h^*(\varepsilon)^{\frac \beta 2} \sum_{j=1}^R \abs{\W_j} \pa{n_{j-1}^{-\frac \beta 2} + n_j^{- \frac \beta 2}} \sqrt{n_{j-1} + n_j}}^2
\end{equation*}
(convention $n_0 = (n_0)^{-1} = 0$).
The idea  is to choose $R = R^*(\varepsilon)$ as large as possible provided the optimal bias parameter $h^*$ lies in $\Hr$. The form of the refiners $n_i = M^{i-1}$ implies that $\underline n != M^{\frac{R(R-1)}{2}}$ so that 
\begin{equation*}
	h^*(\varepsilon, R) = (1 + 2\alpha R)^{-\frac{1}{2 \alpha R}} |c_{_R}|^{-\frac{1}{\alpha R}} \varepsilon^{\frac{1}{\alpha R}} M^{\frac{R-1}{2}}.
\end{equation*} 
To determine the dependence of $R$ with respect to $\varepsilon$, we consider the auxiliary function 
\begin{equation*}
\tilde	h(\varepsilon, R) = (1+4 \alpha)^{-\frac{1}{2 \alpha R}} \widetilde c^{\,-\frac{1}{\alpha}} \varepsilon^{\frac{1}{\alpha R}}M^{\frac{R-1}{2}},
\end{equation*}
and let $P$ be the polynomial function
\begin{equation*}
	P(R) = \frac {R(R-1)}{2}\log(M) - R \log(K) - \frac{1}{\alpha} \log\pa{\sqrt{1+4 \alpha} / \varepsilon},
\end{equation*}
such that  $\tilde h(\varepsilon, R) = {\mathbf h}\, e^{\frac{P(R)}{R}}$, where $K = \widetilde c^{\frac{1}{\alpha}} \mathbf{h}$. Note that the polynomial function $P$ has a unique positive root $R_+(\varepsilon)$ given by 
\[
R_+(\varepsilon) =  \frac 12 +\frac{\log(K)}{\log(M)} + \sqrt{\Big(\frac 12+\frac{\log(K)}{\log(M)}\Big)^2+2\,\frac{\log\pa{\sqrt{1+4\alpha}/\varepsilon}}{\alpha \log M}}
\]
so that $\tilde	h(\varepsilon, R) = {\mathbf h}$. We then consider $R^*(\varepsilon)= \lceil R_+(\varepsilon)\rceil$ and define $h^*(\varepsilon)$ as the projection of $  h^*(\varepsilon, R^*(\varepsilon))$ on $\Hr$ so that $h^*(\varepsilon)\le {\mathbf h}$ and $h^*(\varepsilon)$ si equal to ${\mathbf h}$ for small enough $\varepsilon$. 

Let us show that our choice $h^*(\varepsilon)$ is admissible, i.e. $\bias(\varepsilon) = \bias(h^*(\varepsilon), R^*(\varepsilon), q^*)$ satisfies $\bias(\varepsilon)^2<\varepsilon^2$ at least for small enough $\varepsilon$.  Elementary computations show that 
\begin{align*}
\bias(\varepsilon)^2&= \Big(c_{_{R^*(\varepsilon)}}M^{\frac{R^*(\varepsilon)(R^*(\varepsilon)-1)}{2} \alpha} h^*(\varepsilon )^{\alpha R^*(\varepsilon)} \Big)^2\Big(1+\eta_{R^*(\varepsilon),n}\big(h^*(\varepsilon)\big)\Big)^2\\
&=  (1+4\alpha)^{-1}\varepsilon^2 e^{-2\alpha P(R^*(\varepsilon))}\Big(\frac{c_{R^*(\varepsilon)}}{\tilde c_{R^*(\varepsilon)}}\Big)^2\Big(1+\eta_{R^*(\varepsilon),n}\big(h^*(\varepsilon)\big)\Big)^2.
\end{align*}
First note that we have $\lim_{R\to +\infty} |c_R|^{\frac 1R} = \tilde c_{_\infty}$ and $\tilde c > \tilde c_{_\infty}$. Moreover,  Claim~6 of  Proposition~\ref{asympalpha:2}  in  Appendix~A  and the assumption on $\eta_{R}$ imply  that 
$$
\sup_{0<h'<{\mathbf h}}|\eta_{R^*(\varepsilon),\underline n}(h')|\le B_{\alpha}(M) \sup_{h'\in (0,{\mathbf h})}|\eta_{R^*(\varepsilon)}(h')|\le  B_{\alpha}(M) \sup_{R\ge 1}\sup_{h'\in (0,{\mathbf h})}|\eta_{R}(h')|<+\infty.
$$
As a consequence of the assumption made on the functions $\eta_{_R}$, it is clear that $\bias(\varepsilon)^2= o(\varepsilon^2)$ since $R^*(\varepsilon) \to +\infty$ as $\varepsilon \to 0$. Hence, our choice for the bias parameter is admissible, at least for small enough $\varepsilon$. 

Likewise, the assumption on the  functions $\eta_R$  implies $\lim_{\varepsilon \to 0} \frac{\big(\eta_{R^*(\varepsilon),\underline n}(h^*(\varepsilon)+1\big)^2-1}{2\alpha R^*(\varepsilon)}=0$. 

We have  then  proved that
\begin{equation*}
	\limsup_{\varepsilon \to 0} \pa{l(\varepsilon, R^*(\varepsilon)) \inf_{\substack{h \in \Hr \\ \,|\!\bias(h, R, q^*)| < \varepsilon}} \times \Cost \pa{\bar Y^{N, q^*}_{h, \underline n}}} \le \frac{M \var(Y_0)}{{\mathbf h}}
\end{equation*}
with 
\begin{equation*}
	l(\varepsilon, R) = \varepsilon^2 \pa{1 + \theta \,h^*(\varepsilon, R)^{\frac{\beta}{2}} \sum_{j=1}^R \abs{\W_j} \pa{n_{j-1}^{-\frac{\beta}{2}} + n_j^{-\frac{\beta}{2}}} \sqrt{n_{j-1} +n_j}}^{-2}.
\end{equation*}
It follows from Claim~5 of  Proposition~\ref{asympalpha:2} in  Appendix~A  that $\max_{j=1,\dots,R} \abs{\W_i} \le \W_{\alpha}(M)$. On the other hand, standard computations show that, for every $j=2,\ldots,R$, 
\begin{equation}\label{eq:ni}
  \pa{n_{j-1}^{-\frac{\beta}{2}} + n_j^{-\frac{\beta}{2}}} \sqrt{n_{j-1} +n_j}= M^{\beta-1}M^{j\frac{1-\beta}{2}}\big(1+M^{-\frac{\beta}{2}}\big)\big(1+M\big)^{\frac 12}.
\end{equation} 
Moreover, with our convention on $n_0$, it still  holds true as an inequality ($\le$) for $j=1$. So 
\begin{equation*}
	l(\varepsilon, R) \ge \varepsilon^2 \pa{1+\theta h^*(\varepsilon, R)^{\frac \beta 2} \W_\alpha(M) M^{\beta-1} \sqrt{1+M} (1+M^{-\frac\beta 2}) \sum_{j=1}^R M^{j \frac{1-\beta}{2}}}^{-2}.
\end{equation*}

\noindent {\sc Step~2}: Now we will inspect successively the three cases depending on the strong rate convergence parameter $\beta > 0$.
\begin{description}
	\item[Case $\beta = 1$.] In that case,
		\begin{align*}
			l(\varepsilon, R^*(\varepsilon)) &\ge \varepsilon^2 \pa{1+\theta h^*(\varepsilon)^{\frac \beta 2} \W_\alpha(M) \sqrt{1+M} (1+M^{-\frac 1 2}) R^*(\varepsilon)}^{-2}, \\ 
			& \ge \varepsilon^2 \pa{1+\theta\, {\mathbf h}^{\frac \beta 2} \W_\alpha(M) \sqrt{1+M} (1+M^{-\frac 1 2}) R_+(\varepsilon)}^{-2}, 
		\end{align*}
		and, as $R^*(\varepsilon)^2\sim R_+^2(\varepsilon) \sim \frac{2}{\alpha \log(M)} \log(1/\varepsilon)$ as $\varepsilon \to 0$, we get~\eqref{res:MLRR} with $K_{_{\rm ML2R}}(\alpha, 1, M)$ given by~\eqref{eq:K_MLRR} keeping in mind  that $V_1 = \var(Y_0) \theta^2$.

	\item[Case $\beta > 1$.] Noting that $\sum_{j=1}^R M^{j \frac{1-\beta}{2}} \le \frac{M^{\frac{1-\beta}{2}}}{1-M^{\frac{1-\beta}{2}}}$, we get 
		\begin{equation*}
			l(\varepsilon, R^*(\varepsilon)) \ge \varepsilon^2 \pa{1+\theta\, {\mathbf h}^{\frac \beta 2} \frac{\W_\alpha(M) M^{\frac{\beta-1}{2}} \sqrt{1+M} (1+M^{-\frac \beta 2})}{1-M^{\frac{1-\beta}{2}}}}^{-2},
		\end{equation*}
		which yields~\eqref{res:MLRR} with $K_{_{\rm ML2R}}(\alpha, \beta, M)$ given by~\eqref{eq:K_MLRR}.  

	\item[Case $\beta < 1$.] In that setting, we note this time that $\sum_{j= 1}^R M^{j\frac{1-\beta}{2}}\le \frac{M^{(R+1)\frac{1-\beta}{2}}}{M^{\frac{1-\beta}{2}}-1}$ so that 
		\begin{equation*}
			l(\varepsilon, R^*(\varepsilon)) \ge \varepsilon^2 \pa{1+\theta {\mathbf h}^{\frac \beta 2} \frac{\W_\alpha(M) \sqrt{1+M} (1+M^{-\frac \beta 2})}{M^{\frac{1-\beta}{2}}-1} M^{(R_+(\varepsilon)-1) \frac{1-\beta}{2}}}^{-2}.
		\end{equation*}
	As $R_+(\varepsilon)$ satisfies $\tilde h(\varepsilon, R_+(\varepsilon)) = {\mathbf h}$, we obtain $M^{\frac{R_+(\varepsilon)-1}{2}} = (1+4\alpha)^{\frac{1}{2\alpha R_+(\varepsilon)}}{\mathbf h} \,\widetilde c^{\frac{1}{\alpha}} \varepsilon^{-\frac{1}{\alpha R_+(\varepsilon)}}$. We have $\varepsilon^{-\frac{1}{\alpha R_+(\varepsilon)}}\sim e^{\sqrt{\frac{\log(M)}{2\alpha} \log(1/\varepsilon)}}$ as $\varepsilon\to 0$.
	Elementary, though tedious, computations yield~\eqref{res:MLRR} with $K_{_{\rm ML2R}}(\alpha, \beta, M)$ given by~\eqref{eq:K_MLRR}.  
\end{description}

\noindent $(b)$ The choice for $R^*(\varepsilon)$ follows by considering the auxiliary function 
$$
\tilde h(\varepsilon, R) = (1 + 2\alpha)^{-\frac{1}{2\alpha}} |c_1|^{-\frac{1}{\alpha}} \varepsilon^{\frac 1 \alpha} M^{R-1}.% = {\mathbf h}.
$$
Then, the proof follows the same lines as that of $(a)$.
%\end{proof}

\begin{Remark}[On the constraint $\mathbf h$] In the proof we chose to saturate the constraint $h^* \le \mathbf{h}$. If we consider $h^* = \chi$ where $\chi$ is a free parameter in $(0, \mathbf{h}]$, then the asymptotic constants $K(\alpha,\beta,M)$ for the renormalized optimized cost in Theorem~\ref{prop:asympReps} depends on $\chi$ and one verifies the following facts:

\begin{itemize}

\item When $\beta <1$, one can write $K_{_{\rm ML2R}}(\alpha,\beta,M, \chi)=  {\chi}^{1-\beta}   K_{_{\rm ML2R}}(\alpha,\beta,M,1)$ which this time suggests to start the simulation with a small upper bias parameter $\chi < {\mathbf h}$.

\item When $\beta=1$, the asymptotic constant $K_{_{\rm ML2R}}(\alpha,1,M, \chi)$ {\em does not depend} on $\chi$. This suggests that the choice of the upper bias parameter is not decisive, at least for high accuracy computations ($\varepsilon$ close to $0$). The choice $\chi = \mathbf{h}$ remains  the most natural.

\item When $\beta >1$, the asymptotic cost of the simulation increases in $\varepsilon^2$ like  a (virtual)  unbiased one. In that very case, it appears that the asymptotic constant  $K_{_{\rm ML2R}}(\alpha,\beta,M, \chi)$ can itself be optimized as a function of  $\chi$. Namely, if we set 
\[
\kappa_1= \frac{\var(Y_0) M}{\chi} \quad \mbox{ and }\quad \kappa_2=\theta^2   \frac{\W_\alpha(M)^2 M^{\beta-1} (1+M)(1+M^{-\beta })}{(1-M^{\frac{1-\beta}{2}})^2},
\]
then
\[
{\chi}_{opt}=  \beta^{-\frac{2}{\beta+1}} \kappa_2^{-\frac{1}{\beta+1}} \quad \mbox{ and }\quad K_{_{\rm ML2R}}(\alpha,\beta,M, {\chi}_{opt})=(\beta+1)^2\beta^{-\frac{2}{\beta+1}} \,\kappa_1\, \kappa_2^{\frac{1}{\beta+1}}.
\]
\end{itemize}
\end{Remark}

\section{Examples of applications}\label{sec:appli}
\subsection{Brownian diffusion approximation}\label{sec:brown}

\paragraph{Euler scheme} In fact, the (one-step) Richardson-Romberg extrapolation is well-known as an efficient mean to reduce the time discretization error induced by the use of an Euler scheme to simulate a Brownian diffusion. In this field of Numerical Probability, its introduction goes back to Talay and Tubaro in their seminal paper~\cite{TATU} on weak error expansion, followed by the case of non-smooth functions in~\cite{BATA}, under an H\"ormander hypo-ellipticity assumption. 

It relies on the following theorem.
\begin{Theorem}\label{theoromberg1}
	Let $b: \R^d \to \R^d$, $\sigma: \R^d \to \calig{M}(d,q)$ and let $(W_t)_{t \ge 0}$ be a $q$-dimensional standard Brownian motion defined on a probability space $(\Omega, \calig{A}, \P)$. Let $X=(X_t)_{t\in [0,T]}$ be a diffusion process, strong solution to the Stochastic Differential Equation ($SDE$)
\begin{equation} \label{SDE}
	\d X_t = b(X_t) \d t + \sigma(X_t) \d W_t, \; t\in\left[0,T\right],\; X_0 = x_0 \in \R^d,
\end{equation}
and its continuous Euler scheme $\bar X^h= (\bar X^h_t)_{t\in [0,T]}$ with bias (step) parameter $h = T/n$ defined by
\begin{equation*}
	\bar{X}^h_t=X_0+\int_0^tb\big(\bar{X}^h_{\underline{s}}\big)\d s+\int_0^t\sigma\big(\bar{X}^h_{\underline{s}}\big)\d W_s,
	\quad \text{ where } \quad \underline{s}= kh \text{ on } \left[k h, (k+1)h \right),\; k=0,\ldots,n.
\end{equation*}
where 
\[
\underline s = kh \mbox{ on } [kh, (k+1)h),\quad k=0,\ldots,n.
\]
%\smallskip
\noindent  $(a)$ Smooth setting (Talay-Tubaro~\cite{TATU}): If  $b$ and $\sigma$ are  infinitely differentiable with bounded  partial derivatives and if $f:\R^d\rightarrow \R$ is an infinitely differentiable function, whose all partial derivatives have polynomial growth, then, for a fixed $T>0$ and every integer $R \in \N^*$,
\begin{equation}\label{RExpansion}
	\esp{f(\bar{X}^h_T)} - \esp{f(X_T)} = \sum_{k=1}^R c_k h^k + O \pa{h^{R+1}}
\end{equation}
where the coefficients $c_k$ depend on $b$, $\sigma$, $f$, $T$ (but {\em not on $h$}).

\smallskip
\noindent $(b)$ (Hypo-)Elliptic setting (Bally-Talay~\cite{BATA}): If  $b$ and $\sigma$ are infinitely differentiable with bounded  partial derivatives and if $\sigma $ is uniformly elliptic in the sense that 
\[
\forall\, x\!\in \R^d, \quad \sigma\sigma^*(x) \ge \varepsilon_0I_q, \; \varepsilon_0>0
\]
or, more generally, if $(b,\sigma)$ satisfies the strong  H\"ormander hypo-ellipticity assumption, then~\eqref{RExpansion} holds true  for every bounded Borel function $f:\R^d\rightarrow \R$.
\end{Theorem}
Other results based on the direct expansion of the density of the Euler scheme allow to deal with a drift $b$ with linear growth (see~\cite{KOMA}, in a uniformly elliptic setting, see also~\cite{JulGuy} at order $1$ in a tempered distribution framework). It is commonly shared by the ``weak error community",  relying on an analogy with recent results  on the existence of smooth density from the diffusion,  that if the hypo-ellipticity assumption is satisfied except at finitely many points that are never visited by the diffusion, then the claim~$(b)$ remains true. The boundedness assumption on  $\sigma$ is probably more  technical than a mandatory assumption. For a recent  review on weak error, we refer to~\cite{Jourdain}.

To deal with our abstract multilevel framework, we consider for a fixed horizon $T>0$, the family of Euler schemes $\bar X^h$ with  step $h\!\in \Hr=\{\frac Tn,\, n\ge 1\}$. We set $Y_h = f(\bar X^h_T)$ and $Y_0 = f(X_T)$ for a function $f$ either smooth enough with polynomial growth or simply  Borel and bounded,   depending on the smoothness of $b$ and $\sigma$ and the (hypo-)ellipticity of $\sigma$. The above theorem says that condition~\eqref{weak_error} is satisfied with $\bar R = +\infty$ and $\alpha = 1$. However, for a fixed $\bar R$,  the differentiability assumption on $b$, $\sigma$ and $f$  can be relaxed by simply assuming that these three functions are ${\cal C}_b^{\bar R+5}$ on $[0,T]\times \R^d$. 

On the other hand, as soon as $f:\R^d \to \R $ is Lipschitz continuous, it is classical results that~\eqref{strong_error} is satisfied with $\beta = 1$ as an easy consequence of the fact that the (continuous) Euler scheme  $\bar X^h$ converges for the sup-norm toward $X$ in  $\L^2$ (in fact in every  $\L^p$-space) at rate $\sqrt{h}$ as the step $h$ goes to $0$. 

In such a setting, we can implement multilevel estimators with $\alpha= \beta=1$.

\paragraph{Milstein scheme} The Milstein scheme is a second order scheme  which satisfies~\eqref{strong_error} with $\beta = 2$ and~\eqref{weak_error} still with $\alpha = 1$   (like the Euler scheme). Consequently, provided  it can be implemented, the resulting multilevel estimators   should be   designed with these parameters. 

However, the main drawback of the Milstein scheme when the $SDE$ is driven by a multidimensional Brownian motion ($q\ge 2$), is that  it  requires the simulation of L\'evy areas, for which there is no known efficient method (except in dimension 2). In a recent work~\cite{GILESZPRUCH}, Giles and Szpruch introduce  a suitable {\em antithetic multilevel correction estimator} which  avoids the simulation of these L\'{e}vy areas. This approach can be easily combined with our weighted version of \MLMC.

Note that in the $\beta > 1$ case, Rhee and Glynn introduced in \cite{Rhee} a class of finite-variance optimally randomized multilevel estimators which are unbiased with a square root convergence rate.

\paragraph{Path-dependent functionals} When a functional $F: {\cal C}([0,T], \R^d)\to \R$  is Lipschitz continuous for the sup-norm, it is straightforward that $F(\bar X^h)$ and $F(X)$ satisfy~\eqref{strong_error}, with $\beta = 1$ and  $\Hr = \{\frac Tn, \,n\ge 1\}$, (but this is no longer true if one considers the {\em stepwise constant} Euler scheme since the rate of convergence is then $\sqrt{\log n/n}\asymp \sqrt{-h\log h}$). More generally, if $F$ is $\beta$-H\" older, $\beta\!\in (0,1]$, then this family satisfies~\eqref{strong_error}. High order expansions of the weak error are not available in the general case, however first order expansion have been established for specific functionals like $F(\w)= f\Big(\int_0^T \w(s)ds\Big)$ or $F(\w)= f(\w(T))\mbox{\bf 1}_{\{\tau_D(\w) >T\}}$ where $\tau_D(\w)$ is the exit time of a domain $D$ of $\R^d$ showing that~\eqref{weak_error} holds with $\alpha = 1$ and $\bar R = 1$ (see $e.g.$~\cite{LATE, GOB}). 
More recently, new results on first order weak error expansions have been obtained for functionals of the form $F(\w)= f\pa[1]{\w(T), \sup_{t\in[0,T]}\w(t)}$ (see~\cite{GILES3} and~\cite{AJKH}). Thus,  for the weak error expansion, it is shown in~\cite{AJKH} that, for every $\eta>0$, there exists a real constant $C_{\eta}>0$ such that
\[
	\abs{\esp[1]{f\big(X_{_T}, \sup_{t\in [0,T]} X_t\bigr)} - \esp[1]{f\big(\bar X^n_{_T}, \sup_{t\in [0,T]} \bar X^n_t\big)}} \le \frac{C_{\eta}}{N^{\frac{2}{3}-\eta}}.
\]
For a review of recent results on approximation of solutions of SDEs, we again refer to~\cite{Jourdain}.

\begin{Remark} Note that, as concerns the \MLMC estimator, in the general setting of the discretization of a Brownian diffusion by an Euler scheme, a Central Limit Theorem (with stable weak convergence) has been obtained in~\cite{KEBAIERBENA}. In fact  both     the \MLRR and   \MLMC estimators attached to the design matrices~\eqref{tempMLRR}  and~\eqref{tempMLMC} satisfy, under a sharp version of~\eqref{strong_error}, a    Central Limit Theorem (see~\cite{Giorgi}) as $\varepsilon\to 0$. In the case of \MLRR it requires an in-depth analysis of the asymptotic behaviour of the weight vector $(\W_i)= (\W^{\alpha,R}_i)_{1\le i\le R}$ as $R$ goes to $\infty$. 
\end{Remark}

\subsection{Nested Monte Carlo} \label{sec:nested}
The purpose of the so-called {\em nested} Monte Carlo method is to compute by simulation quantities of the form
\begin{equation*}
        \esp{f \pa[1]{\espc{X}{Y}} }
\end{equation*}
where $(X,Y)$ is a couple of $\R\times \R^{q_{_Y}}$-valued random  variable defined on a probability space $(\Omega,{\cal A}, \P)$ with  $X\!\in \L^2(\P)$ and $f:\R\to \R$ is a Lipschitz continuous function  with Lipschitz coefficient $[f]_{\rm Lip}$. Such quantities often appear in financial applications, like compound option pricing or risk  estimation (see~\cite{BROADIE}) and in actuarial sciences  (see~\cite{LOI}) where nested Monte Carlo is widely implemented. The idea of replacing conditional expectations by Monte Carlo estimates also appears in~\cite{Belo} where the authors derive a multilevel dual Monte Carlo algorithm for pricing American style derivatives.

We make the following more stringent assumption: there exists a Borel function
$F:\R^{q_{_Z}}\times \R^{q_{_Y}} \to \R$ and a  
random variable $Z: (\Omega,{\cal A})\to \R^{q_{_Z}}$ independent of $Y$ such that
\[
X= F(Z,Y).
\]
Then, if $X\!\in \L^2$, one has the following representation
\[
        \espc{X}{Y}(\omega)= \Big( \esp{F(Z,y)} \Big)_{|y=Y(\omega)} =  
\int_{\R^{q_{_Z}}} F(z,Y(\omega)) \P_{Z}(dz).
\]
To comply with the multilevel framework, we set
\[
        \Hr = \{ 1/K, \, K\ge 1\}, \quad Y_0 = f\pa[1]{\espc{X}{Y}}, \quad   
Y_{\frac 1 K} =f\pa{\frac 1K \sum_{k=1}^K F(Z_k,Y)}
\]
where $(Z_k)_{k\ge 1}$ is an \iid sequence of copies of $Z$ defined on $(\Omega,{\cal A}, \P)$ and independent of $Y$ (up to an enlargement of the probability space if necessary).

The following proposition shows that the nested Monte Carlo method is eligible for multilevel simulation when $f$ is regular enough with the same parameters as the Euler scheme for Brownian diffusions.

\begin{Proposition}  Assume $X\!\in \L^{2 R}$. If $f$ is Lipschitz continuous and $2 R$ times  differentiable with $f^{(k)}$ bounded, $k=R,\ldots,2R$, the nested  Monte Carlo satisfies~\eqref{strong_error} with $\beta = 1$ and ~\eqref{weak_error} with $\alpha =1$ and $\bar R = R-1$.  
\end{Proposition}

\begin{Remark} When $f$ is no longer smooth, typically if it is the indicator function of an interval, it is still possible to show that nested Monte Carlo is eligible for multilevel Richardson-Romberg approach $e.g.$ in  the more constrained  framework developed in~\cite{JHJ, GORJU} where $X$ can be viewed as  an additive perturbation of $Y$.  Assuming enough regularity in $y$ on the joint density $g_N(y,z)$ of $Y$ and the renormalized perturbation,  yields an expansion of the weak error (but seems in a different scale). However, in this work we focus on the regular case (see~\cite{VLGP2} for the non regular case and applications in actuarial sciences).
\end{Remark}

The proof follows from the two lemmas below.

\begin{Lemma}[Strong approximation error]  Assume $f$ is Lipschitz continuous. For every $h,\,h'\!\in \Hr \cup\{0\}$,
\begin{equation} \label{eq:result_strong_nested}
	\normLp{2}{Y_{h'}-Y_h}^2 \le [f]^2_{\rm Lip} \pa{\normLp{2}{X}^2 - \normLp{2}{\espc{X}{Y}}^2} |h'-h|.
\end{equation}
so that $(Y_h)_{h\in \Hr}$ satisfies~\eqref{strong_error} with $\beta=1$ and the alternative assumption \eqref{var_error'} from Remark~\ref{rq:strong_assumption}).
\end{Lemma}

\begin{proof}
    Let $h=\frac 1K$, $h'=\frac{1}{K'}$, $K$, $K'\in \N^*$, $K\le K'$. Now set for convenience $\widetilde X_k = F(Z_k,Y)-\espY{F(Z_k,Y}$, $M_k=  \sum_{\ell=1}^k \widetilde X_k$ and ${\cal G}_{k} = \sigma(Y, Z_1,\ldots,Z_k)$, $k\ge 0$. It is clear that $(M_k)_{k\ge 0}$ is  a square integrable martingale (null at time $0$) satisfying $\espc{(M_k-M_{k-1})^2}{{\cal G}_{k-1}} = \pa{\normLp{2}{X}^2 - \normLp{2}{\espc{X}{Y}}^2}$. Elementary computations yield for every integers $K' \ge K \ge 1$, %K$, $K'\ge 1$, $K\le K'$, 
\begin{align*}
	\normLp{2}{Y_h - Y_{h'}}^2 &= \normLp[4]{2}{f\pa[3]{ \frac{1}{K'} \sum_{k=1}^{K'} F(Z_k, Y)} - f\pa[3]{\frac{1}{K} \sum_{k=1}^K F(Z_k, Y)}   }^2 \\
	&\le [f]^2_{\rm Lip} \normLp[3]{2}{\frac{1}{K'} \sum_{k=1}^{K'} F(Z_k, Y)  - \frac{1}{K} \sum_{k=1}^K F(Z_k, Y)}^2\\
	&=  [f]^2_{\rm Lip} \normLp[3]{2}{\frac{1}{K'} \sum_{k=1}^{K'} \widetilde X_k  - \frac{1}{K} \sum_{k=1}^K \widetilde X_k}^2=  [f]^2_{\rm Lip} \normLp[3]{2}{\frac{M_{K'}}{K'}   - \frac{M_K}{K}}^2
\end{align*}
since $\espY{F(Z_k,Y)} = \espY{F(Z,Y)}$ does not depend on $k$ owing to the independence of $(Z_k)_{k\ge 1}$ and $Y$. Then elementary computations show that 
\[
\normLp[3]{2}{\frac{M_{K'}}{K'}   - \frac{M_K}{K}}^2= \frac{K'-K}{KK'}  \normLp{2}{\tilde X_k}^2= (h-h')\pa{\normLp{2}{X}^2 - \normLp{2}{\espc{X}{Y}}^2}.
\]
The case $h'=0$ can be treated likewise (or by letting $K'$ go to infinity).
\end{proof}

\begin{Lemma}[Weak error]  Let $f:\R\to \R$ be a $2R$ times differentiable function with $f^{(k)}$, $k=R,\ldots,2R$, bounded over the real line. Assume $X\!\in \L^{2R}(\P)$. Then there exists $c_1,\ldots,c_{R-1}$ such that
\begin{equation} \label{eq:result_weak_nested}
	\forall h \in \Hr, \quad \esp{Y_h} = \esp{Y_0} + \sum_{r=1}^{R-1} c_r h^r + O\big(h^R\big).
\end{equation}
Consequently $(Y_h)_{h\in \Hr}$ satisfies~\eqref{weak_error} with $\alpha =1$ and $\bar R = R-1$.
\end{Lemma}

\begin{proof}
Let $K\ge 1$ and $\widetilde X_k = F(Z_k,Y)- \espY{F(Z_k,Y)} =F(Z_k,Y)-Y_0$, $k=1,\ldots,K$. By the multinomial formula, we get
\[
(\widetilde X_1+\cdots+\widetilde X_{_K})^k = \sum_{k_1+\cdots+k_K=k}  
\frac{k!}{k_1!\cdots k_K!}\widetilde X_1^{k_1}\cdots \widetilde  
X_K^{k_K}.
\]
Then, taking conditional expectation given $Y$, yields
\[
	\espY{(\widetilde X_1+\cdots+\widetilde X_{_K})^k} = k!  
	\sum_{k_1+\cdots+k_K=k} \prod_{i=1}^K \frac{\espY[1]{\widetilde X^{k_i}}}{k_i!}
\]
since $\espY[1]{\widetilde X_i^{k_i}} = \espY[1]{\widetilde X^{k_i}}$. As $\espY[1]{\widetilde X_i} = 0$, we obtain 
\[
	\espY{(\widetilde X_1+\cdots+\widetilde X_{_K})^k} = k!  
\sum_{k_1+\cdots+k_K=k, \, k_i\neq 1} \prod_{i=1}^K \frac{\espY[1]{\widetilde X^{k_i}}}{k_i!}.
\]

Let $I=I(k)$ denote the generic set of indices $i$ such that $k_i\neq  0$. It is clear that $1\le |I|\le k/2$. By symmetry, we have now that
\begin{align*}
	\sum_{k_1+\cdots+k_K=k, \, k_i\neq 1} \prod_{i=1}^K \frac{\espY[1]{\widetilde X^{k_i}}}{k_i!}&= \sum_{1\le \ell\le (k/2)\wedge  
K}\sum_{
    \substack{I\subset\{1,\ldots,K\}, |I|=\ell, \\ \sum_{i\in I} k_i =k, k_i  
\ge 2}}  \prod_{i=1}^K \frac{\espY[1]{\widetilde X^{k_i}}}{k_i!}\\
&= \sum_{1\le \ell\le k/2}\Big(\begin{array}{c}K\\  
\ell\end{array}\Big)\sum_{  \sum_{1\le i\le \ell } k_i =k-2\ell}  
\prod_{i=1}^\ell\frac{\espY[1]{\widetilde X^{2+k_i}}}{(2+k_i)!}.
\end{align*}
As a consequence, for every integer $R\ge 1$,
\begin{align*}
	\espY[1]{Y_h} &= \espY[1]{Y_0}  +\sum_{k=1}^{2R-1} \frac{f^{(k)}\big(  
\espY{X}\big)}{k! K^k} \E^Y (\widetilde X_1+\cdots+\widetilde X_{_K})^k  
  + {\mathbf R}_{2R-1}(Y) \\
  &= \espY{Y_0}  +\sum_{k=1}^{2R-1} \frac{ f^{(k)}\big(  
  \espY{X} \big)}{k! K^k}  \sum_{1\le \ell\le (k/2)\wedge  
K}\Big(\begin{array}{c}K\\ \ell\end{array}\Big)c_{k,\ell}+ {\mathbf  
R}_{2R-1}(Y)
\end{align*}
where
\[
a_{k,\ell} = \sum_{k_1+\cdots+k_\ell =k-2\ell}  
\;\prod_{i=1}^\ell\frac{\espY[1]{\widetilde X^{2+k_i}}}{(2+k_i)!}
\]
and
\begin{equation*}
|{\mathbf R}_{2R-1}(Y)| \le \frac{ \|f^{(2R)}\|_{\rm sup}}{(2R)!}  
	\frac{1}{K^{2R}} \espY{ \big| \widetilde X_1+\cdots+\widetilde  
	X_{K}\big|^{2R}}.
\end{equation*}
By the Marcinkiewicz-Zygmund Inequality we get
\begin{align*}
|{\mathbf R}_{2R-1}(Y)|&\le  (B^{MZ}_{2R})^{2R}\frac{  
\|f^{(2R)}\|_{\rm sup}}{(2R)!} \frac{1}{K^{2R}} \espY{| \widetilde  
X^2_1+\cdots+\widetilde X^2_{K}\big|^{R} }            \\
&\le \|f^{(2R)}\|_{\rm sup} \frac{ (B^{MZ}_{2R})^{2R}}{(2R)!}  
\frac{1}{K^{R}} \espY[1]{\widetilde X^{2R}}
\end{align*}
where $B^{MZ}_{p} = 18\frac{p^{\frac 32}}{(p-1)^{\frac 12}}$, $p>1$  
(see~\cite{SHIR} p.499).
Now, we write the polynomial function $x(x-1)\cdots(x-\ell+1)$ on the canonical  
basis $1,x,\ldots,x^n$,\ldots  as follows
\[
x(x-1)\cdots(x-\ell+1)=\sum_{m=0}^{\ell} b_{\ell,m} x^m\qquad  
(b_{\ell,\ell}=1 \mbox{ and } b_{\ell,0}=0).
\]
Hence,
\begin{equation*}
	\espY{Y_h} = \espY{Y_0}  +\sum_{k=1}^{2R-1}\sum_{\ell =1}^{\frac  
	k2}\sum_{m=1}^{\ell} \frac{ f^{(k)}\big( \espY{X} \big)}{k!}  
\frac{1}{K^{k-m}}a_{k,\ell} b_{\ell,m} +O\big(K^{-R}\big)
  \end{equation*}
  where $K^RO(K^{-R})$ is bounded by a deterministic constant. For  
every $r\!\in \{1,\ldots, R-1\}$, set
  \[
  J_{R,r} = \big\{(k,l,m) \!\in \N^3,\; 1\le k\le 2R-1, \,  
1\le \ell\le k/2,\; 1\le m\le \ell, \; k=m+r\big\}
  \]
(note that one always has $k\ge (2 m)\vee 1$ so that $k-m\ge 1$ when  
$k,l,m$ vary in the admissible index set). We finally get
\begin{align*}
	\espY{Y_h} &= \espY{Y_0}  +\sum_{r=1}^{2R-1} \Big(\sum_{(k,\ell,m) \in  
J_{R,r}}   \hskip -0.25 cm\frac{ f^{(k)}\big( \espY{X}\big)}{k!}   
a_{k,\ell} b_{\ell,m}      \Big)\frac{1}{K^r}+O\big(K^{-R}\big).\\
&= \espY{Y_0}  +\sum_{r=1}^{R-1} \frac{c_r}{K^r}+O\big(K^{-R}\big).
  \end{align*}
Taking the expectation in the above equality yields the announced  
result.\end{proof}

\begin{Remark} 
	Though it is not the only term included in the final $O(K^{-R})$, it is worth noticing that $\Big( \frac{  (B^{MZ}_{2R})^{2R} }{ (2R)!}  \Big)^{\frac 1R} \sim  (36  R)^2\Big(\frac{2R}{e}\Big)^{-2}\sim 18\, e^{2}$ as $R\to+\infty$ owing  to Stirling's formula. This suggests that, if all the derivatives of  $f$ are uniformly bounded,  $\displaystyle \limsup_{R\to +\infty}| c_{_R}|^{\frac  1R}<+\infty$.
\end{Remark}

\section{Numerical experiments}\label{sec:NumEx}
\subsection{Practitioner's corner} \label{sec:practi}
We summarize here the study of the Section \ref{sec:MLRR}. We have proved in Theorem~\ref{thm:phaseI}, Proposition~\ref{thm:phaseII} and Theorem~\ref{prop:asympReps} that the asymptotic optimal parameters (as $\varepsilon$ goes to 0) $R$, $h$, $q$ and $N$ depend on structural parameters $\alpha$, $c_1$, $\beta$, $V_1$, $\var(Y_0)$ and ${\mathbf h}$ (recall that $\theta = \sqrt{V_1 / \var(Y_0)}$). Note that we did not optimize the design matrix $\mT$ and the refiners $n_i$, $i=2,\dots,R$.
%the design of the multilevel estimators, namely the design matrix $\mT$ and the refiners $n_i$, $i=2,\dots,R$. We propose in this Section a numerical procedure to choose a good value of $M$ in the case $n_i = M^{i-1}$.

\subsubsection*{About structural parameters}
Implementing \MLMC or \MLRR estimator needs to know both the weak and strong rates of convergence of the biased estimator $Y_h$ toward $Y_0$. The exponents $\alpha$ and $\beta$ are generally known by a mathematical study of the approximation (see Section \ref{sec:brown} for Brownian diffusion discretization and Section \ref{sec:nested} for nested Monte Carlo). The parameter $V_1$ comes from the strong approximation rate assumption~\eqref{strong_error} and a natural approximation for $V_1$ is 
\begin{equation*}
	V_1 \simeq \limsup_{h \to 0} h^{-\beta} \normLp{2}{Y_h-Y_0}^2
\end{equation*}
Since $Y_0$ cannot be simulated at a reasonable computational cost, one may proceed as follows to get a good empirical estimator of $V_1$. First, assume that, in fact, $\|Y_h-Y_0\|^2_2\sim V_1h^{\beta}$ as $h \to 0$ but that this equivalence still holds as an approximation for not too small parameters $h$. Then, one derives from Minkowski's Inequality that, for every integer $M\ge 1$,
\[
\big\|Y_h-Y_{\frac hM}\big\|_{_2} \le \big\|Y_h-Y_{0}\big\|_{_2}+\big\|Y_0-Y_{\frac hM}\big\|_{_2}
\]
so that 
\[
V_1\gtrsim (1+M^{-\frac{\beta}{2}})^{-2} h^{-\beta} \normLp{2}{Y_h - Y_{\frac{h}{M}}}^2.
\]
As a consequence, if we choose $M=M_{\max}$ large enough (see~\eqref{eq:M_max} below), we are led to consider the following estimator 
\begin{equation}\label{eq:hatV1}
	\widehat{V}_1(h) =  \big(1+M_{\max}^{-\frac {\beta}{2}}\big)^{-2} h^{-\beta}\|Y_h-Y_{\frac{h}{M_{\max}}}\|^2_{_2}.
\end{equation}
If the assumption \eqref{strong_error} is replaced by one of the alternative assumptions \eqref{var_error} and \eqref{var_error'} proposed in Remark~\ref{rq:strong_assumption}, the estimator of $V_1$ must be modified. For instance if we consider the assumption $\var\pa{Y_h-Y_{h'}} \le V_1 \abs{h-h'}^{\beta}$, a standard estimator of $V_1$ becomes $\hat V_1(h) = \pa{1- M^{-\beta}}^{-1} h^{-\beta} \var\pa{Y_h - Y_{\frac{h}{M}}}$, $M$ being fixed.

The estimation of the real constants $c_i$, $c_1$ for crude Monte Carlo and an \MLMC estimators and $\widetilde c = \lim_{R \to \infty} |c_{_R}|^{\frac{1}{R}}$ for the \MLRR estimator is much more  challenging. So, these methods are usually  implemented in a blind way by considering the coefficients $c_1$ and $|c_{_R}|^{\frac{1}{R}}$ equal to $1$.

Note that, even in a crude Monte Carlo method, such structural parameters are useful (and sometimes necessary) to deal with the bias error (see Proposition \ref{pro:MC}).

\subsubsection*{Design of the Multilevel}
The standard design matrix is fixed by the template~\eqref{tempMLRR} for the multilevel Richardson-Romberg estimator and by the template~\eqref{tempMLMC} for the multilevel Monte Carlo estimator. Alternative choices could be to consider for the \MLRR estimator another design matrix $\mT$ satisfying~\eqref{Abeta} like $\mT^j = -\w_j e_1 + \w_j e_j$ for $j\in\ac{2,\dots,R}$ which reads 
\begin{equation} \label{tempMLRR2}
	\mT = \pa{\begin{array}{ccccc}
			1 & -\w_2 & -\w_3 & \cdots & -\w_R \\
			0 & \w_2 & 0 & \cdots & 0 \\
				0 & 0 & \w_3 & \ddots & 0 \\
			\vdots  &\vdots  &\ddots & \ddots &\vdots \\
	%		\vdots & \vdots & \vdots &\vdots & \vdots\\
			0 & 0 & \cdots & 0 & \w_R \\
	\end{array}}.
\end{equation}
We could also consider a lower triangular design matrix (through it does not satisfy the conventional assumption $T^1=e_1$)
\begin{equation} \label{temMLRR3}
	\mT = \pa{\begin{array}{cccccc }
			\widetilde \W_1 & 0 & \cdots & \cdots & \cdots & 0 \\
			-\widetilde \W_1  & \widetilde \W_{2} & 0 & \cdots & \cdots & 0 \\
			\vdots & \ddots & \ddots & \ddots &\ddots & \vdots\\
			\vdots & & \ddots & \ddots & \ddots & 0\\
			0 & \cdots & \cdots & -\widetilde \W_{R-2} & \widetilde \W_{R-1} & 0 \\
			0 & \cdots & \cdots & \cdots & -\widetilde \W_{R-1} & 1 \\
	\end{array}} \quad 
	\text{where} \quad \widetilde \W_j = \sum_{k=1}^j \w_k.
\end{equation}
 
The refiners can be specified by users but it turns out that the parametrized family $n_i = M^{i-1}$, $i=1,\ldots,R$ ($M\!\in \N$, $M\ge 2$) seems the best compromise between variance control and implementability. The parameter $\alpha$ being settled, all the related quantities like $(\W_i(R,M))_{1\le i\le M}$ can be tabulated for various values of $M$ and $R$ and can be stored offline.

\subsubsection*{Taking advantage of $c_1=0$} When $c_1=0$, only $R-1$ weights are needed to cancel the (remaining) coefficients up to order $R$ \ie $c_r$, $r=2,\ldots, R-1$ (instead of $R$). One easily shows that, if $\big(\w^{(R-1)}_r\big)_{r=1,\ldots,R-1}$  denotes the weight vector at order $R-1$ associated to refiners $n_1=1<n_2,\ldots,n_{_{R-1}}$ (for a given $\alpha$), then the weight vector $\widetilde \w^{(R)}$ at order $R$ (with size $R-1$) reads
\[
\widetilde \w^{(R)}_r= \frac{n_r^{\alpha} \w^{(R-1)}_r}{\sum_{1\le s\le R-1}n_s^{\alpha} \w^{(R-1)}_s}, \;\; r=1, \ldots, R-1.
\]

\subsubsection*{Optimal parameters}
\paragraph{Diffusion approximation}
In the case $n_i = M^{i-1}$ (with the convention $n_0 = n_0^{-1} = 0$), we can summarize the asymptotic optimal value of the parameters $q$, $R$, $h$ and $N$ in Table~\ref{tab:opt_param_MLRR} for the~\eqref{tempMLRR} estimator and  in Table~\ref{tab:opt_param_MLMC} for the~\eqref{tempMLMC} estimator. 
\begin{table}[!ht] 
	\centering
	\begin{tabular}{c|c}
        $R(\varepsilon)$ & $\ds \left\lceil \frac 12 + \frac{\log\pa[1]{\widetilde c^{\frac{1}{\alpha}} \mathbf{h}}}{\log(M)} 
		+ \sqrt{ \pa[3]{\frac 12 + \frac{\log\pa[1]{\widetilde c^{\frac{1}{\alpha}} \mathbf{h}}}{\log(M)} }^2 + 2 \frac{\log\pa{A/\varepsilon}}{\alpha \log(M)}}\ \right\rceil, \quad A = \sqrt{1+4\alpha} $
	\\ 
	\midrule
    $h(\varepsilon)$ & $\ds {\mathbf h}/ \left\lceil(1+2 \alpha R)^{-\frac{1}{2 \alpha R}} \varepsilon^{\frac{1}{\alpha R}} \widetilde c^{-\frac{1}{\alpha}} M^{\frac{R-1}{2}} {\mathbf h}\right\rceil$  
	\\
	\midrule
	$q(\varepsilon)$ & 
	$ 
\begin{aligned}
	q_1 &= \mu_{_R}^* \big(1 +\theta h^{\frac \beta 2}\big) \\
	q_j &= \mu_{_R}^* \theta h^{\frac \beta 2} \pa{\abs{\W_j(R,M)} \frac{n_{j-1}^{-\frac{\beta}{2}} + n_j^{-\frac{\beta}{2}}}{\sqrt{n_{j-1}+n_j}}},\; j = 2,\dots,R;\; \mu_{_R}^*\;s.t.\,   \sum_{1\le j\le R}q_j= 1\\
\end{aligned}
		$
		\\
	\midrule
	$N(\varepsilon)$ & $\ds \pa{1+\frac{1}{2 \alpha R}} \frac{\ds \var(Y_0)
\pa{1 + \theta h^{\frac{\beta}{2}} \sum_{j=1}^R \abs{\W_j(R,M)} \pa{n_{j-1}^{-\frac{\beta}{2}} + n_j^{-\frac{\beta}{2}}}\sqrt{n_{j-1}+n_j}}} {\ds \varepsilon^2 \mu^*_{_R} }$ \\
\end{tabular}
\caption{Optimal parameters for the \MLRR estimator (standard case).}
\label{tab:opt_param_MLRR}\end{table}

\begin{table}[!ht]
	\centering
	\begin{tabular}{c|c}
		$R(\varepsilon)$ & $\ds \left\lceil 1 + \frac{\log\pa[1]{\abs{c_1}^{\frac{1}{\alpha}} \mathbf{h}}}{\log(M)} + \frac{\log(A / \varepsilon)}{\alpha \log(M)}\right\rceil, \quad A = \sqrt{1+2 \alpha} $
	\\ 
	\midrule
	$h(\varepsilon)$ & $\ds  {\mathbf h}/\left\lceil (1+2 \alpha)^{-\frac{1}{2 \alpha}} \varepsilon^{\frac{1}{\alpha}} \abs{c_1}^{-\frac{1}{\alpha}} M^{R-1}{\mathbf h} \right\rceil$	\\
	\midrule
	$q(\varepsilon)$ & 
	$ 
\begin{aligned}
	q_1 &= \mu_{_R}^* (1 +\theta h^{\frac \beta 2}) \\
	q_j &= \mu_{_R}^* \theta h^{\frac \beta 2} \pa{
\frac{n_{j-1}^{-\frac{\beta}{2}} + n_j^{-\frac{\beta}{2}}}{\sqrt{n_{j-1}+n_j}}}
, \; j = 2,\dots,R; \; \mu_{_R}^*\;s.t.\, \sum_{1\le j\le R}q_j= 1 \\
\end{aligned}
		$
		\\
	\midrule
	$N(\varepsilon)$ & $\ds \pa{1+\frac{1 }{2 \alpha}}\frac{\ds \var(Y_0)
\pa{1 + \theta h^{\frac{\beta}{2}} \sum_{j=1}^R \pa{n_{j-1}^{-\frac{\beta}{2}} + n_j^{-\frac{\beta}{2}}}\sqrt{n_{j-1}+n_j}}} {\ds \varepsilon^2 \mu_{_R}^*}$ \\
\end{tabular}
\caption{Optimal parameters for the \MLMC estimator (standard case).}
	\label{tab:opt_param_MLMC}
\end{table}

\paragraph{Nested Monte Carlo} In a Nested Monte Carlo framework, the unitary complexity is given by~\eqref{eq:nested_compl}. \begin{itemize}
    \item The unitary cost term $n_{j-1}+n_j$ in Tables~\ref{tab:opt_param_MLRR} and \ref{tab:opt_param_MLMC} must be replaced by $n_j$. 
    \item The unitary variance term $n_{j-1}^{-\beta/2}+n_j^{-\beta/2}$ must be replaced by $(\frac{1}{n_{j-1}}-\frac{1}{n_j})^{\beta/2}$.
    \end{itemize}

\paragraph{Optimization of the root $M$}
    Note that these optimal parameters given in the above Tables only depend on the structural parameters and on the user's choice of the root $M\ge 2$ for the refiners. For a fixed $\varepsilon > 0$, if we emphasize the dependance in $M = M(\varepsilon)$ \ie $R(M)$, $h(M)$, $q(M)$ and $N(M)$ the global cost $C_\varepsilon$ as a function of $M$ is given by  
\begin{equation} \label{cout_a_minimiser}
C_\varepsilon(M) = \Cost(\bar Y^{N(M), q(M)}_{h(M), \underline n}) = N(M) \cost(h(M), R(M), q(M)),
\end{equation}
where $\cost(h, R, q) = \frac{1}{h} \sum_{j=1}^R q_j \sum_{i=1}^R n_i \ind{\mT_i^j \neq 0}$ (in the framework of Section~\ref{sec:brown}) and $\cost(h, R, q) = \frac{1}{h} \sum_{j=1}^R q_j \max_{1 \le i \le R} n_i \ind{\mT_i^j \neq 0}$ (in the framework of Section~\ref{sec:nested}). 
This function can be optimized for likely values of $M$. In our numerical experiments, we consider
\begin{equation} \label{eq:M_max}
	M = \argmin_{M \in \ac{2,\dots,M_{\max}}} C_\varepsilon(M) \quad \text{with} \quad M_{\max} = 10.
\end{equation}

\subsection{Correlation between $Y_{\frac{h}{n_i}}$ and $Y_{\frac{h}{n_{i-1}}}$} \label{sec:correlation}
\paragraph{Diffusion approximation}
In many situations (like $e.g.$ the numerical experiments carried out below), discretization schemes of Brownian diffusions need to be simulated with various steps (say $\frac{T}{nn_i}$ and $\frac{T}{nn_{i+1}}$ in our case). This requires to simulate consistent Brownian increments over $[0,\frac Tn]$, then $[\frac{(k-1)T}{n},\frac{kT}{n}]$, $k=2,\ldots,n$. This can be performed by simulating recursively the Brownian increments over  all successive sub-intervals of interest, having in mind that the  ``quantum" size for  the simulation is given by $\frac{T}{nm}$ where $m= gcd(n_1,\ldots,n_{_R})$. This recursive refinement is also known as the Brownian Bridge simulation procedure.
One can also produce  once and for all an {\em abacus} of coefficients to compute  by induction the needed increments from smaller subintervals up to the root interval of length $\frac Tn$. This is done $e.g.$ in~\cite{PAG0} up to $R=5$  for  $\alpha=1$ and up to $R=3$ for $\alpha =\frac 12$. 

\paragraph{Nested Monte Carlo}
In a Nested Monte Carlo the relation between $Y_{\frac{h}{n_i}}$ and $Y_{\frac{h}{n_{i-1}}}$ is simply based on the following rule: the $n_{i-1}/h$ first terms of the sequence of copies of $Z$ used to simulate $Y_{\frac{h}{n_{i-1}}}$ must be used to simulate  $Y_{\frac{h}{n_i}}$.

\subsection{Methodology}
We compare the two \MLMC and \MLRR estimators for different biased problems. In the sequel, we consider the standard design matrix~\eqref{tempMLRR} for the \MLRR estimator, idem for the \MLMC estimator. After a crude evaluation of $\var(Y_0)$ and $V_1$ (using~\eqref{eq:hatV1}) we compute the ``optimal" parameter $M$ solution to~\eqref{eq:M_max}. The others parameters are specified according to Tables~\eqref{tab:opt_param_MLRR} and~\eqref{tab:opt_param_MLMC} with $\tilde c = c_1 = 1$. 
%In the numerical simulations we do not round down $R$, we round to the nearest integer.

The empirical bias error $\tilde{\mu}_{_L}$ of the estimator $\bar Y_{h, \underline n}^{N,q}$ is obtained using $L = 256$ independent replications of the estimator, namely
\begin{equation*}
	\tilde{\mu}_{_L} = \frac{1}{L} \sum_{\ell=1}^{L} (\bar Y_{h, \underline n}^{N,q})^{(\ell)} - I_0,
\end{equation*}
where $I_0 = \esp{Y_0}$ is the true value. 

The empirical $\L^2$--error or empirical root mean squared error (RMSE) $\tilde \varepsilon_L$ of the estimator used in our numerical experiments is given by 
	\begin{equation} \label{eq:L2_empiric_error}
	\tilde{\varepsilon}_{_L} = \sqrt{
		\frac{1}{L} \sum_{\ell=1}^{L} \pa{(\bar Y_{h, \underline n}^{N,q})^{(\ell)} - I_0}^2}.
	\end{equation}

The computations were performed on a computer with 4 multithreaded(16) octo-core processors (Intel(R) Xeon(R) CPU E5-4620 @ 2.20GHz). The code of one estimator runs on a single thread (program in \texttt{C++11} available on request).

\subsection{Euler scheme of a geometric Brownian motion: pricing of European options}
We consider a geometric Brownian motion $(S_t)_{t\in [0,T]}$, representative in a Black-Scholes model of the dynamics of a risky asset price  between time $t=0$ and time $t=T$:
\[
S_t = s_0 e^{(r-\frac{\sigma^2}{2})t +\sigma W_t}, \; t\!\in [0,T],\; S_0=s_0>0,
\]
where $r$ denotes the    (constant)  ``riskless"  interest rate, $\sigma$ denotes the volatility and $W=(W_t)_{t\in [0,T]}$ is a standard Brownian motion defined on a probability space $(\Omega,{\cal A}, \P)$. The price or {\em premium} of a so-called {\em vanilla} option with payoff $\varphi$ is given by $e^{-rT}\esp{\varphi(S_{_T})}$ and the price of a {\em path dependent} option with functional payoff $\varphi$ is given by $e^{-rT}\esp{\varphi((S_t)_{t \in [0,T]})}$. 
Since $(S_t)_{t \in [0,T]}$ is solution to the diffusion $SDE$
\[
dS_t = S_t(r \d t + \sigma d W_t), \quad S_0 = s_0 > 0,
\]
one can compute the price of an option by a Monte Carlo simulation in which the true process $(S_t)_{t\in [0,T]}$ is replaced by its Euler scheme $(\bar S_{kh})_{0\le k\le n}$, $h = \frac{T}{n}$ (even if we are aware that $S_{_T}$ can be simulated). The bias parameter set $\Hr$ is then defined by $\Hr = \ac{T/n, \; n \ge 1}$ and $\mathbf{h} = T$.

Although nobody would adopt any kind of Monte Carlo simulation to compute option price in this model since a standard difference method on the Black-Scholes parabolic PDE is much more efficient to evaluate a vanilla option and many path-dependent ones, it turns out that the time discretization of a Black-Scholes model and its Euler scheme is a very demanding {\em benchmark} to test and evaluate the performances of Monte Carlo method(s). As a consequence, it is quite appropriate to carry out numerical tests with \MLRR and \MLMC.

\subsubsection{Vanilla Call option ($\alpha = \beta = 1$)}
The Black-Scholes parameters considered here are $s_0 = 100$, $r = 0.06$ and $\sigma = 0.4$. The payoff is a European {\em Call} with maturity $T=1$ year and strike $K= 80$. 

In such a regular diffusion setting (both drift and diffusion coefficients are ${\cal C}^{\infty}_b$ and the payoff function is Lipschitz continuous), one has $\alpha = \beta = 1$.
The parameters $\theta= \sqrt{V_1/\var(Y_0)}$ and $\var(Y_0)$ have been roughly estimated following the procedure~\eqref{eq:hatV1} on a sample of size $100\,000$ described in Section~\ref{sec:practi}, leading to the values $V_1 \simeq 56$ and $\var(Y_0) \simeq 876$ (so that $\theta \simeq 0.25$).
The empirical $\L^2$--error $\tilde \epsilon_{L}$ is estimated using $L = 256$ runs of the algorithm and the bias is computed using the true value of the price $I_0 = 29.4987$ provided by the Black-Scholes formula.

The results are summarized in Table~\ref{tab:call_MLRR} for the \MLRR estimator and in Table~\ref{tab:call_MLMC} for the \MLMC estimator. 

\begin{table}[ht!]\label{tab:bs_MLRR} 
\centering 
\begin{tabular}{c|c|c|c|c|c||c|c|c|c|c}
	$k$ & $\varepsilon=2^{-k}$ & $\L^2$--error & time $(s)$ & bias & variance 
	& $R$ & $M$ & $h^{-1}$ & $N$ & $\Cost$ \\
	\midrule
1 & 5.00\e{-01} & 3.91\e{-01} & 3.02\e{-02} & 1.47\e{-01} & 1.31\e{-01} & 2 & 5 & 1 & 1.50\e{+04} & 2.47\e{+04} \\
2 & 2.50\e{-01} & 2.18\e{-01} & 1.12\e{-01} & 8.99\e{-02} & 3.96\e{-02} & 2 & 9 & 1 & 5.91\e{+04} & 1.06\e{+05} \\
3 & 1.25\e{-01} & 9.28\e{-02} & 5.59\e{-01} & -5.61\e{-04} & 8.62\e{-03} & 3 & 4 & 1 & 3.19\e{+05} & 7.09\e{+05} \\
4 & 6.25\e{-02} & 5.01\e{-02} & 2.12\e{+00} & -1.90\e{-02} & 2.15\e{-03} & 3 & 4 & 1 & 1.27\e{+06} & 2.84\e{+06} \\
5 & 3.12\e{-02} & 2.71\e{-02} & 8.13\e{+00} & -1.15\e{-02} & 6.00\e{-04} & 3 & 5 & 1 & 4.99\e{+06} & 1.15\e{+07} \\
6 & 1.56\e{-02} & 1.35\e{-02} & 3.22\e{+01} & -4.41\e{-03} & 1.63\e{-04} & 3 & 6 & 1 & 1.99\e{+07} & 4.72\e{+07} \\
7 & 7.81\e{-03} & 6.98\e{-03} & 1.31\e{+02} & -2.32\e{-03} & 4.33\e{-05} & 3 & 7 & 1 & 7.98\e{+07} & 1.95\e{+08} \\
8 & 3.91\e{-03} & 3.57\e{-03} & 5.51\e{+02} & -9.35\e{-04} & 1.19\e{-05} & 3 & 9 & 1 & 3.25\e{+08} & 8.37\e{+08} \\
\end{tabular}
\caption{Call option ($\alpha = 1, \beta = 1$): Parameters and results of the \MLRR estimator.}
\label{tab:call_MLRR}
\end{table}

As an example, the third line of the Table~\ref{tab:call_MLRR} reads as follows: for a prescribed RMSE error $\varepsilon = 2^{-3} = 0.125$, the \MLRR estimator $\bar Y_{h, \underline n}^{N, q}$ (with design matrix~\eqref{tempMLRR}) is implemented with the parameters $R = 3$, $h = 1$ and refiners $n_i = 4^{i-1}$ (then $n_1 = 1$, $n_2 = 4$ and $n_3 = 16$) and the sample size $N \simeq 319\,000$. The allocation weights $q_i$ (not reported in this Table) are such that the numerical cost $\Cost(\bar Y_{h, \underline n}^{N, q}) \simeq 709\,800$. For such parameters, the empirical RMSE $\tilde \varepsilon_{_L} \simeq 0.0928$ and the computational time of $Y_{h, \underline n}^{N, q} \simeq 0.559$ seconds. The empirical bias error $\tilde \mu_{_L}$ is reported in the 5th column (\emph{bias}) and the empirical unitary variance $\tilde \nu_{_L}$ is reported in the 6th column (\emph{variance}). Recall that $\tilde \varepsilon_{_L} = \sqrt{(\tilde \mu_{_L})^2 + \tilde \nu_{_L}}$.

\begin{table}[ht!]\label{tab:bs_MLMC} 
\centering 
\begin{tabular}{c|c|c|c|c|c||c|c|c|c|c}
	$k$ & $\varepsilon=2^{-k}$ & $\L^2$--error & time $(s)$ & bias & variance  
	& $R$ & $M$ & $h^{-1}$ & $N$ & $\Cost$ \\
	\midrule
1 & 5.00\e{-01} & 5.02\e{-01} & 2.53\e{-02} & 3.87\e{-01} & 1.02\e{-01} & 2 & 4 & 1 & 1.57\e{+04} & 2.32\e{+04} \\
2 & 2.50\e{-01} & 2.85\e{-01} & 1.31\e{-01} & 2.25\e{-01} & 3.04\e{-02} & 2 & 7 & 1 & 6.48\e{+04} & 1.06\e{+05} \\
3 & 1.25\e{-01} & 1.20\e{-01} & 6.28\e{-01} & 8.77\e{-02} & 6.63\e{-03} & 3 & 4 & 1 & 3.64\e{+05} & 7.33\e{+05} \\
4 & 6.25\e{-02} & 6.31\e{-02} & 2.44\e{+00} & 4.45\e{-02} & 2.00\e{-03} & 3 & 6 & 1 & 1.49\e{+06} & 3.32\e{+06} \\
5 & 3.12\e{-02} & 3.42\e{-02} & 1.05\e{+01} & 2.48\e{-02} & 5.59\e{-04} & 3 & 8 & 1 & 6.15\e{+06} & 1.47\e{+07} \\
6 & 1.56\e{-02} & 1.66\e{-02} & 5.17\e{+01} & 1.23\e{-02} & 1.22\e{-04} & 4 & 5 & 1 & 3.06\e{+07} & 8.38\e{+07} \\
7 & 7.81\e{-03} & 7.83\e{-03} & 2.20\e{+02} & 5.06\e{-03} & 3.57\e{-05} & 4 & 7 & 1 & 1.27\e{+08} & 3.82\e{+08} \\
8 & 3.91\e{-03} & 4.48\e{-03} & 9.14\e{+02} & 3.26\e{-03} & 9.43\e{-06} & 4 & 8 & 1 & 5.17\e{+08} & 1.62\e{+09} \\
\end{tabular}
\caption{Call option ($\alpha = 1, \beta = 1$): Parameters and results of the \MLMC estimator.}
\label{tab:call_MLMC} 
\end{table}

Note first that, as expected, the depth parameter $R \ge 2$ and the numerical cost $\Cost(\bar Y_{h, \underline n}^{N,q})$ grow slower for \MLRR than for \MLMC as $\varepsilon$ goes to 0. Consequently, regarding the CPU--time for a prescribed  error $\varepsilon=2^{-k}$, \MLRR is about 10\% to 100\% (twice) faster than \MLMC when $k$ goes from $2$ to $8$. On the other hand, both estimators \MLRR and \MLMC provide an empirical RMSE close to the prescribed RMSE \ie $\tilde \varepsilon_{_L} \le \varepsilon$. We can conclude that the automatic tuning of the algorithm parameters is satisfactory for both estimators.

In Figure~\ref{fig:call}$(a)$ is depicted the CPU--time (4th column) as a function of the {\em empirical} $\L^2$--error (3rd column). It provides a direct comparison of the performance of both estimators. Each point is labeled by the prescribed RMSE $\varepsilon = 2^{-k}$, $k=1,\dots,8$ for easy reading. The plot is in $\log_2$--$\log$ scale. The \MLRR estimator (blue solid line) is below the \MLMC estimator (red dashed line). The ratio of CPU--times for a given $\tilde \varepsilon_{_L}$ shows that \MLRR goes from $1.28$ up to $2.8$ faster, within the range of our simulations. Figure~\ref{fig:call}$(b)$ represents the product (CPU--time)$\times \varepsilon^2$ as a function of $\varepsilon$. %Note that this function is almost constant (as expected) for both estimators. 
\begin{figure}[ht!] 
	\begin{minipage}[b]{.49\linewidth}
        \centering \includegraphics[width=\linewidth]{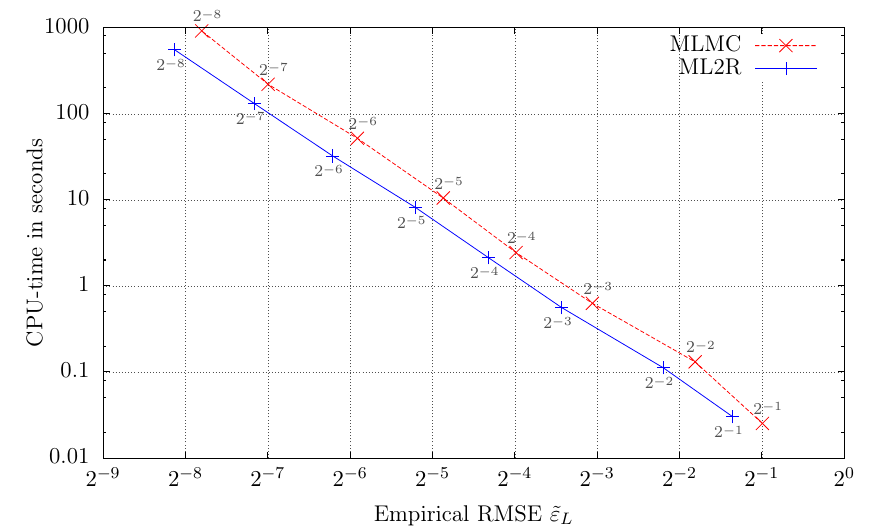}
        \subcaption{CPU--time ($y$--axis, $\log$ scale) as a function of $\tilde \varepsilon_L$ ($x$--axis, $\log_2$ scale).}
    \end{minipage} \hfill
	\begin{minipage}[b]{.49\linewidth}
		\centering \includegraphics[width=\linewidth]{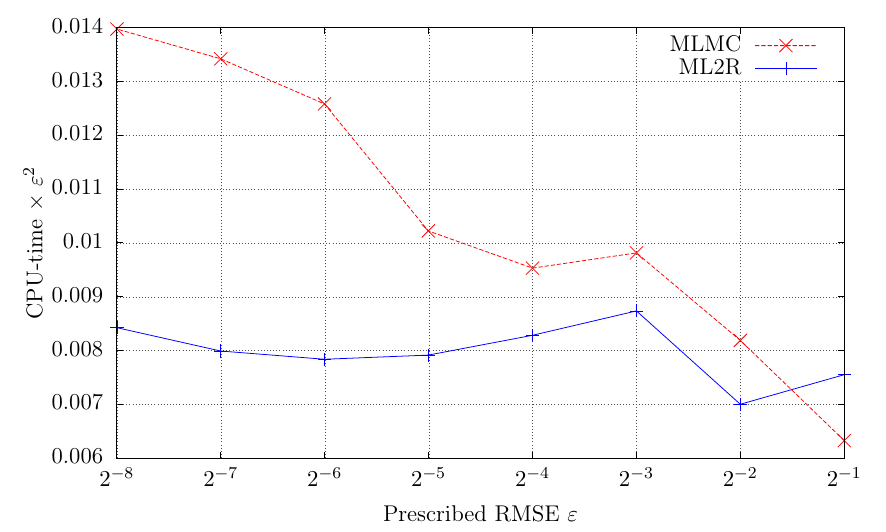}
        \subcaption{CPU--time $\times$ $\varepsilon^2$ ($y$--axis) as a function of $\varepsilon$ ($x$--axis, $\log_2$ scale).}
	\end{minipage}
    \caption{Call option in a Black-Scholes model.}
	\label{fig:call}
\end{figure}

\subsubsection{Lookback option ($\alpha = 0.5$, $\beta = 1$)}
We consider a partial Lookback Call option defined by its functional payoff 
\begin{equation*}
	\varphi(x) = e^{-rT} \pa[1]{x(T) - \lambda \min_{t \in [0,T]} x(t)}_+, \quad x \in \mathcal{C}([0,T], \R)
\end{equation*}
where $\lambda \ge 1$. 
The parameters of the Black-Scholes model are $s_0 = 100$, $r = 0.15$, $\sigma = 0.1$ and $T= 1$ and the coefficient $\lambda$ is set at $\lambda = 1.1$. For these parameters, the price given by a closed-form expression is $I_0 = 8.89343$. 

For such payoff with Lipschitz continuous functional,~\eqref{strong_error} holds with $\beta = 1$ and~\eqref{weak_error} holds with $\alpha = 0.5$. Note that the full expansion $\bar R = +\infty$ is not yet proved to our knowledge. An estimation of structural parameters yields $\var(Y_0) \simeq 41$ and $V_1 \simeq 3.58$ (and then $\theta \simeq 0.29$). Both estimators are implemented using the automatic tuning previously exposed.

The results are summarized in Table~\ref{tab:lookback_MLRR} for the~\MLRR and in Table~\ref{tab:lookback_MLMC} for the~\MLMC. Note first that as a function of the  prescribed $\varepsilon= 2^{-k}$ the ratio between CPU--times goes from $1.1$ ($k=2$) up to $3.5$ ($k=9$), as does the ratio $\Cost(\text{\MLMC}) / \Cost(\text{\MLRR})$. However, the empirical RMSE of \MLMC is greater than $\varepsilon$ (certainly because $c_1 \neq 1$) unlike that of \MLRR.
% (see~\ref{fig:lookback}$(a)$ in Appendix~C).
One observes that the $\L^2$--error of \MLRR has a very small bias $\tilde \mu_L$ (5th column) due to the particular choice of the weights $(\W_i)_{1 \le i \le R}$.

Figure~\ref{fig:lookback}$(a)$ provides a graphical representation of the performance of both estimators, now  as  a function of the empirical RMSE $\widetilde \varepsilon$. It shows that  \MLRR is faster  then \MLMC by a factor that goes from  $18$ up to $48$ within the range of our simulations. 

\begin{table}[ht!]
\centering 
\begin{tabular}{c|c|c|c|c|c||c|c|c|c|c}
	$k$ & $\varepsilon=2^{-k}$ & $\L^2$--error & time $(s)$ & bias & variance 
	& $R$ & $M$ & $h^{-1}$ & $N$ & $\Cost$ \\
	\midrule
1 & 5.00\e{-01} & 3.54\e{-01} & 2.42\e{-03} & -5.80\e{-02} & 1.22\e{-01} & 3 & 6 & 1 & 1.46\e{+03} & 4.40\e{+03} \\
2 & 2.50\e{-01} & 1.80\e{-01} & 1.04\e{-02} & -3.66\e{-02} & 3.10\e{-02} & 3 & 6 & 1 & 5.82\e{+03} & 1.76\e{+04} \\
3 & 1.25\e{-01} & 9.95\e{-02} & 4.17\e{-02} & -3.98\e{-02} & 8.31\e{-03} & 3 & 7 & 1 & 2.30\e{+04} & 7.07\e{+04} \\
4 & 6.25\e{-02} & 5.45\e{-02} & 1.53\e{-01} & -9.53\e{-03} & 2.88\e{-03} & 3 & 10 & 2 & 6.48\e{+04} & 3.55\e{+05} \\
5 & 3.12\e{-02} & 2.31\e{-02} & 8.69\e{-01} & -1.50\e{-03} & 5.33\e{-04} & 4 & 5 & 1 & 4.50\e{+05} & 1.68\e{+06} \\
6 & 1.56\e{-02} & 1.22\e{-02} & 3.43\e{+00} & -8.49\e{-04} & 1.47\e{-04} & 4 & 6 & 1 & 1.77\e{+06} & 6.74\e{+06} \\
7 & 7.81\e{-03} & 6.31\e{-03} & 1.39\e{+01} & -2.76\e{-04} & 3.98\e{-05} & 4 & 7 & 1 & 7.03\e{+06} & 2.74\e{+07} \\
8 & 3.91\e{-03} & 3.34\e{-03} & 5.74\e{+01} & 1.19\e{-04} & 1.11\e{-05} & 4 & 9 & 1 & 2.83\e{+07} & 1.16\e{+08} \\
9 & 1.95\e{-03} & 1.80\e{-03} & 2.10\e{+02} & 1.08\e{-04} & 3.23\e{-06} & 4 & 10 & 2 & 7.88\e{+07} & 5.45\e{+08} \\
\end{tabular}
\caption{Lookback option ($\alpha = 0.5, \beta = 1$): Parameters and results of the \MLRR estimator.}
\label{tab:lookback_MLRR}
\end{table}

\begin{table}[ht!]
\centering 
\begin{tabular}{c|c|c|c|c|c||c|c|c|c|c}
	$k$ & $\varepsilon=2^{-k}$ & $\L^2$--error & time $(s)$ & bias & variance 
	& $R$ & $M$ & $h^{-1}$ & $N$ & $\Cost$ \\
	\midrule
1 & 5.00\e{-01} & 1.35\e{+00} & 1.47\e{-03} & -1.32\e{+00} & 6.60\e{-02} & 2 & 8 & 1 & 1.17\e{+03} & 2.05\e{+03} \\
2 & 2.50\e{-01} & 6.86\e{-01} & 1.13\e{-02} & -6.72\e{-01} & 1.87\e{-02} & 3 & 6 & 1 & 6.80\e{+03} & 1.61\e{+04} \\
3 & 1.25\e{-01} & 3.00\e{-01} & 6.27\e{-02} & -2.91\e{-01} & 5.37\e{-03} & 4 & 6 & 1 & 3.59\e{+04} & 1.11\e{+05} \\
4 & 6.25\e{-02} & 1.96\e{-01} & 2.73\e{-01} & -1.92\e{-01} & 1.57\e{-03} & 4 & 8 & 1 & 1.49\e{+05} & 5.04\e{+05} \\
5 & 3.12\e{-02} & 9.25\e{-02} & 1.46\e{+00} & -9.03\e{-02} & 4.04\e{-04} & 5 & 7 & 1 & 7.26\e{+05} & 2.93\e{+06} \\
6 & 1.56\e{-02} & 4.38\e{-02} & 6.80\e{+00} & -4.25\e{-02} & 1.20\e{-04} & 5 & 10 & 1 & 3.10\e{+06} & 1.40\e{+07} \\
7 & 7.81\e{-03} & 2.47\e{-02} & 3.26\e{+01} & -2.42\e{-02} & 2.87\e{-05} & 6 & 8 & 1 & 1.42\e{+07} & 7.17\e{+07} \\
8 & 3.91\e{-03} & 9.06\e{-03} & 1.72\e{+02} & -8.64\e{-03} & 7.49\e{-06} & 7 & 8 & 1 & 6.62\e{+07} & 3.89\e{+08} \\
9 & 1.95\e{-03} & 6.16\e{-03} & 7.34\e{+02} & -6.00\e{-03} & 1.97\e{-06} & 7 & 9 & 1 & 2.71\e{+08} & 1.66\e{+09} \\
\end{tabular}
\caption{Lookback option ($\alpha = 0.5, \beta = 1$): Parameters and results of the \MLMC estimator.}
\label{tab:lookback_MLMC}
\end{table}

\begin{figure}[ht!] 
	\begin{minipage}[b]{.49\linewidth}
        \centering \includegraphics[width=\linewidth]{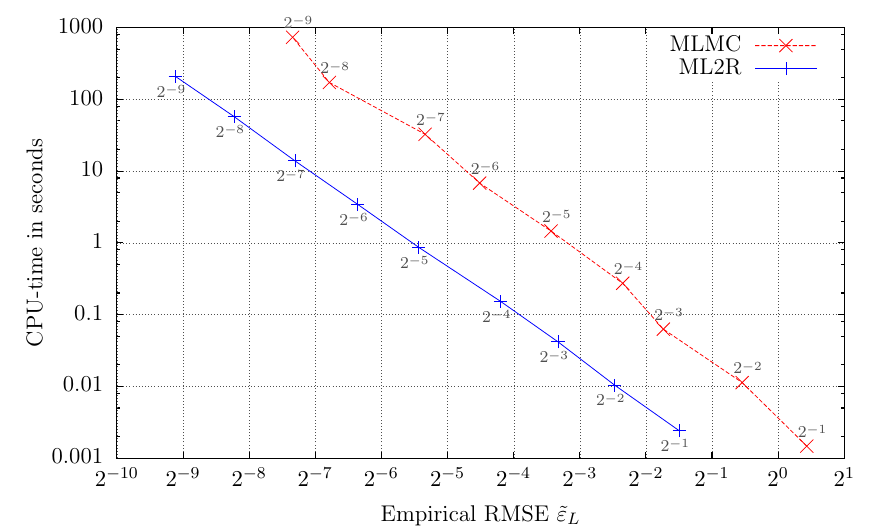}
        \subcaption{CPU--time ($y$--axis, $\log$ scale) as a function of $\tilde \varepsilon_L$ ($x$--axis, $\log_2$ scale).}
    \end{minipage} \hfill
	\begin{minipage}[b]{.49\linewidth}
		\centering \includegraphics[width=\linewidth]{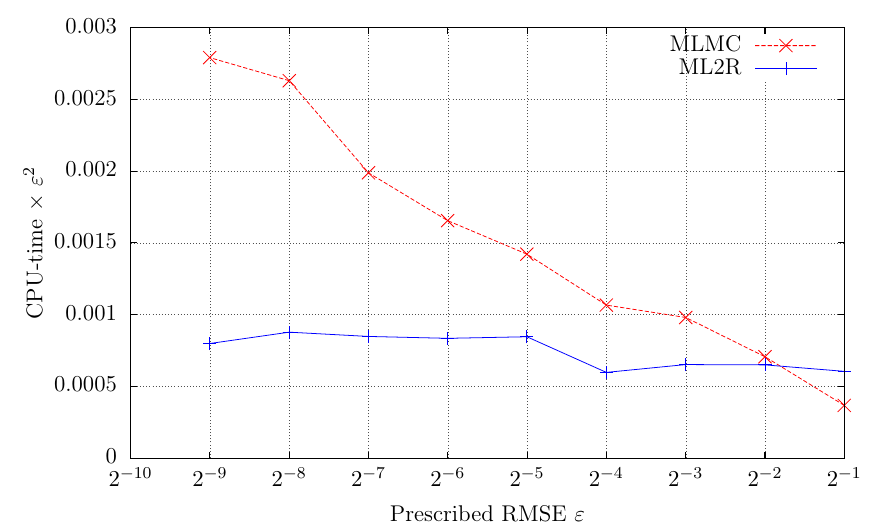}
        \subcaption{CPU--time $\times$ $\varepsilon^2$ ($y$--axis) as a function of $\varepsilon$ ($x$--axis, $\log_2$ scale).}
	\end{minipage}
    \caption{Lookback option in a Black-Scholes model.}
    \label{fig:lookback}
\end{figure}

\subsubsection{Barrier option ($\alpha = 0.5$, $\beta = 0.5$)}
We consider now an up-and-out call option to illustrate the case $\beta = 0.5 < 1$ and $\alpha = 0.5$. This path-dependent option with strike $K$ and barrier $B > K$ is defined by its functional payoff
\begin{equation*}
	\varphi(x) = e^{-r T} (x(T) - K)_+ \ind{\max_{t \in [0,T]} x(t) \le B}, \quad x \in \mathcal{C}([0,T], \R).
\end{equation*}
The parameters of the Black-Scholes model are $s_0 = 100$, $r = 0$, $\sigma = 0.15$ and $T = 1$. With $K = 100$ and $B = 120$, the price computed by closed-form solution is $I_0 = 1.855225$.

We consider here a simple (and highly biased) approximation of $\ds \max_{t \in [0, T]} S_t$ by $\ds \max_{k \in \{1,\dots,n\}} \bar S_{kh}$. This allows us to compare both estimators in the case $\beta = 0.5$. Like in the Lookback option, we assume that~\eqref{weak_error} holds with $\alpha = 0.5$ and $\bar R = +\infty$. 
A first computational stage gives us $\var(Y_0) \simeq 303$, $V_1 \simeq 5.30$ and $\theta \simeq 0.41$.

The results are summarized in Table~\ref{tab:barrier_MLRR} for \MLRR and in Table~\ref{tab:barrier_MLMC} for \MLMC. 
\begin{table}[ht!]
\centering
\begin{tabular}{c|c|c|c|c|c||c|c|c|c|c}
	$k$ & $\varepsilon=2^{-k}$ & $\L^2$--error & time $(s)$ & bias & variance 
	& $R$ & $M$ & $h^{-1}$ & $N$ & $\Cost$ \\
	\midrule
1 & 5.00\e{-01} & 3.85\e{-01} & 6.07\e{-03} & -3.92\e{-02} & 1.46\e{-01} & 3 & 4 & 1 & 2.65\e{+03} & 1.17\e{+04} \\
2 & 2.50\e{-01} & 1.94\e{-01} & 2.29\e{-02} & -3.82\e{-02} & 3.62\e{-02} & 3 & 4 & 1 & 1.06\e{+04} & 4.66\e{+04} \\
3 & 1.25\e{-01} & 1.14\e{-01} & 9.65\e{-02} & -2.00\e{-02} & 1.26\e{-02} & 3 & 7 & 1 & 4.02\e{+04} & 2.07\e{+05} \\
4 & 6.25\e{-02} & 6.28\e{-02} & 5.05\e{-01} & -5.45\e{-03} & 3.92\e{-03} & 3 & 10 & 2 & 1.34\e{+05} & 1.44\e{+06} \\
5 & 3.12\e{-02} & 2.83\e{-02} & 3.05\e{+00} & 1.24\e{-03} & 8.01\e{-04} & 4 & 5 & 1 & 1.01\e{+06} & 7.94\e{+06} \\
6 & 1.56\e{-02} & 1.49\e{-02} & 1.31\e{+01} & 6.98\e{-04} & 2.22\e{-04} & 4 & 6 & 1 & 4.15\e{+06} & 3.54\e{+07} \\
7 & 7.81\e{-03} & 7.81\e{-03} & 5.79\e{+01} & 7.82\e{-04} & 6.03\e{-05} & 4 & 7 & 1 & 1.71\e{+07} & 1.58\e{+08} \\
8 & 3.91\e{-03} & 4.13\e{-03} & 2.77\e{+02} & -2.01\e{-05} & 1.71\e{-05} & 4 & 9 & 1 & 7.39\e{+07} & 7.81\e{+08} \\
\end{tabular}
\caption{Barrier option ($\alpha = 0.5, \beta = 0.5$): Parameters and results of the \MLRR estimator.}
\label{tab:barrier_MLRR}
\end{table}
\begin{table}[ht!]
\centering 
\begin{tabular}{c|c|c|c|c|c||c|c|c|c|c}
	$k$ & $\varepsilon=2^{-k}$ & $\L^2$--error & time $(s)$ & bias & variance 
	& $R$ & $M$ & $h^{-1}$ & $N$ & $\Cost$ \\
	\midrule
1 & 5.00\e{-01} & 7.83\e{-01} & 2.26\e{-03} & 7.25\e{-01} & 8.73\e{-02} & 2 & 8 & 1 & 1.36\e{+03} & 2.83\e{+03} \\
2 & 2.50\e{-01} & 4.03\e{-01} & 2.05\e{-02} & 3.67\e{-01} & 2.75\e{-02} & 3 & 6 & 1 & 1.03\e{+04} & 3.57\e{+04} \\
3 & 1.25\e{-01} & 1.81\e{-01} & 1.83\e{-01} & 1.56\e{-01} & 8.30\e{-03} & 4 & 6 & 1 & 7.18\e{+04} & 4.28\e{+05} \\
4 & 6.25\e{-02} & 1.09\e{-01} & 9.52\e{-01} & 9.71\e{-02} & 2.47\e{-03} & 4 & 8 & 1 & 3.27\e{+05} & 2.40\e{+06} \\
5 & 3.12\e{-02} & 5.33\e{-02} & 8.38\e{+00} & 4.70\e{-02} & 6.27\e{-04} & 5 & 7 & 1 & 2.11\e{+06} & 2.40\e{+07} \\
6 & 1.56\e{-02} & 2.61\e{-02} & 6.16\e{+01} & 2.22\e{-02} & 1.88\e{-04} & 5 & 10 & 1 & 1.09\e{+07} & 1.74\e{+08} \\
7 & 7.81\e{-03} & 1.41\e{-02} & 4.90\e{+02} & 1.23\e{-02} & 4.51\e{-05} & 6 & 8 & 1 & 6.40\e{+07} & 1.43\e{+09} \\
8 & 3.91\e{-03} & 5.58\e{-03} & 6.05\e{+03} & 4.43\e{-03} & 1.15\e{-05} & 7 & 8 & 1 & 4.37\e{+08} & 1.67\e{+10} \\
\end{tabular}
\caption{Barrier option ($\alpha = 0.5, \beta = 0.5$): Parameters and results of the \MLMC estimator.}
\label{tab:barrier_MLMC}
\end{table}

See Figure~\ref{fig:barrier} for a graphical representation. Note that since $\beta = 0.5$, we observe that the function (CPU--time)$\times \varepsilon^2$ increases much faster for \MLMC than \MLRR as $\varepsilon$ goes to 0 which agrees with the theoretical asymptotic rates from Theorem~\ref{prop:asympReps}. In fact, in this highly biased example with slow strong convergence rate, the ratio $\Cost(\text{\MLMC}) / \Cost(\text{\MLRR})$ as a function of the  prescribed $\varepsilon= 2^{-k}$ goes from $1.1$ ($k=2$) up to $22$ ($k=8$), likewise the ratio between CPU--times behaves. When looking at this ratio as a function of the empirical RMSE, it even goes from $3$ up to $61$ which is huge having in mind that \MLMC provides similar gains with respect to a crude Monte Carlo simulation. 
\begin{figure}[ht!] 
	\begin{minipage}[b]{.49\linewidth}
        \centering \includegraphics[width=\linewidth]{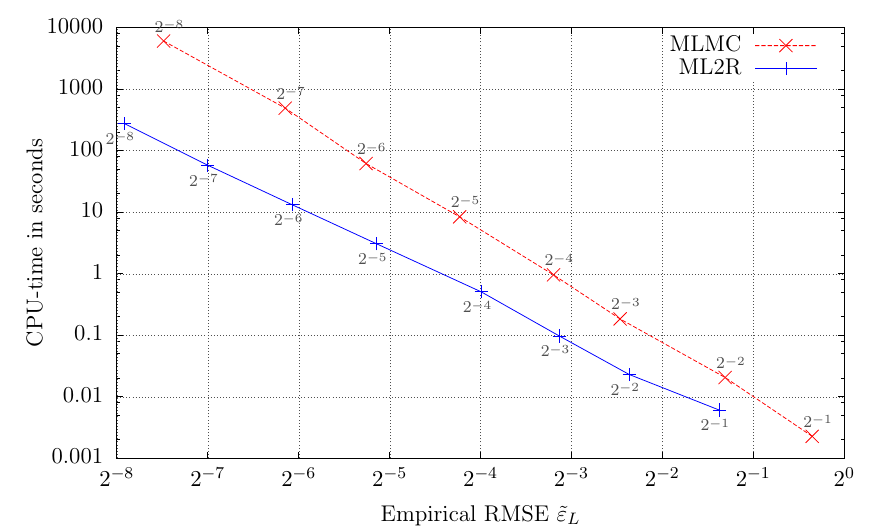}
        \subcaption{CPU--time ($y$--axis, $\log$ scale) as a function of $\tilde \varepsilon_L$ ($x$--axis, $\log_2$ scale).}
    \end{minipage} \hfill
	\begin{minipage}[b]{.49\linewidth}
		\centering \includegraphics[width=\linewidth]{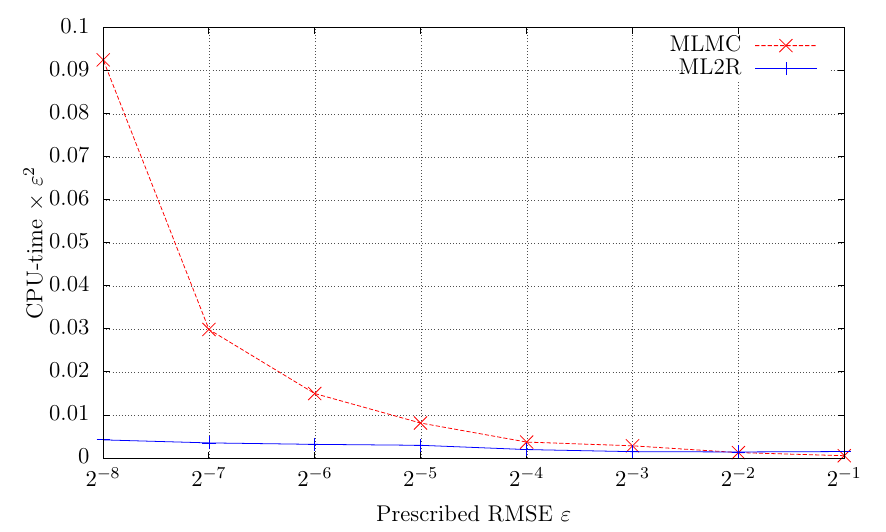}
        \subcaption{CPU--time $\times$ $\varepsilon^2$ ($y$--axis) as a function of $\varepsilon$ ($x$--axis, $\log_2$ scale).}
	\end{minipage}
    \caption{Barrier option in a Black-Scholes model.}
    \label{fig:barrier}
\end{figure}

\subsection{Nested Monte Carlo: compound option pricing ($\alpha = \beta = 1$)}
A compound option is simply an option on an option. The payoff of a compound option involves the value of another option. A compound option has then two expiration dates $T_1 < T_2$ and two strike prices $K_1$ and $K_2$. We consider here the example of a European style {\em Put} on a {\em Call} where the underlying risky asset $S$ is still given by a Black-Scholes process with parameters $(r,\sigma)$. At the first expiration date $T_1$, the holder has the right to sell a new {\em Call}  option at the strike price $K_1$. The new {\em Call}  has expiration date $T_2$ and strike price $K_2$. The payoff of such a {\em Put}-on-{\em Call}  option writes 
\begin{equation*}
	\pa{K_1 - \espc{(S_{T_2} - K_2)_+}{S_{T_1}}}_+
\end{equation*}

To comply with the multilevel framework, we set $\Hr = \{1/K, \, K\ge 1\}$,
\begin{equation*}
	Y_0 = f\pa[1]{\espc{(S_{T_2} - K_2)_+}{S_{T_1}}}, \quad  Y_{\frac 1 K} =f\pa{\frac 1K \sum_{k=1}^K (F(Z^k, S_{T_1}) - K_2)_+}
\end{equation*} where $(Z^k)_{k \ge 1}$ is an \iid sequence of standard Gaussian ${\cal N}(0;1)$, $f(x) = (K_1 - x)_+$ and $F$ is such that 
\begin{equation*}
	S_{T_2} = F(G, S_{T_1}) = S_{T_1} e^{(r-\frac{\sigma^2}{2}) (T_2 - T_1) + \sigma \sqrt{T_2-T_1} Z}.
\end{equation*}

Note that, in this experiments, the underlying process $(S_t)_{t \in [0,T_2]}$ is not discretized in time. The bias error is exclusively due to the inner Monte Carlo estimator of the conditional expectation.

The parameters used for the underlying process $(S_t)_{t \in [0,T_2]}$ are $S_0 = 100$, $r = 0.03$ and $\sigma = 0.3$. The parameters of the  {\em Put}-on-{\em Call}   payoff are $T_1 = 1/12$, $T_2 = 1/2$ and $K_1 = 6.5$, $K_2 = 100$. Section~\ref{sec:nested} strongly suggests that~\eqref{strong_error} and~\eqref{weak_error} are satisfied with $\beta = \alpha = 1$. A crude computation of other structural parameters yields $\var(Y_0) \simeq 9.09$, $V_1 \simeq 7.20$ and $\theta \simeq 0.89$.

The results are summarized in Table~\ref{tab:nested_MLRR} for \MLRR and in Table~\ref{tab:nested_MLMC} for \MLMC. 
\begin{table}[ht!]
\centering 
\begin{tabular}{c|c|c|c|c|c||c|c|c|c|c}
	$k$ & $\varepsilon=2^{-k}$ & $\L^2$--error & time $(s)$ & bias & variance 
	& $R$ & $M$ & $h^{-1}$ & $N$ & $\Cost$ \\
	\midrule
1 & 5.00\e{-01} & 4.36\e{-01} & 8.82\e{-04} & 3.17\e{-01} & 8.95\e{-02} & 2 & 5 & 1 & 6.53\e{+02} & 1.37\e{+03} \\
2 & 2.50\e{-01} & 2.70\e{-01} & 4.91\e{-03} & 2.14\e{-01} & 2.70\e{-02} & 2 & 9 & 1 & 2.51\e{+03} & 6.33\e{+03} \\
3 & 1.25\e{-01} & 1.18\e{-01} & 2.67\e{-02} & 8.42\e{-02} & 6.89\e{-03} & 3 & 3 & 1 & 1.75\e{+04} & 4.65\e{+04} \\
4 & 6.25\e{-02} & 5.94\e{-02} & 1.05\e{-01} & 3.79\e{-02} & 2.09\e{-03} & 3 & 4 & 1 & 6.27\e{+04} & 1.87\e{+05} \\
5 & 3.12\e{-02} & 3.36\e{-02} & 4.02\e{-01} & 2.31\e{-02} & 5.97\e{-04} & 3 & 5 & 1 & 2.41\e{+05} & 7.84\e{+05} \\
6 & 1.56\e{-02} & 1.89\e{-02} & 1.17\e{+00} & 1.38\e{-02} & 1.65\e{-04} & 3 & 6 & 1 & 9.52\e{+05} & 3.32\e{+06} \\
7 & 7.81\e{-03} & 1.20\e{-02} & 5.13\e{+00} & 1.00\e{-02} & 4.45\e{-05} & 3 & 7 & 1 & 3.80\e{+06} & 1.41\e{+07} \\
8 & 3.91\e{-03} & 6.37\e{-03} & 2.26\e{+01} & 5.30\e{-03} & 1.25\e{-05} & 3 & 9 & 1 & 1.54\e{+07} & 6.28\e{+07} \\
9 & 1.95\e{-03} & 2.48\e{-03} & 1.06\e{+02} & 1.89\e{-03} & 2.62\e{-06} & 4 & 4 & 1 & 8.22\e{+07} & 3.26\e{+08} \\
\end{tabular}
\caption{Nested compound option ($\alpha = 1$, $\beta = 1$): Parameters and results of the \MLRR estimator.}
\label{tab:nested_MLRR}
\end{table}
\begin{table}[ht!]
\centering 
\begin{tabular}{c|c|c|c|c|c||c|c|c|c|c}
	$k$ & $\varepsilon=2^{-k}$ & $\L^2$--error & time $(s)$ & bias & variance 
	& $R$ & $M$ & $h^{-1}$ & $N$ & $\Cost$ \\
	\midrule
1 & 5.00\e{-01} & 8.97\e{-01} & 5.54\e{-04} & 8.59\e{-01} & 6.62\e{-02} & 2 & 4 & 1 & 6.38\e{+02} & 1.14\e{+03} \\
2 & 2.50\e{-01} & 5.74\e{-01} & 4.25\e{-03} & 5.56\e{-01} & 2.05\e{-02} & 2 & 7 & 1 & 2.64\e{+03} & 5.76\e{+03} \\
3 & 1.25\e{-01} & 2.69\e{-01} & 2.37\e{-02} & 2.58\e{-01} & 6.08\e{-03} & 3 & 4 & 1 & 1.72\e{+04} & 4.57\e{+04} \\
4 & 6.25\e{-02} & 1.32\e{-01} & 1.13\e{-01} & 1.24\e{-01} & 1.95\e{-03} & 3 & 6 & 1 & 6.98\e{+04} & 2.26\e{+05} \\
5 & 3.12\e{-02} & 7.21\e{-02} & 4.99\e{-01} & 6.81\e{-02} & 5.69\e{-04} & 3 & 8 & 1 & 2.88\e{+05} & 1.06\e{+06} \\
6 & 1.56\e{-02} & 3.78\e{-02} & 1.57\e{+00} & 3.59\e{-02} & 1.40\e{-04} & 4 & 5 & 1 & 1.53\e{+06} & 6.21\e{+06} \\
7 & 7.81\e{-03} & 1.43\e{-02} & 8.70\e{+00} & 1.27\e{-02} & 4.28\e{-05} & 4 & 7 & 1 & 6.32\e{+06} & 3.02\e{+07} \\
8 & 3.91\e{-03} & 9.78\e{-03} & 3.63\e{+01} & 9.17\e{-03} & 1.15\e{-05} & 4 & 8 & 1 & 2.58\e{+07} & 1.31\e{+08} \\
9 & 1.95\e{-03} & 4.95\e{-03} & 1.68\e{+02} & 4.61\e{-03} & 3.21\e{-06} & 4 & 10 & 1 & 1.07\e{+08} & 6.06\e{+08} \\
\end{tabular}
\caption{Nested compound option ($\alpha = 1$, $\beta = 1$): Parameters and results of the \MLMC estimator.}
\label{tab:nested_MLMC}
\end{table}

Note on Figure~\ref{fig:nested} that \MLRR is faster than \MLMC as a function of the empirical RMSE by a factor approximately equal to $5$ within the range of our simulations.
\begin{figure}[ht!] 
	\begin{minipage}[b]{.49\linewidth}
        \centering \includegraphics[width=\linewidth]{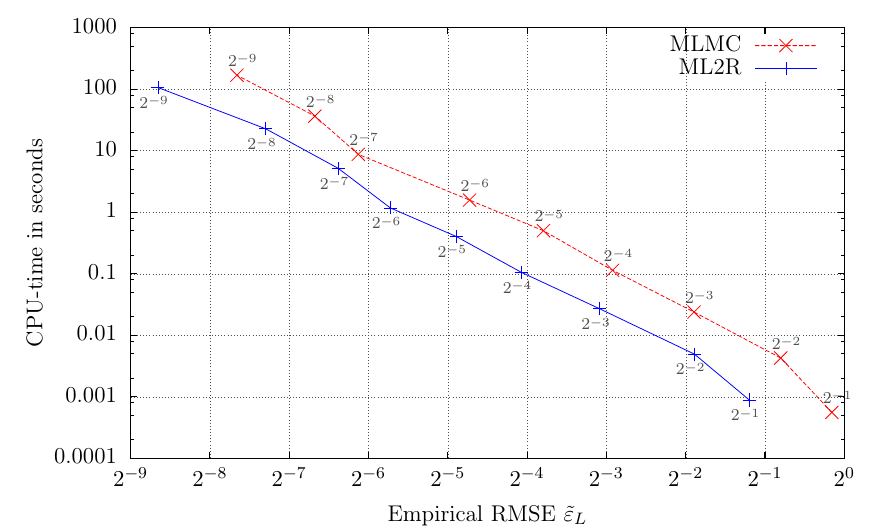}
        \subcaption{CPU--time ($y$--axis, $\log$ scale) as a function of $\tilde \varepsilon_L$ ($x$--axis, $\log_2$ scale).}
    \end{minipage} \hfill
	\begin{minipage}[b]{.49\linewidth}
		\centering \includegraphics[width=\linewidth]{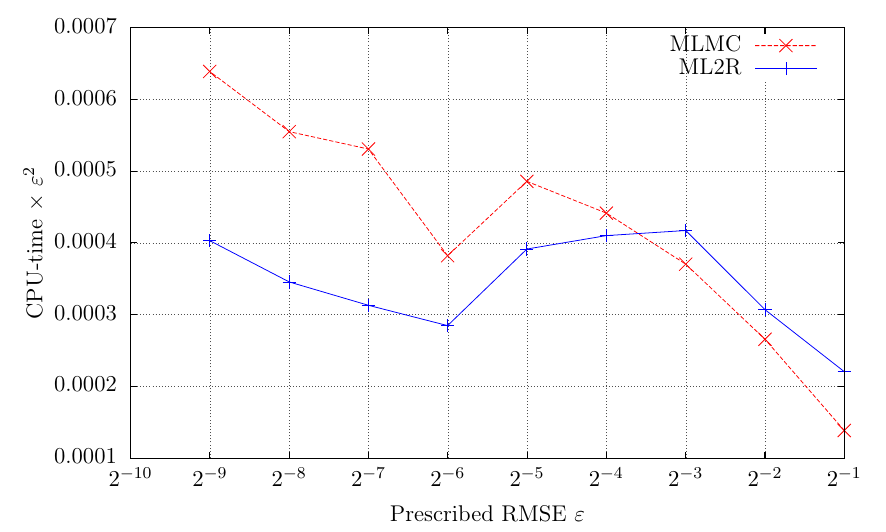}
        \subcaption{CPU--time $\times$ $\varepsilon^2$ ($y$--axis) as a function of $\varepsilon$ ($x$--axis, $\log_2$ scale).}
	\end{minipage}
    \caption{Nested compound option in a Black-Scholes model.}
    \label{fig:nested}
\end{figure}

\appendix
\section{Appendix}

\begin{Lemma}\label{lem:VdM} $(a)$
	The solution of the system $V \!\w = e_1$ where $V$ is a Vandermonde matrix 
	\begin{equation*}
		V = V(1,n_2^{-\alpha},\dots,n_{_R}^{-\alpha}) = \pa{\begin{array}{cccc}
		1 & 1 & \cdots & 1 \\
				1 & n_2^{-\alpha} & \cdots & n_{_R}^{-\alpha} \\
		\vdots & \vdots & \cdots & \vdots \\
		1 & n_2^{-\alpha (R-1)} & \cdots & n_{_R}^{-\alpha (R-1)} \\
		\end{array}},
	\end{equation*}
is given by $\ds \w_i = \frac{(-1)^{R-i} n_i^{\alpha(R-1)}}{\ds \prod_{1 \le j < i}(n^{\alpha} _i-n^{\alpha} _j)\prod_{i < j \le R}(n^{\alpha} _j-n^{\alpha} _i)}$.

\smallskip 
\noindent $(b)$ Furthermore 
\[
\widetilde \w_{_{R+1}}= \sum_{i=1}^R\frac{\w_i}{n_i^{\alpha R}}=\frac{(-1)^{R-1}}{\prod_{1\le i\le R} n^{\alpha}_i}.
\] 
\end{Lemma}

\begin{proof} $(a)$ Let $a_i = n_i^{-\alpha}$.	Note that by Cramer's rule the solution of this linear system is given by $\w_i = \frac{\det(V_i)}{\det(V)}$ where $V_i$ is the matrix formed by replacing the $i$--th column of $V$ by the column vector $e_1$.   The first point is that $V_i$ is again a Vandermonde matrix of type $V_i = V(1,\dots,a_{i-1},0,a_{i+1},\dots,a_R)$. On the other hand, the determinant of a square Vandermonde matrix can be expressed as $\det(V) = \prod_{1 \le j < k \le n} \pa{a_k - a_j}$. We have for every $i \in \ac{1,\dots,R}$  
\begin{equation*}
	\w_i = \frac{\ds \prod_{1 \le j < k \le R; j,k \neq i} (a_k - a_j) \prod_{1 \le j < i} (-a_j) \prod_{i<k\le R} a_k} {\ds \prod_{1 \le j < k \le R} \pa{a_k - a_j}}
	= \frac{\ds \prod_{1 \le j < i} (-a_j) \prod_{i<k\le R} a_k} 
	{\ds \prod_{1 \le j < i} (a_i-a_j) \prod_{i<k\le R} (a_k-a_i)} 
\end{equation*}
Using that $a_i = n_i^{-\alpha}$, $i=1,\ldots,R$, we get 
\begin{equation*}
\frac{\prod_{1 \le j < i} (-a_j)}{\prod_{1 \le j < i}(a_i - a_j)} = \frac{n_i^{\alpha(i-1)}}{\prod_{1 \le j < i} (n^{\alpha} _i - n^{\alpha} _j)}
\end{equation*}
and 
\begin{equation*}
\frac{\prod_{i <k \le R} a_k}{\prod_{i < k \le R}(a_k - a_i)} 
= \frac{(-1)^{R-i} n_i^{\alpha(R-i)}}{\prod_{i < k \le R} (n^{\alpha} _k - n^{\alpha} _i)}
\end{equation*}
which completes the proof.

\smallskip 
\noindent $(b)$  follows by setting $x=0$  in the decomposition
\[
\frac{1}{\prod_{1\le i\le R} (x-n^{\alpha} _i) } = \sum_{i=1}^R \frac{1}{(x-n^{\alpha} _i)\prod_{j\neq i} (n^{\alpha} _i-n^{\alpha} _j) }.
\]
\end{proof}

\begin{Proposition} \label{asympalpha:2} When $n_i= M^{i-1}$, $i=1,\ldots,R$, the following holds true for the coefficients 
$\w_i= \w_i(R,M)
$.

\begin{enumerate}
\item Closed form for $\w_i$, $i=1,\ldots,R$: 
\[
\w_i =\w_i(R,M)= (-1)^{R-i} \frac{M^{-\frac \alpha 2 (R-i)(R-i+1)}}{\prod_{1\le j\le i-1} (1-M^{-j \alpha})\prod_{1\le j\le R-i}(1-M^{-j \alpha})},\; i=1,\ldots,R.
\]
\item Closed form for $\widetilde \w_{_{R+1}}$: 
\[
	\widetilde \w_{_{R+1}} = (-1)^{R}M^{-\frac{R(R-1)}{2}\alpha}.
\]
\item A useful upper bound:
\[
\sup_{R \in \N^*}\sum_{i=1}^{R-1} |\w_i(R,M)
|\le \frac{M^{-\alpha}}{\pi_{_{\alpha,M}}^2}\sum_{k\ge 0}M^{-\alpha\frac{k(k+3)}{2}}\quad\mbox{ and }\quad 1\le \w_{_R}(R,M)
 \le \frac{1}{\pi_{\alpha,M}}
\]
where $\pi_{_{\alpha, M}}= \prod_{k\ge 1}(1-M^{-\alpha k})$.
\item Asymptotics of the coefficients $\w_i$  when $M\to +\infty$:
\[
 \lim_{M\to +\infty}\sup_{R\in \N^*} \max_{1\le i\le R-1}|\w_i(R,M)
|=0\; \mbox{ and }\; \lim_{M\to +\infty} \sup_{R\in \N^*} |\w_R(R,M)
-1|=0.
\]

\item Asymptotics of the coefficients $\W_i= \W_i(R,M)$  when $M\to +\infty$:
the  coefficients $\W_i$ are defined in~\eqref{tempMLRR}. It follows from what precedes that they  satisfy $\W_1=1$, 
\begin{equation}\label{eq:W(M)}
  \max_{1\le i\le R}|\W_i(R,M)|\le \W_{\alpha}(M) :=  \frac{M^{-\alpha}}{\pi_{_{\alpha,M}}^2}\sum_{k\ge 0}M^{-\alpha\frac{k(k+3)}{2}}+ \frac{1}{\pi_{\alpha,M}}
\end{equation}
and
\[
  \max_{1\le i\le R}|\W_i(R,M)-1|\le \W_{\alpha}(M)-1 \sim M^{-\alpha}\to 0 \; \mbox{ as } M\to +\infty.
\]
In particular, the matrix $\mT=\mT(R,M)$ in~\eqref{tempMLRR} converges toward the matrix of the standard Multilevel Monte Carlo~\eqref{tempMLMC} at level $M$ when $M\to+\infty$.
\item One more useful inequality
\[
\forall\, R\!\in \N,\quad \frac{1}{  |\widetilde \w_{_{R+1}}|}\sum_{r=1}^R\frac{|\w_r(R,M)|}{n_r^{\alpha R}  } \le B_{\alpha}(M)\frac{1}{\pi_{\alpha, M}^2} \sum_{k\ge 0}M^{-\frac{\alpha}{2}k(k+1)}.
\]
\end{enumerate}
\end{Proposition}

\begin{proof} Claim~6:
For every $r\!\in \{1,\ldots,R\}$, 
\[
\frac{|\w_r(R,M)
|}{n_r^{\alpha R}  } \le   \frac{M^{-\frac{\alpha}{2}((R-r)(R-r+1)+2(r-1)R)}}{\pi^2_{\alpha,M}}
\]
Noting that $((R-r)(R-r+1)+2(r-1)R)= R(R-1)+r(r-1)$, we derive that
\[
\sum_{r=1}^R\frac{|\w_r(R,M)
|}{n_r^{\alpha R}  } \le \frac{1}{\pi_{\alpha,M}^2} M^{-\alpha \frac{R(R-1)}{2}} \sum_{r=1}^R M^{-\alpha \frac{r(r-1)}{2}} 
\]
which yields the announced inequality since $M^{-\alpha \frac{R(R-1)}{2}} = |\widetilde \w_{_{R+1}}|$.
\end{proof}

\section{Appendix: sketch of proof of Propositions~\ref{pro:MC} and~\ref{pro:multiRR}}
The multistep Richardson-Romberg estimator with the formal framework of Section~\ref{sec:MLRR}, is characterized by the design matrix $\mT = \pa{\w, \mathbf{0}, \dots, \mathbf{0}}$.
Note that the first column is not $e_1$ but this has no influence on what follows. The expansion of $\esp[1]{\bar Y^N_{h,\underline n}}$ follows from Proposition~\ref{pro:multiRR}. No allocation is needed here since only one Brownian motion is involved. The proof of Proposition~\ref{thm:phaseII} applies here with $q = (1, 0, \dots, 0)$. Furthermore
$$
\upphi(\bar Y^N_{h,\underline n})= \var(\langle \w, Y^1_{h,\underline n}\rangle)\frac{|\underline n|}{h} \sim \var(Y_0) \frac{|\underline n|}{h}\text{ as }\; h\to 0
$$
since $Y^1_{h,\underline n}\to Y_0 \mathbf{1}$ in $\L^2$ and $\sum_{i=1}^{R} \w_i =1$. \hfill $\Box$

\bibliographystyle{alpha}
\bibliography{bibli}
\end{document}